\documentclass[10pt]{article}
\usepackage{amsthm,amsfonts,amsmath,amssymb,latexsym,epsfig,graphics,color}
\usepackage[matrix,arrow,curve]{xy}
\usepackage{makeidx}

\title{Symplectic homology, autonomous Hamiltonians, and Morse-Bott
  moduli spaces}

\author{
\begin{tabular}{cc}
Fr\'ed\'eric {\sc Bourgeois} & Alexandru {\sc
Oancea} \\
& \\
{\it Universit\'e Libre de Bruxelles} & {\it Universit\'e
          Louis Pasteur} \\
{\it B-1050 Bruxelles, Belgium} & {\it F-67084 Strasbourg, France} \\ 
& \\
{\tt fbourgeo@ulb.ac.be} & {\tt oancea@math.u-strasbg.fr} \\
& 
\end{tabular}
}

\date{11 March 2008}

\newtheorem{PARA}{}[section]
\newtheorem{theorem}[PARA]{Theorem}

\newtheorem{lemma}[PARA]{Lemma}
\newtheorem{proposition}[PARA]{Proposition}
\newtheorem{definition}[PARA]{Definition}

\theoremstyle{definition}
\newtheorem{remark}[PARA]{Remark}
\newtheorem{example}[PARA]{Example}
\newcommand{\para}{\begin{PARA}\rm}
\newcommand{\arap}{\end{PARA}\rm}
\newcommand{\dfn}{\begin{definition}\rm}
\newcommand{\nfd}{\end{definition}\rm}
\newcommand{\rmk}{\begin{remark}\rm}
\newcommand{\kmr}{\end{remark}\rm}
\newcommand{\xmpl}{\begin{example}\rm}
\newcommand{\lpmx}{\end{example}\rm}
\newcommand{\cA}{\mathcal{A}}
\newcommand{\cB}{\mathcal{B}}

\newcommand{\cD}{\mathcal{D}}
\newcommand{\cF}{\mathcal{F}}
\newcommand{\cE}{\mathcal{E}}

\newcommand{\cH}{\mathcal{H}}

\newcommand{\cJ}{\mathcal{J}}
\newcommand{\cK}{\mathcal{K}}
\newcommand{\cL}{\mathcal{L}}
\newcommand{\cM}{\mathcal{M}}

\newcommand{\cO}{\mathcal{O}}
\newcommand{\cP}{\mathcal{P}}

\newcommand{\cS}{\mathcal{S}}
\newcommand{\cT}{\mathcal{T}}

\newcommand{\oeps}{{\overline{\epsilon}}}
\newcommand{\okappa}{{\overline{\kappa}}}
\newcommand{\ukappa}{{\underline{\kappa}}}
\newcommand{\og}{{\overline{\gamma}}}
\newcommand{\ug}{{\underline{\gamma}}}
\newcommand{\oev}{\overline{\mathrm{ev}}}
\newcommand{\uev}{\underline{\mathrm{ev}}}
\newcommand{\oD}{{\overline{D}}}
\newcommand{\uD}{{\underline{D}}}
\newcommand{\oS}{{\overline{S}}}
\newcommand{\uS}{{\underline{S}}}
\newcommand{\oQ}{{\overline{Q}}}
\newcommand{\uQ}{{\underline{Q}}}
\newcommand{\ou}{{\overline{u}}}
\newcommand{\uu}{{\underline{u}}}
\newcommand{\oU}{\overline{U}}
\newcommand{\uU}{\underline{U}}
\newcommand{\oV}{\overline{V}}
\newcommand{\uV}{\underline{V}}
\newcommand{\opsi}{\overline{\Psi}}
\newcommand{\upsi}{\underline{\Psi}}
\newcommand{\ozeta}{\overline{\zeta}}
\newcommand{\uzeta}{\underline{\zeta}}
\newcommand{\tzeta}{{\widetilde{\zeta}}}
\newcommand{\txi}{{\widetilde{\xi}}}
\newcommand{\otheta}{\overline{\theta}}
\newcommand{\utheta}{\underline{\theta}}
\newcommand{\one}
{{{\mathchoice \mathrm{ 1\mskip-4mu l} \mathrm{ 1\mskip-4mu l}
\mathrm{ 1\mskip-4.5mu l} \mathrm{ 1\mskip-5mu l}}}}

\newcommand{\C}{{\mathbb{C}}}

\newcommand{\R}{{\mathbb{R}}}

\renewcommand{\u}{{\mathbf{u}}}
\renewcommand{\v}{{\mathbf{v}}}

\newcommand{\x}{{\mathbf{x}}}
\newcommand{\Z}{{\mathbb{Z}}}
\newcommand{\coker}{\mathrm{ coker }}  
\newcommand{\im}{\mathrm{ im }}        
\newcommand{\Det}{{\mathrm{Det}}}      

\newcommand{\Max}{\mathit{ Max}}
\renewcommand{\min}{\mathit{ min}}
\newcommand{\ind}{\mathrm{ind}}
\newcommand{\Jreg}{\cJ_{\mathrm{reg}}}   

\newcommand{\Freg}{\cF_{\mathrm{reg}}}

\newcommand{\eps}{{\varepsilon}}
\newcommand{\om}{{\omega}}

\newcommand{\tcO}{{\widetilde{\cO}}}
\newcommand{\ttcO}{{\stackrel{\raisebox{-1mm}{$\approx$}}{\raisebox{.01mm}{$\cO$}}}}
\newcommand{\ttcOs}{{\stackrel{\raisebox{-1mm}{$\approx$}}{\raisebox{.01mm}{$\cO$}}{}\hspace{-1mm}^s}}
\newcommand{\ttcOu}{{\stackrel{\raisebox{-1mm}{$\approx$}}{\raisebox{.01mm}{$\cO$}}{}\hspace{-1mm}^u}}
\newcommand{\tp}{{\widetilde{p}}}
\newcommand{\tq}{{\widetilde{q}}}

\newcommand{\tcS}{{\widetilde{\cS}}}
\newcommand{\tcSg}{{\widetilde{\cS}_{\textrm{good}}}}
\newcommand{\tcSb}{{\widetilde{\cS}_{\textrm{bad}}}}
\newcommand{\tw}{{\widetilde{w}}}
\newcommand{\tx}{{\widetilde{x}}} 

\def\NABLA#1{{\mathop{\nabla\kern-.5ex\lower1ex\hbox{$#1$}}}}
\def\Nabla#1{\nabla\kern-.5ex{}_{#1}}
\def\Tabla#1{\Tilde\nabla\kern-.5ex{}_{#1}}
\renewcommand{\Tilde}{\widetilde}

\newcommand{\p}{{\partial}}
\newcommand{\dbar}{{\bar\partial}}
\newcommand{\tpar}{{\parallel\mskip-7mu\vert}}

\newenvironment{enum}
{\begin{enumerate}}
{\end{enumerate}}

\makeindex

\begin{document}

\maketitle
\pagestyle{myheadings}
\markright{Symplectic homology for autonomous Hamiltonians}


\begin{abstract}
We define Floer homology for a time-independent, or autonomous
Hamiltonian on a symplectic manifold with contact type boundary, under
the assumption that its $1$-periodic orbits are transversally
nondegenerate. Our construction is based on Morse-Bott techniques for
Floer trajectories. Our main motivation is to understand the
relationship  between linearized contact homology of a fillable
contact manifold and symplectic homology of its filling.
\end{abstract}


\tableofcontents

\vfill \pagebreak


\section{Introduction} 

One crucial hypothesis in the definition of
Floer homology~\cite{F} of a Hamiltonian $H$ on a symplectic
manifold $(W,\om)$ is that the $1$-periodic orbits of the Hamiltonian
vector field $X_H$ are nondegenerate. Unless they are all constant --
which happens if the Hamiltonian is $C^2$-small -- this forces $H$ to be
time-dependent. The purpose of this paper is to define Floer homology
for a time-independent, or autonomous Hamiltonian $H:W\to \R$ under
the assumption that its $1$-periodic orbits are transversally
nondegenerate. This last condition is generic in the space of
autonomous Hamiltonians.   

Although this generalization of Floer homology is interesting by itself,
our main motivation
is to understand the relationship between linearized contact
homology of a fillable contact manifold $(M,\xi)$ and 
symplectic homology of the filling $(W,\om)$. In this case
there is a natural class of time-independent Hamiltonians on $W$ whose
nonconstant $1$-periodic orbits correspond precisely to closed Reeb
orbits on $M=\p W$, and for which the Floer trajectories can be
related to holomorphic cylinders in the
symplectization $M\times \R$~\cite{BOcontact}. The goal of the present
paper is to relate the Floer trajectories of a specific time-dependent
perturbation to the Floer trajectories of the unperturbed
Hamiltonian. Thus Floer homology for time-independent Hamiltonians
serves as a bridge between symplectic homology and linearized contact
homology. Moreover, the moduli spaces of Floer trajectories for
autonomous Hamiltonians are related to the moduli spaces 
defining $S^1$-equivariant symplectic homology~\cite{BO,BOcontact}.   

The Morse-Bott analysis in this
paper is, to the best of our knowledge, new to the literature, being
based on ideas contained in the first author's
Ph.D. dissertation~\cite{B} within the context of contact
homology. Although our situation is that of critical manifolds of
dimension one, the complexity of the analytical setup is the same as
that of the higher dimensional case. 

We must mention
at this point Frauenfelder's inspired approach~\cite[Appendix~A]{Fr}
in which he defines a complex for a Morse-Bott function on a finite
dimensional manifold via ``flow lines with
cascades'' -- these being our Floer trajectories with gradient
fragments -- and in which, without proving the correspondence with gradient
trajectories for some perturbed Morse function, he directly shows
deformation invariance of the resulting chain complex. 

\medskip 

We now describe the structure of
the paper. We give in the introduction only a loose statement of our main
Correspondence Theorem~\ref{thm:degen} and we recall in
Section~\ref{sec:SH} the construction of symplectic homology. Although
this is well-known to specialists we still need to establish
notations, and we seize the occasion to set up a general framework
using the Novikov 
ring and nontrivial homotopy classes of periodic
orbits. 

Section~\ref{sec:MBcomplex} describes
the Morse-Bott complex 
and formally states the Correspondence Theorem~\ref{thm:degen}. The
latter is complemented by Proposition~\ref{prop:signs} which describes
how the coherent orientation signs for the Morse-Bott complex are
related to the ones for the Floer complex. 

Section~\ref{sec:FredMB} contains the proofs of the  
previous transversality, compactness, gluing and orientation
statements. Finally, the Appendix contains the statements
concerned with the asymptotic behaviour of the various types of Floer
trajectories that we use. These asymptotic estimates enter crucially
in the proof of the compactness statements, as well as in the
definition of the Fredholm setup for gluing.

\medskip 

We end the introduction with an informal presentation of our
results. Let $H:W\to \R$ be 
an autonomous Hamiltonian defined on a symplectic manifold
$(W,\om)$. We assume that $H$ is a Morse function and that the 
nonconstant $1$-periodic orbits of $H$ are transversally
nondegenerate. The set $\cP(H)$ of $1$-periodic orbits of $H$ is the
set of critical points of the Hamiltonian action functional and
consists of isolated elements $\gamma_{\tp}$ corresponding to critical
points $\tp\in\mathrm{Crit}(H)$, and of nonisolated elements coming in
families $S_\gamma$ which are Morse-Bott nondegenerate circles. These
correspond to reparametrizations of some given orbit
$\gamma\in\cP(H)$, with $\gamma:S^1=\R/\Z\to W$. 

For each circle $S_\gamma$ we choose a perfect Morse function
$f_\gamma:S_\gamma\to \R$ with exactly one maximum $\Max$ and one minimum
$\min$. We denote by $\gamma_\min$, $\gamma_\Max$ the orbits in $S_\gamma$
corresponding to the minimum and the maximum of $f_\gamma$ respectively. 
We choose a chart $S^1\times \R^{2n-1}\ni(\tau,p)$ and a smooth
cut-off function $\rho_\gamma:S^1\times \R^{2n-1}\to \R$ in the
neighbourhood of each $\gamma(S^1)\subset W$, and we denote by 
$\ell_\gamma\in \Z^+$ the maximal positive integer such that
$\gamma(\theta+ 1/ {\ell_\gamma})=\gamma(\theta)$, $\theta\in
S^1$. 

Following~\cite{CFHW}, for $\delta>0$ small enough the time-dependent
Hamiltonian 
$$
H_\delta:S^1\times W\to \R,
$$
$$
H_\delta(\theta,\tau,p) := H - \delta \sum_{S_\gamma}
\rho_\gamma(\tau,p) f_\gamma(\tau - \ell_\gamma\theta)
$$
has only nondegenerate $1$-periodic orbits. Moreover, these are of 
the following two types: they are either constant orbits $\gamma_\tp$
corresponding to critical points $\tp\in\mathrm{Crit}(H)$, or they are
nonconstant orbits of the form $\gamma_p\in\cP(H)$ for 
$p\in\textrm{Crit}(f_\gamma)$. Thus, out of each circle $S_\gamma$ of
periodic orbits for $H$ there are exactly two orbits surviving for
$H_\delta$, namely $\gamma_\min$ and $\gamma_\Max$.

Let $J$ be a generic time-dependent almost complex structure on
$W$. Given $p\in \mathrm{Crit}(f_\og)$, $q\in\mathrm{Crit}(f_\ug)$ we
denote by 
$$
\cM(\og_p,\ug_q;H_\delta,J)
$$
the moduli space of Floer
trajectories for the pair $(H_\delta,J)$ modulo reparametrization,
with negative asymptote $\og_p$ and positive asymptote $\ug_q$. 
We also denote by 
$$
\cM(S_\og,S_\ug;H,J)
$$
the moduli space
of Floer trajectories for the pair $(H,J)$ modulo reparametrization,
with negative asymptote in $S_\og$ and positive asymptote in $S_\ug$. 
Our goal is to describe the moduli spaces of the first type in terms of
moduli spaces of the second type. 

We denote by 
$$
\cM(p,q;H, \{f_\gamma\},J)
$$
the moduli space of Floer trajectories for the pair $(H,J)$ 
with intermediate gradient fragments, consisting of tuples 
$$
[\u]=(c_m,[u_m],c_{m-1},[u_{m-1}],\ldots,[u_1],c_0),  
\quad m\ge 0
$$
such that: 
\begin{enum}
\item[(i)] $[u_i]\in
\cM(S_{\gamma_i},S_{\gamma_{i-1}};H,J)$, $i=1,\ldots,m$ with
$\gamma_m:=\og$, $\gamma_0:=\ug$; 
\item[(ii)] $c_m$ is
a semi-infinite gradient trajectory of $f_\og=f_{\gamma_m}$ 
connecting $\og_p$ to the endpoint of $u_m$; 
\item[(iii)] $c_j$, $j=1,\ldots,m-1$ is a finite gradient trajectory of
$f_{\gamma_j}$ connecting the endpoints of $u_{j+1}$ and $u_j$; 
\item[(iv)] $c_0$ is a semi-infinite gradient trajectory of
  $f_\ug=f_{\gamma_0}$ connecting the endpoint of $u_1$ to
  $\ug_q$. 
\end{enum} 
We give a pictogram of such an element $[\u]$ with $m\ge 1$ in
  Figure~\ref{fig:wtilde} on 
page~\pageref{fig:wtilde}, where one should read $c_i$ instead of
$v_i$. If $m=0$ such an element $[\u]$ is simply an infinite gradient
  trajectory of some $f_\gamma$. 
Let us note that, just as the space of Floer trajectories for a
nondegenerate Hamiltonian can be compactified by adding ``broken''
Floer trajectories, the space of Floer trajectories with intermediate gradient
fragments can be compactified
  by adding ``broken'' such objects, with
an obvious meaning. We denote by $\overline \cM(p,q;H,\{f_\gamma\},J)$
  these compactified moduli spaces. 

Our main result is the following comparison theorem. 

\medskip 

\noindent {\bf Theorem.} {\it The 
following assertions hold. 
\begin{enum}
\item any sequence $[v_n]\in\cM(\og_p,\ug_q;H_{\delta_n},J)$,
$\delta_n\to 0$ converges to an element of 
$\overline \cM(p,q;H,\{f_\gamma\},J)$;
\item any element of $\overline \cM(p,q;H,\{f_\gamma\},J)$ can be
  obtained as such a limit; 
\item there is a bijective correspondence between elements of
$\cM(\og_p,\ug_q;H_\delta,J)$ and elements of
$\cM(p,q;H,\{f_\gamma\},J)$ if the difference of index of the
endpoints is equal to one, or equivalently if the moduli spaces have
dimension zero.
\end{enum}
}

The rigorous forms for the statements (i), (ii), (iii) are given in
Proposition~\ref{prop:compact}, Proposition~\ref{prop:family} and
Theorem~\ref{thm:degen} respectively. 
Unsurprisingly, the Fredholm setup for the previous theorem uses
Sobolev norms with exponential weights since we have degenerate
asymptotics. Similarly, due to the convergence estimates in the
Appendix, there are such weights centered on the portions of the
Floer cylinders approaching gradient fragments. For each peak in the
weight, there is a special section supported around this peak which
has constant norm with respect to $\delta\to 0$. For each gradient
fragment this section corresponds to the reparametrization shift 
of the underlying gradient trajectory. As $\delta\to 0$,
the corresponding peak explodes and thus forbids all infinitesimal
variations except for the single degree of freedom coming from Morse
theory.

To be useful for homological calculations the above theorem needs
to be complemented by a statement concerning signs. We describe in
Section~\ref{sec:ori} how to construct coherent orientations on the
relevant spaces of Fredholm operators and how
to obtain signs $\epsilon(\u)$ and $\epsilon(u_\delta)$ for elements
$[\u]\in\cM(p,q;H,\{f_\gamma\},J)$ and
$u_\delta\in\cM(\og_p,\ug_q;H_\delta,J)$ when the corresponding moduli
spaces are zero-dimensional. 
We recall in Remark~\ref{rmk:good and bad} the definition of {\it good
  orbits} borrowed from Symplectic Field Theory, where it plays a
crucial role in all orientation and signs problems. In the following
statement we denote again by $m\ge 0$ the number
of nonconstant Floer trajectories involved in $\u$. 

\medskip 

\noindent {\bf Proposition~\ref{prop:signs}.} {\it 
Assume the moduli spaces under consideration have dimension zero. The
bijective correspondence between elements
$u_\delta\in\cM(\og_p,\ug_q;H_\delta,J)$ and 
$[\u]\in\cM(p,q;H,\{f_\gamma\},J)$ 
changes signs as follows: 
\begin{enum} 
\item If $m\ge 1$ we have 
$$
\epsilon(\u)=(-1)^{m-1}\epsilon(u_\delta);
$$
\item If $m=0$ we have $\u=u_\delta$ and
$\epsilon(\u)=\epsilon(u_\delta)$, $p$ is the minimum and $q$ is the 
maximum of the same function $f_\gamma$, the moduli space 
$\cM(p,q;H,\{f_\gamma\},J)$ consists of 
the two gradient lines of $f_\gamma$ running from $p$ to $q$, and
their signs are different if and only if the underlying orbit $\gamma$
is good.  
\end{enum} 
}

This result has two pleasant consequences. On the one hand we can
construct a ``Morse-Bott'' chain complex which computes symplectic
homology by counting with suitable signs rigid elements in the
moduli spaces $\cM(p,q;H,\{f_\gamma\},J)$. On the other hand, this
chain complex singles out algebraically the good orbits and can be
used to relate the symplectic homology of a manifold $(W,\om)$
with contact type boundary to the 
linearized contact homology of its boundary -- the latter being
defined by a chain complex involving only good orbits. As already mentioned
at the beginning of this section, this is achieved
in~\cite{BOcontact}. 

\medskip 

{\it Acknowledgements.} A.O. has benefited from a Swiss National Fund
grant under the supervision of Prof. Dietmar Salamon at ETH Z\"urich. 
Both authors acknowledge financial support from the Fonds National de
la Recherche Scientifique, Belgium, the Forschungsinstitut
f\"ur Mathematik, Z\"urich, the Institut
de Recherche Math\'ematique Avanc\'ee at Universit\'e Louis Pasteur,
Strasbourg, as well as from the Mathematisches Forschungsinstitut,
Oberwolfach. The authors would like to thank an anonymous referee for
having carefully read through the arguments in the paper and for
having kindly suggested a solution to a gap in the original proof of
Proposition~\ref{prop:Surjectivity_udelta}.


\section{Symplectic homology} \label{sec:SH}

We define in this section the symplectic homology groups
of a symplectically aspherical manifold with contact type boundary.
Our construction is modelled on those of Cieliebak, Floer, Hofer and
Viterbo~\cite{C,CFH,FH,V}. We consider nontrivial homotopy classes of
loops and we use the Novikov ring.

Let $(W,\om)$ be a compact symplectic manifold with contact type
boundary $M:=\p W$. This means that there exists a vector field 
$X$\index{$X$, Liouville vector field}
defined in a neighbourhood of $M$, transverse and pointing outwards
along $M$, and such that
$$
\cL _X \om = \om.
$$
Such an $X$ is called a {\bf Liouville vector field}. The $1$-form
$\lambda:=(\iota_X\om)|_M$ is a contact form on $M$. We denote by
$\xi$\index{$\xi$, contact distribution} the contact distribution
defined by $\lambda$. The {\bf Reeb 
   vector field} $R_\lambda$\index{$R_\lambda$, Reeb vector field} 
is uniquely defined by the
conditions $\ker \, \om|_M = \langle R_\lambda \rangle$ and
$\lambda(R_\lambda)=1$. We denote by $\phi_\lambda$ the flow of
$R_\lambda$. The {\bf action spectrum} of $(M,\lambda)$ is
defined by
$$
\textrm{Spec}(M,\lambda) := \{ T \in \R^+\, | \, \textrm{ there is a
   closed } R_\lambda\textrm{-orbit of period } T\}.
 \index{$\textrm{Spec}(M,\lambda)$}
$$
We assume throughout this paper the condition
\begin{equation} \label{eq:asph}
\int_{T^2} f^*\om =0 \quad \mbox{for all smooth } f:T^2\to W.
\end{equation}  
This guarantees that the energy of a Floer 
trajectory does not depend 
on its homology class, but only on its 
endpoints (see below). 
Condition~\eqref{eq:asph} 
plays an important 
role in the Morse-Bott description of the 
symplectic homology 
groups. Our main class of examples is provided 
by exact symplectic 
forms.

Let $\phi$ be the flow of $X$. We parametrize a neighbourhood $U$ of $M$ by
$$
G: M \times [-\delta, 0] \to U, \qquad (p,t) \mapsto \phi^t(p).
$$
Then $d(e^t\lambda)$ is a symplectic form on $M\times \R^+$ and
$G$ satisfies $G^*\om = d(e^t \lambda)$.
We denote
$$
\widehat W : = W \ \bigcup _{G} \ M\times \R^+
\index{$\widehat W$, symplectic completion}
$$
and endow it with the symplectic form
$$
\widehat \om : =
\left\{\begin{array}{ll}
\om, & \textrm{ on } W, \\
d(e^t \lambda), & \textrm{ on } M\times \R^+.
\end{array} \right.
\index{$\widehat \om$, symplectic form on the completion $\widehat W$}
$$

Given a time-dependent Hamiltonian $H :S^1\times \widehat W \to \R$,
we define the {\bf Hamiltonian vector field} 
$X^\theta_H$\index{$X_H$, Hamiltonian vector field} 
by
$$
\widehat \om (X^\theta_H,\cdot) = d H_\theta, \qquad \theta\in S^1 = \R/\Z,
$$
where $H_\theta:=H(\theta,\cdot)$. We denote by $\phi_H$ the flow of
$X_H^\theta$, defined by $\phi_H^0=\textrm{Id}$ and
$$
  \frac d {d\theta} \phi_H^\theta (x) = X^\theta_H(\phi_H^\theta(x)), 
\qquad \theta\in \R.
$$

Let $\cH$ be the set of {\bf admissible Hamiltonians}\index{$\cH$,
  admissible Hamiltonians}, 
consisting of functions $H:S^1\times \widehat W \to \R$ which satisfy
\renewcommand{\theenumi}{\roman{enumi}}
\begin{enum}
  \item $H < 0$ on $W$;
  \item $H(\theta,p,t) = \alpha e^t + \beta$ for $t$ large enough, with
    $\alpha\notin \textrm{Spec}(M,\lambda)$;
  \item every $1$-periodic orbit $\gamma:S^1\to \widehat W$ of
    $X^\theta_H$ is nondegenerate, i.e.
   $$
    \det \left(\one - d\phi_H^1(\gamma(0)) \right) \neq 0.
   $$
\end{enum}
We denote by $\cP(H)$ the set of $1$-periodic orbits of $X^\theta_H$
and by $\cP^a(H)$ the set of $1$-periodic orbits in a given free
homotopy class $a$ in $\widehat W$.\index{$\cP(H), \cP^a(H)$} 

Let $\cJ$ denote the set of {\bf admissible almost complex
structures}\index{$\cJ$, admissible a.c. structures}
$$
J:S^1 \to \textrm{End}(T\widehat W), \qquad J^2=-\one
$$
which are compatible with $\widehat \om$ and have the following
standard form for $t$ large enough:
\begin{equation} \label{eq:J}
\left\{ \begin{array}{rcl}
J_{(p,t)} |_\xi & = & J_0, \\[.2cm]
J_{(p,t)} \frac \p  {\p t} & = & R_\lambda.
\end{array}\right.
\end{equation}
Here $J_0$ is any compatible complex structure on the symplectic
bundle $(\xi,d\lambda)$ which is independent of $\theta$ and $t$.

Let us fix a reference loop $l_a:S^1\to \widehat W$ for each free
homotopy class $a$ in $\widehat W$ such that $[l_a]=a$. If $a$ is the
trivial homotopy class we choose $l_a$ to be a constant loop.
Recall that free homotopy classes of loops in $\widehat W$ are in
one-to-one correspondence with conjugacy classes in $\pi_1(\widehat
W)$. As a consequence, the inverse $a^{-1}$ of a free homotopy class
is well-defined. We require that $l_{a^{-1}}$ coincides with the loop
$l_a$ with the opposite orientation.

The {\bf Hamiltonian action functional} acts on pairs $(\gamma,[\sigma])$
consisting of a loop $\gamma\in C^\infty(S^1,\widehat W)$ and the
homology class (rel boundary) of a map $\sigma:\Sigma \to \widehat W$
defined on a Riemann surface $\Sigma$ with two boundary components
$\p_0\Sigma$ (with the opposite boundary orientation) and
$\p_1\Sigma$ (with the boundary orientation), which satisfies
\begin{equation} \label{eq:sigma}
\sigma|_{\p_0\Sigma} = l_{[\gamma]}, \qquad \sigma|_{\p_1\Sigma}=\gamma.
\end{equation}
Its values are defined by
\begin{equation} \label{eq:AH}
\cA_H(\gamma,[\sigma]) :=
-\int_{\Sigma} \sigma^*\widehat \om - \int _{S^1}
H(\theta,\gamma(\theta)) \, d\theta.
\index{$\cA_H$, action functional}
\end{equation}
The differential $d\cA_H(\gamma,[\sigma]):C^\infty(S^1,\gamma^*T\widehat
W)\to \R$ is given by
$$
d\cA_H(\gamma,[\sigma])\zeta := \int_{S^1} \widehat \om (\dot \gamma -
X_H^\theta (\gamma),\zeta)\, d\theta.
$$
Therefore the critical points of $\cA_H$ are pairs $(\gamma,[\sigma])$
such that $\gamma\in \cP(H)$. We fix from now on, for each $\gamma\in
\cP(H)$, a map $\sigma_\gamma$ satisfying~(\ref{eq:sigma}); then the set of all
pairs $(\gamma,[\sigma])$ can be identified with $H_2(W;\Z)$ for 
fixed $\gamma$.

Let us choose a symplectic trivialization
$$
\Phi_a:S^1\times \R^{2n} \to l_a^*T\widehat W
$$
for each free
homotopy class $a$ in $\widehat W$. If $a$ is the trivial homotopy
class we choose the trivialization to be constant. Moreover, we
require that $\Phi_{a^{-1}}(\theta,\cdot)=\Phi_a(-\theta,\cdot)$,
$\theta\in S^1=\R/\Z$. For
each $\gamma\in
\cP(H)$ there exists a unique (up to homotopy) trivialization
$$
\Phi_\gamma: \Sigma \times \R^{2n} \to \sigma_\gamma^*T\widehat W
$$
such that $\Phi_\gamma=\Phi_{[\gamma]}$ on $\p_0\Sigma \times \R^{2n}$. Let
\begin{equation} \label{eq:triv}
\Psi:[0,1]\to \textrm{Sp}(2n), \qquad
\Psi(\theta):= \Phi_\gamma^{-1} \circ d\phi_H^\theta(\gamma(0)) \circ
\Phi_\gamma. 
\end{equation}
Because $\gamma$ is nondegenerate we can define the
{\bf Conley-Zehnder index} 
$\mu(\gamma)$\index{$\mu(\gamma)$, index of Reeb orbit} by
\begin{equation}
   \label{eq:mu}
   \mu(\gamma):=\mu(\gamma,\sigma_\gamma) := -\mu_{CZ}(\Psi),
  \index{$\mu(\gamma)$, index of Reeb orbit}   
\end{equation}
where $\mu_{CZ}(\Psi)$ is the Conley-Zehnder index of a path of
symplectic matrices~\cite{RS}.

\begin{remark} \rm If, in the previous construction, we replace
   $\sigma_\gamma$ with $\sigma_\gamma \# A$ for some $A\in H_2(W;\Z)$, then the
   resulting index will be
  \begin{equation} \label{eq:c1}
   \mu(\gamma,\sigma_\gamma \# A) = \mu(\gamma,\sigma_\gamma) -
   2\langle c_1(TW),A\rangle.
  \end{equation}
\end{remark}

We define the {\bf Novikov ring} 
$\Lambda_\om$\index{$\Lambda_\om$, Novikov ring} 
as the set of formal linear
combinations $\lambda:= \sum_{A\in H_2(W;\Z)} \lambda_A e^A$,
$\lambda_A \in \Z$ such that
$$
\#\left\{ A \, | \, \lambda_A \neq
   0, \ \om(A) \le c \right\}  < \infty
$$
for all $c>0$. The multiplication in $\Lambda_\om$ is given by
$$
\lambda * \lambda' := \sum_{A,B\in H_2(W;Z)} \lambda _A \lambda'_B
e^{A+B}.
$$
We note that, if $\om$ is exact, then $\Lambda_\om = \Z[H_2(W;\Z)]$.
We define a grading on $\Lambda_\om$ by $|e^A| := -2\langle c_1(TW),A\rangle$.
For each free homotopy class $a$ in $\widehat W$ and each admissible
Hamiltonian $H$ we define the {\bf symplectic chain group} 
$SC_*^a(H)$\index{$SC_*^a(H)$, symplectic chain group} 
as the free $\Lambda_\om$-module generated by elements
$\gamma\in \cP^a(H)$. The grading is given by
$$
|e^A\gamma| := \mu(\gamma) - 2\langle c_1(TW),A\rangle.
$$

We define the space of Floer trajectories
$\widehat \cM^A(\og,\ug;H,J)$
\index{$\cM^A(\og,\ug;H,J),\widehat\cM^A(\og,\ug;H,J)$|(} 
as the set of solutions $u:\R\times S^1\to \widehat W$ of the equation
\begin{equation}
   \label{eq:floer}
   \p_s u + J_\theta (\p_\theta u - X_H^\theta) =0,
\end{equation}
such that
\begin{equation}
   \label{eq:lim}
    \lim_{s\to -\infty} u(s,\theta) = \og(\theta), \qquad \lim_{s\to \infty}
    u(s,\theta) = \ug(\theta), \qquad \lim_{s\to \pm\infty} \p_s u =0
\end{equation}
uniformly in $\theta$ and
\begin{equation}
   \label{eq:A}
   [\sigma_{\og} \# u]  = [\sigma_{\ug} \# A ].
\end{equation}

\begin{remark} \rm Under the nondegeneracy assumption on $\og$, $\ug$
   condition~(\ref{eq:lim}) is equivalent to the finiteness of the
   energy
   \begin{equation}
     \label{eq:E} \index{$\cE(u)$, energy}
     \cE(u) := \cE_{J,H}(u) := 
    \frac 1 2 \int_{\R\times S^1} 
\left( |\p_su|^2_\theta + |\p_t u -
     X_H^\theta|^2_\theta \right) \, dsd\theta.
   \end{equation}
\end{remark}

Because  $\og$, $\ug$ are nondegenerate the
linearized operator $D_u: W^{1,p}(\R \times S^1,u^*T\widehat W) \to
L^p(\R \times S^1,u^*T\widehat W)$, $p>2$ given by
\begin{equation}
   \label{eq:D} \index{$D_u$, linearized operator}
   D_u \zeta := \nabla _s \zeta + J_\theta \nabla_\theta \zeta +
   (\nabla_\zeta J_\theta) \p_\theta u - \nabla_\zeta \left( J_\theta
     X_H^\theta \right), \ u\in \widehat \cM^A(\og,\ug;H,J)
\end{equation}
is Fredholm with index
\begin{equation} \label{eq:ind}
\ind (D_u) = \mu(\og) - \mu(\ug) + 2\langle c_1(TW),A\rangle.
\end{equation}
An almost complex structure $J\in \cJ$ is called {\bf regular for} 
$u\in \widehat \cM^A(\og,\ug;H,J)$ if $D_u$ is surjective, and it is
called {\bf regular} if $D_u$ is surjective for all $\og,\ug\in
\cP(H)$, $A\in H_2(W;\Z)$ and $u\in \widehat \cM^A(\og,\ug;H,J)$. It
is proved in~\cite{FHS} that the space 
$\Jreg(H)$\index{$\cJ_{\textrm{reg}}(H)$}
of regular almost
complex structures is of the second category in $\cJ$. For every $J\in
\Jreg(H)$ the space $\widehat \cM^A(\og,\ug;H,J)$ is a smooth manifold
of dimension $\mu(\og) - \mu(\ug) + 2\langle c_1(TW),A\rangle$. From
now on we fix some $J\in \Jreg(H)$. 

If $\og\neq\ug$ or $A\neq 0$, the additive group $\R$ acts freely on
$\widehat \cM^A(\og,\ug;H,J)$ by $s_0\cdot u(\cdot,\cdot) :=
u(s_0+\cdot,\cdot)$. We define the {\bf moduli space of Floer
trajectories} by 
$$
\cM^A(\og,\ug;H,J):= \widehat \cM^A(\og,\ug;H,J)/\R.
$$
\index{$\cM^A(\og,\ug;H,J),\widehat\cM^A(\og,\ug;H,J)$|)} 
Its dimension is
$$
\dim \, \cM^A(\og,\ug;H,J) := \mu(\og) - \mu(\ug) + 2\langle
c_1(TW),A\rangle-1.
$$
If $\og=\ug$ and $A=0$, the space $\widehat \cM^0(\og,\og;H,J)$ consists
of a single point, correponding to a constant solution (i.e. 
independent of $s$).
The $\R$ action is then trivial and we define the moduli space by
$\cM^0(\og,\og;H,J):= \widehat \cM^0(\og,\og;H,J)
\index{$\cM^0(\og,\ug;H,J),\widehat\cM^0(\og,\ug;H,J)$}$.
A straightforward application of the maximum principle~\cite{V} using 
the special form of admissible Hamiltonians for large $t$ shows that
all solutions of equations~(\ref{eq:floer}) and~(\ref{eq:lim}) are
contained in a compact set. Moreover, by condition~(\ref{eq:asph}),
there are no $J$-holomorphic spheres that can bubble off. 
Therefore the moduli space $\cM^A(\og,\ug;H,J)$ can be
compactified~\cite{F} to a space $\overline \cM^A(\og,\ug;H,J)$
consisting of all tuples 
$$
([u_k],[u_{k-1}],\ldots,[u_1]), \qquad
[u_i]\in \cM^{A_i}(\og_i,\ug_i;H,J)
$$
such that
$\ug_1=\ug$, $\og_i=\ug_{i+1}$, $\og_k=\og$ and $\sum_i A_i=A$. We
call such a tuple $([u_k],[u_{k-1}],\ldots,[u_1])$ a {\bf broken
   trajectory} of level $k$. The topology of the compactified moduli
space is described by the following notion of convergence: a sequence
$[u^\nu]\in \cM^A(\og,\ug;H,J)$ is said to converge to the broken
trajectory $([u_k],[u_{k-1}],\ldots,[u_1])$ if there exist sequences
$s_i^\nu\in\R$, $1\le i\le k$ such that $s_i^\nu\cdot u^\nu$ converges
uniformly on compact sets to $u_i$.

\begin{figure}
         \begin{center}
\input{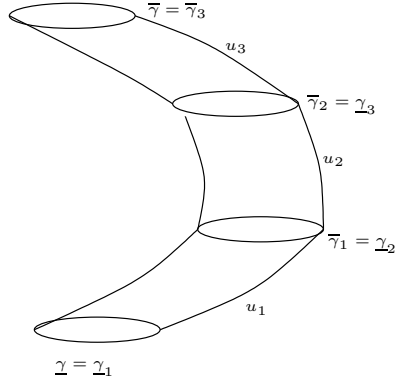}
\caption{Broken trajectory. \label{fig:broken}}
         \end{center}
\end{figure}

If the space $\widehat \cM^A(\og,\ug;H,J)$ is nonempty then its
dimension is strictly positive due to the action of $\R$. In this
case, the broken trajectories involved in the compactification have
level at most $\dim \, \widehat \cM^A(\og,\ug;H,J)$. In particular, when
$\mu(\og) - \mu(\ug) + 2\langle c_1(TW),A\rangle=1$ the moduli space
$\cM^A(\og,\ug;H,J)$
is compact and consists of a finite number of points. In this
situation one can associate a sign 
$\epsilon(u)$\index{$\epsilon(u)$, $\epsilon(u_\delta)$} 
to each element $[u]$ of this
moduli space~\cite{FH} (see also Section~\ref{sec:ori}). We
define the {\bf Floer differential} 
$$
\p : SC_*^a(H) \to SC_{*-1}^a(H)
\index{$\p$, Floer differential}
$$
by
\begin{equation}
   \label{eq:diff}
   \p \og := \sum_{\substack{
   \ug,\ A \\
  \mu(\og) - \mu(\ug) + 2\langle c_1(TW),A\rangle =1
  }}
\sum_{[u]\in \cM^A(\og,\ug;H,J)} \epsilon(u) e^A \ug.
\end{equation}

According to Floer~\cite{F} we have $\p ^2=0$. We define the {\bf
   symplectic homology groups} of the pair $(H,J)$ by
$$
SH_*^a(H,J) := H_*(SC_*^a(H),\p).
$$
  
\begin{remark} \label{rmk:asph} 
{\rm
 In view of 
condition~\eqref{eq:asph} the Novikov ring  
$\Lambda_\om$ can be replaced by $\Z[H_2(W;\Z)]$, or even by $\Z$ at
the price of losing the grading. Indeed, the energy of a Floer
trajectory depends only on its endpoints, hence the moduli spaces 
 $\overline 
\cM(\og,\ug;H,J):=  \bigcup_A \overline \cM^A(\og,\ug;H,J)$
 are compact. Therefore the sum~\eqref{eq:diff} involves only a 
 finite number of classes $A$. }  
\end{remark}

By a standard argument~\cite{F} the groups $SH_*^a(H,J)$ do not depend
on $J\in \Jreg(H)$. Nevertheless, they {\it do} depend on $H$ and, in
order to obtain an invariant of $(W,\om)$, we need an additional algebraic
limit construction. We define an {\bf admissible homotopy of
   Hamiltonians} as a map $H:\R\times S^1\times \widehat W\to \R$ with
the following properties:
\begin{enum}
\item $H(s,\cdot,\cdot) = H_- \in \cH$ for $s\le -1$,
   $H(s,\cdot,\cdot) = H_+ \in \cH$ for $s\ge 1$;
\item $H < 0$ on $W$ and there exist $t_0\ge 0$ and functions
   $\alpha,\beta:\R \to \R$ such that, for all $t\ge t_0$, we have
   $$
    H(s,t,p) = \alpha(s) e^t + \beta(s);
   $$
\item $\p_s H \ge 0$.
\end{enum}
An {\bf admissible homotopy of almost complex structures} is a map
$J:\R\to\cJ$ such that $J(s)=J_-$ for $s\le -1$ and $J(s)=J_+$ for
$s\ge 1$. Given an admissible homotopy of Hamiltonians one defines
regular admissible homotopies of almost complex structures in the
usual way, by linearizing the equation
\begin{equation}
   \label{eq:floers}
   \p_su (s,\theta) + J(s,\theta,u(s,\theta))(\p_\theta u(s,\theta) -
   X_H^\theta(s,\theta,u(s,\theta))) =0,
\end{equation}
subject to the limit conditions
\begin{equation}
   \label{eq:lims}
   \lim_{s\to -\infty} u(s,\cdot) = \og \in \cP(H_-), \qquad
   \lim_{s\to \infty} u(s,\cdot) = \ug \in \cP(H_+).
\end{equation}
Regular admissible homotopies of almost complex structures form again
a set of the second category in the space of admissible homotopies and
the rigid behaviour of $H$ for $t\ge t_0$, together with the condition
$\p_sH\ge 0$, ensures again that solutions of~(\ref{eq:floers})
and~(\ref{eq:lims}) stay in a compact set (see~\cite{O}). The usual
count of solutions of~(\ref{eq:floers}) and~(\ref{eq:lims}) induces
the {\bf monotonicity morphism}
\begin{equation}
   \label{eq:mono}
   \sigma :SH_*^a(H_-) \to SH_*^a(H_+),
\end{equation}
which does not depend on the choice of admissible homotopy connecting
$H_-$ and $H_+$. These morphisms form a direct system on
the set $\left\{ SH_*^a(H), \ H\in \cH \right\}$ and
we define the {\bf symplectic homology groups} of $(W,\om)$ by
\begin{equation*}
   SH_*^a(W,\om) := \lim _{\substack{ \to \\ H\in \cH}
    } SH_*^a(H). 
 \index{$SH_*^a(W,\om)$, symplectic homology}
\end{equation*}
According to~\cite[Lemma~3.7]{C} and~\cite[Theorem~1.7]{V} these
groups do not depend on the choice of the Liouville vector field $X$.


\section{The Morse-Bott chain complex} \label{sec:MBcomplex}

In this section we apply the Morse-Bott formalism of~\cite{B} to the
case of Hamiltonians $H:\widehat W \to \R$ having circles of
$1$-periodic orbits. 

We denote by $\cP_\lambda$
\index{$\cP_\lambda,\cP_\lambda^b,\cP_\lambda^{i^{-1}(a)}$|(} 
the set of closed unparametrized
$R_\lambda$-orbits in $M$. For each free homotopy class of loops $b$
in $M$ we denote by $\cP_\lambda^b$ the set of all $\gamma\in
\cP_\lambda$ in the homotopy class $b$. The inclusion
$i:M\hookrightarrow W$ induces a map (still denoted by $i$) between
the sets of free homotopy classes of loops in $M$ and $W$
respectively. For each
free homotopy class $a$ in $W$ we denote
$$
\cP_\lambda^{i^{-1}(a)} := \bigcup _{b\in i^{-1}(a)} \cP_\lambda^b.
\index{$\cP_\lambda,\cP_\lambda^b,\cP_\lambda^{i^{-1}(a)}$|)} 
$$
We assume in this section that the closed Reeb orbits on $M$ are
transversally nondegenerate in $M$. This means that, for every orbit $\gamma$
of period $T>0$, we have
$$
\det \left( \one - d\phi_\lambda^T(\gamma(0))|_\xi \right) \neq 0.
$$
This can always be achieved by an arbitrarily small perturbation of
$\lambda$ or, equivalently, of $X$, and such perturbations do not
change the symplectic homology groups. If all orbits $\gamma\in
\cP_\lambda$ are transversally nondegenerate one can assign to each of
them a Conley-Zehnder index $\mu_{CZ}(\gamma)$ according to the
following recipe. 

We fix a reference loop $l_b:S^1\to M$ for each free
homotopy class $b$ in $M$ such that $[l_b]=b$. If $b$ is the
trivial homotopy class we choose $l_b$ to be a constant loop and we
require that $l_{b^{-1}}$ coincides with $l_b$ with the opposite
orientation. We define symplectic trivializations
$$
\Phi_b:S^1\times \R^{2n-2} \to l_b^*\xi
$$
as follows. For each class $b$ we choose a homotopy $h_{ab}:S^1\times
[0,1] \to W$ from $l_a$, $a=i(b)$ to $l_b$ such that
\begin{equation} \label{eq:CRtrivhom}
h_{a^{-1}b^{-1}}(\tau,\cdot) = h_{ab}(-\tau,\cdot).
\end{equation}
We extend the trivialization $\Phi_a:S^1\times \R^{2n} \to
l_a^*T\widehat W$ over the homotopy $h_{ab}$ to get a trivialization
$\Phi'_b:S^1\times \R^{2n} \to l_b^*T\widehat W$. This trivialization
is homotopic to another one, still denoted $\Phi'_b$, such that
\begin{eqnarray}
   \Phi'_b(S^1\times \R^{2n-2} \times \{0\} \times \{0\}) & = &
   l_b^*\xi, \nonumber \\
   \Phi'_b(S^1\times \{0\} \times \R \times \{0\}) & = & l_b^*\langle
   \frac \p {\p t} \rangle, \label{eq:homotopic_triv} \\
   \Phi'_b(S^1\times \{0\} \times \{0\} \times \R) & = & l_b^*\langle
   R_\lambda \rangle. \nonumber
\end{eqnarray}
We define $\Phi_b:=\Phi'_b|_{S^1\times \R^{2n-2} \times \{0 \} \times
   \{0\}}$. If $b$ is the trivial homotopy
class we choose $h_{ab}$ to be a path of
constant loops, so that $\Phi_b$ is constant.

We fix for each $\gamma\in
\cP_\lambda$ a map $\sigma_\gamma:\Sigma \to M$, where $\Sigma$
is a Riemann surface with two boundary components
$\p_0\Sigma$ (with the opposite boundary orientation) and
$\p_1\Sigma$ (with the boundary orientation), satisfying
\begin{equation} \label{eq:sigmaM}
\sigma|_{\p_0\Sigma} = l_{[\gamma]}, \qquad \sigma|_{\p_1\Sigma}=\gamma.
\end{equation}
For each $\gamma\in \cP_\lambda$ there exists a unique (up to
homotopy) trivialization
$$
\Phi_\gamma: \Sigma \times \R^{2n-2} \to \sigma_\gamma^*\xi
$$
such that $\Phi_\gamma=\Phi_{[\gamma]}$ on $\p_0\Sigma \times \R^{2n-2}$. Let
\begin{equation} \label{eq:CRtriv}
\Psi:[0,T]\to \textrm{Sp}(2n-2), \quad
\Psi(\tau):= \Phi_\gamma^{-1} \circ d\phi_\lambda^\tau(p) \circ
\Phi_\gamma, \quad 
p\in \im\,\gamma.
\end{equation}
Because $\gamma$ is nondegenerate we can define the
{\bf Conley-Zehnder index} $\mu(\gamma)$ by
\begin{equation}
   \label{eq:CRmu}
   \mu(\gamma):=\mu(\gamma,\sigma_\gamma) := \mu_{CZ}(\Psi),
\end{equation}
where $\mu_{CZ}(\Psi)$ is the Conley-Zehnder index of a path of
symplectic matrices~\cite{RS}.

\begin{remark} \rm If, in the previous construction, we replace
   $\sigma_\gamma$ with $\sigma_\gamma \# A$ for some $A\in
H_2(M;\Z)$, then the 
   resulting index will be
  \begin{equation} \label{eq:CRc1}
   \mu(\gamma,\sigma_\gamma \# A) = \mu(\gamma,\sigma_\gamma) +
   2\langle c_1(\xi),A\rangle.
  \end{equation}
Note that $c_1(\xi)=i^*c_1(TW)$ because $i^*TW=\xi \oplus \langle
\frac \p {\p t},R_\lambda \rangle$. Moreover, the parity of
$\mu(\gamma)$ is well-defined independently of the trivialization of
$\xi$ along $\gamma$.
\end{remark}

\begin{remark} \label{rmk:good and bad}
For each simple orbit $\gamma\in \cP_\lambda$ we denote by $\gamma^k$,
$k\in \Z^+$ its positive iterates. The parity of the Conley-Zehnder
index of all the odd, respectively even iterates is the same. If
these two parities differ we say that all even iterates $\gamma^{2k}$,
$k\in \Z^+$ are {\bf bad orbits}. It is proved
in~\cite[Lemma~3.2.4]{U} that the even iterates of a simple orbit 
$\gamma$ of period $T$ are bad if and only if
$d\phi_\lambda^T(p)|_\xi$, $p\in \im\,\gamma$ has an odd number of
real negative eigenvalues strictly smaller than $-1$ (see also
Lemma~\ref{lem:cover}).  
The orbits in $\cP_\lambda$ which are not bad are called {\bf good
orbits}. 
\end{remark}

We define a new class $\cH'$\index{$\cH'$, autonomous admissible
  Hamiltonians} of admissible 
Hamiltonians consisting of elements $H:\widehat W\to \R$ such that
\begin{enum}
\item $H|_W$ is a $C^2$-small Morse function and $H<0$ on $W$;
\item $H(p,t)=h(t)$ outside $W$, where $h(t)$ is a strictly
   increasing function with $h(t)=\alpha e^t+\beta$,
   $\alpha,\beta\in\R$, $\alpha \notin \textrm{Spec}(M,\lambda)$ for
   $t$ bigger than some $t_0$, and such that $h''-h'>0$ on $[0,t_0[$.
\end{enum}

Note that the $1$-periodic orbits of $X_H$ in $W$ are constant and
nondegenerate by assumption~(i). A direct computation shows that
\begin{equation} \label{eq:XHR}
X_h(p,t) = -e^{-t}h'(t) R_\lambda.
\end{equation}

The $1$-periodic orbits of $X_H$ fall in two classes:
\begin{enumerate}
\item[(1)] critical points of $H$ in $W$;
\item[(2)] nonconstant $1$-periodic orbits of $X_h$, located on levels
   $M\times \{t\}$, $t\in ]0,t_0[$, which are in one-to-one
   correspondence with closed $-R_\lambda$-orbits of period
   $e^{-t}h'(t)$.
\end{enumerate}

Recall that, for every critical point $\tp\in \mathrm{Crit}(H)$, the
corresponding constant $X_H$-orbit $\gamma_{\tp}$ has Conley-Zehnder index
$$ \mu(\gamma_{\tp}) = \ind(\tp;-H) - n, \qquad n= \frac 1
2 \dim \, W,
$$
where $\ind(\tp;-H)$ is the Morse index of $\tp$
with respect to $-H$~\cite[Lemma~7.2]{SZ}.

Let $\alpha := \lim_{t \to \infty} e^{-t} H(p,t)$.
We denote by $\cP_\lambda^{\le \alpha}$
\index{$\cP_\lambda^{\le \alpha}$}
the set of all $\gamma \in
\cP_\lambda$ such that $\int \gamma^*\lambda \le \alpha$.
Because $H$ is independent of $\theta$, every
orbit $\gamma\in\cP_\lambda^{\le \alpha}$ gives rise to a whole circle
of nonconstant $1$-periodic orbits $\gamma_H$ of $X_H$. We denote by
$S_\gamma$\index{$S_\gamma$, circle of orbits} 
the set of such orbits and identify $S_\gamma$ with its
image under the natural embedding $S_\gamma \to \widehat W$ given by
$\gamma_H \mapsto \gamma_H(0)$. Note that all elements of $S_\gamma$
differ by a shift in the parametrization, and that the $\gamma_H$
are noninjective if their minimal period is smaller than $1$.

\begin{lemma} Let $H\in \cH'$. Every nonconstant $1$-periodic orbit
   $\gamma_H$ of $H$ is transversally nondegenerate in $\widehat W$.
\end{lemma}

\begin{proof}
We have to show that the only eigenvector of
$d\phi_H^1(\gamma_H(0))$ corresponding to the eigenvalue $1$ is $\dot
\gamma_H(0)$. To this effect we note that $\xi$ is an invariant space
and that
$$
d\phi_H^1(\gamma_H(0))|_\xi =
\left( d\phi_\lambda^{e^{-t}h'(t)} \right) ^{-1}(\gamma_H(0))|_\xi .
$$
Because we have assumed that all $R_\lambda$-orbits are transversally
nondegenerate in $M$, it follows that $d\phi_H^1(\gamma_H(0))|_\xi$ has
no eigenvalue equal to one. On the other hand we have
$$
d\phi_H^1(\gamma_H(0))\frac \p {\p t} = \frac \p {\p t} -
e^{-t}(h''-h')R_\lambda.
$$
The conclusion follows because $h''(t)-h'(t)>0$.
\end{proof}

For each $\gamma\in \cP_\lambda^{\le \alpha}$ we choose a Morse function
$f_\gamma:S_\gamma\to \R$\index{$f_\gamma$, Morse function on $S_\gamma$}
with exactly one maximum $M$ and one minimum $m$.
We fix from now on an element $H\in \cH'$ and, for each
$\gamma\in\cP_\lambda$ corresponding to a nonconstant
$\gamma_H \in \cP(H)$, we denote by $\ell_\gamma \in \Z^+$
the maximal positive integer such that $\gamma_H(\theta +
\frac 1 {\ell_\gamma}) = \gamma_H(\theta)$ for all $\theta\in S^1$. 
We choose a symplectic trivialization
$\psi:=(\psi_1,\psi_2):U_\gamma\stackrel\sim\longrightarrow V\subset
S^1\times \R^{2n-1}$ between 
open neighbourhoods $U_\gamma\subset \widehat W$ of $\gamma_H(S^1)$ and $V$
of $S^1\times \{ 0 \}$, such that
$\psi_1(\gamma(\theta))=\ell_\gamma\theta$. Here $S^1\times \R^{2n-1}$
is endowed with the 
symplectic form  $\om_0:=\sum_{i=1}^n dq_i\wedge dp_i$, $q_1\in S^1$,
$(p_1,q_2,p_2,\ldots,q_n,p_n)\in \R^{2n-1}$. 
Let $\rho:S^1\times \R^{2n-1}\to [0,1]$ be a smooth cutoff function
supported in a small neighbourhood of $S^1\times \{0\}$ such that
$\rho|_{S^1\times \{0\}} \equiv 1$. For
$\delta>0$ and $(\theta,p,t)\in S^1\times U_\gamma$ we define
\begin{equation} \label{eq:Hdelta}
   H_\delta(\theta,p,t):= h(t) + \delta \rho(\psi(p,t))
   f_\gamma(\psi_1(p,t) - \ell_\gamma\theta).
   \index{$H_\delta$, perturbed Hamiltonian}
\end{equation}
The Hamiltonian $H_\delta$ coincides with $H$ outside the open sets
$S^1\times U_\gamma$. This is precisely the perturbation described
in~\cite[Proposition~2.2]{CFHW}. It is shown therein that, for
$\delta$ sufficiently small, the set $\cP(H_\delta)$ consists of the
following elements:
\begin{enumerate}
\item[(1)] constant orbits, which are the same as those of $H$;
\item[(2)] nonconstant orbits, which are nondegenerate and form pairs
   $(\gamma_\min,\gamma_\Max)$, where $\gamma\in \cP_\lambda^{\le\alpha}$
   and $\gamma_\min$, $\gamma_\Max$\index{$\gamma_\min$, $\gamma_\Max$, surviving orbits}
   coincide with the orbits in
   $S_\gamma$ starting at the minimum and the maximum of
   $f_\gamma$ respectively.
\end{enumerate}

\begin{lemma} \label{lem:maslov}
   The periodic orbits $\gamma_\min, \gamma_\Max \in \cP(H_\delta)$ satisfy
   \begin{equation}
     \label{eq:maslov}
     \mu(\gamma_\min) = \mu(\gamma) + 1, \qquad 
     \mu(\gamma_\Max) = \mu(\gamma).
   \index{$\mu(\gamma)$, index of Reeb orbit}   
   \end{equation}
\end{lemma}

\begin{proof}
   We denote by $\gamma_H$ the $1$-periodic orbit of $X_H$
   corresponding to $\gamma\in \cP_\lambda^{\le\alpha}$. We define the
   Robbin-Salamon index of $\gamma_H$ by
$$
\mu_{RS}(\gamma_H) := \mu(\Psi),
$$
where $\Psi:[0,1]\to \textrm{Sp}(2n)$ is given by~(\ref{eq:triv}) and
$\mu(\Psi)$ is the Robbin-Salamon index of an arbitrary path of
symplectic matrices~\cite[\S4]{RS}. It is shown
in~\cite[Proposition~2.2]{CFHW} that
\begin{equation} \label{eq:RS}
-\mu(\gamma_\min) = \mu_{RS}(\gamma_H) - \frac 1 2, \qquad
-\mu(\gamma_\Max)=\mu_{RS}(\gamma_H) + \frac 1 2.
\end{equation}
Note that $\gamma_H$ has the orientation of $-R_\lambda$. Define
$\widetilde \Psi:[0,1]\to \textrm{Sp}(2n)$ by
$$
\widetilde \Psi(\theta):= \Phi_\gamma^{-1}(-\theta) \circ
d\phi_H^{-\theta}(\gamma(0)) \circ \Phi_\gamma(0),
$$
where $\Phi_\gamma:\R/\Z\times \R^{2n} \to \gamma_H^*T\widehat W$ is the
trivialization involved in~(\ref{eq:triv}). Then
$\mu_{RS}(-\gamma_H)=-\mu_{RS}(\gamma_H)=\mu(\widetilde \Psi)$.

Let $\textrm{Sp}^*(2n)\subset \textrm{Sp}(2n)$ be the set of
symplectic matrices with no eigenvalue equal to $1$ and
recall that we have denoted by $a$ free homotopy classes of loops
in $W$ and by $b$ free homotopy classes of loops in $M$.
By our choice~(\ref{eq:CRtrivhom}) and~(\ref{eq:homotopic_triv}) of
trivializations of $T\widehat W$ over the reference loops $l_b$, $b\in
i^{-1}(a)$ we deduce that the path $\widetilde \Psi$ is homotopic with
endpoint in $\textrm{Sp}^*(2n)$ to the path
$$
[0,1]\to \textrm{Sp}(2n) \ : \ \theta \mapsto \Psi_\lambda\left(T\theta \right)
\oplus \left(\begin{array}{cc} 1 & 0 \\ T \theta & 1
              \end{array}
        \right).
$$
Here $T:=e^{-t}(h''(t)-h'(t))$ and $\Psi_\lambda: [0,T]\to
\textrm{Sp}(2n-2)$ is defined by~(\ref{eq:CRtriv}). By the symplectic
shear axiom for the Robbin-Salamon index~\cite[Theorem~4.1]{RS} the
index of the above path is $\mu(\gamma) + \frac 1 2$. As a consequence
$$
-\mu_{RS}(\gamma_H)= \mu_{RS}(-\gamma_H) = \mu(\gamma) + \frac 1 2.
$$
Together with~(\ref{eq:RS}) this yields the conclusion of the Lemma.
\end{proof}

Let $p \in {\rm Crit}(f_\gamma)$; then $\gamma_p \in \cP(H_\delta)$ for all
$\delta\in]0,\delta_0]$ if $\delta_0$ is small enough, and
Lemma~\ref{lem:maslov} says that
$\mu(\gamma_p)=\mu(\gamma)+\textrm{ind}(p;f_\gamma)$. 
If $\tp$ is a critical point of $H$ in $W$ we
denote by $\ug_{\tp}\in\cP(H)$ the corresponding constant orbit.
Our goal is to describe the boundary points as $\delta\to 0$ of
\begin{equation} \label{eq:bigMdelta}
\cM^A_{]0,\delta_0[}(\og_p,\ug_q;H,\{f_\gamma\},J) := \bigcup _{0< \delta <
   \delta_0} \{ \delta \} \times \cM^A(\og_p,\ug_q;H_\delta,J),
\end{equation}
\index{$\cM^A_{]0,\delta_0]}(\og_p,\ug_q;H,\{f_\gamma\},J)$}
with
$$
\mu(\og_p)-\mu(\ug_q) + 2\langle c_1(TW),A \rangle =1,
$$
where
$$
\og,\ug\in \cP_\lambda^{\le\alpha}, \ p\in \mathrm{Crit}(f_{\og}),\ q\in
\mathrm{Crit}(f_{\ug}),\ A\in H_2(W;\Z),\ J\in \cJ
$$
or
$$
\og\in\cP_\lambda^{\le\alpha}, \ p\in \mathrm{Crit}(f_{\og}),\ q\in
\mathrm{Crit}(H),\ A\in H_2(W;\Z),\ J\in \cJ.
$$

Our description is very
similar to that of~\cite{B} within the setting of contact homology.
We fix $J\in \cJ$, $\og,\ug\in\cP_\lambda^{\le\alpha}$ and $\tq\in
\mathrm{Crit}(H)$. We define two Morse-Bott spaces of {\bf Floer 
   trajectories} $\widehat \cM^A(S_{\og},S_{\ug};H,J)$ 
\index{$\cM^A(S_{\og},S_{\ug};H,J),\widehat\cM^A(S_{\og},S_{\ug};H,J)$|(}
and
$\widehat \cM^A(S_{\og},\tq;H,J)$ as follows.
\index{$\cM^A(S_{\og},\tq;H,J),\widehat\cM^A(S_{\og},\tq;H,J)$|(}

For $\og,\ug\in\cP_\lambda^{\le\alpha}$ we denote by $\widehat
\cM^A(S_{\og},S_{\ug};H,J)$ the set
of solutions $u:\R\times S^1 \to \widehat W$ of the Floer
equation~(\ref{eq:floer}) subject to the asymptotic conditions
\begin{equation}
   \label{eq:MBlim}
   \lim_{s\to -\infty} u(s,\theta) = \og_H(\theta), \qquad \lim_{s\to
     \infty} u(s,\theta)= \ug_H(\theta), \qquad \lim_{s\to \pm\infty}
   \p_s u =0
\end{equation}
uniformly in $\theta$, with
\begin{equation}
   \label{eq:MBlim1}
   \og_H\in S_{\og}, \qquad \ug_H\in S_{\ug}
\end{equation}
and
$$
[\sigma_{\og}\# u]=[\sigma_{\ug}\# A].
$$
It is implicit in the above definition that the orbits $\og_H$ and
$\ug_H$ may vary for different elements of~$\widehat
\cM^A(S_{\og},S_{\ug};H,J)$.

For $\og\in \cP_\lambda^{\le\alpha}$ and $\tq\in \mathrm{Crit}(H)$ we denote by
$\widehat \cM^A(S_{\og},\tq;H,J)$ the set
of solutions $u:\R\times S^1 \to \widehat W$ of the Floer
equation~(\ref{eq:floer}) subject to the asymptotic conditions
\begin{equation}
   \label{eq:MBlim2}
   \lim_{s\to -\infty} u(s,\theta) = \og_H(\theta), \qquad \lim_{s\to
     \infty} u(s,\theta)= \tq, \qquad \lim_{s\to \pm\infty}
   \p_s u =0
\end{equation}
uniformly in $\theta$, with
\begin{equation}
   \label{eq:MBlim3}
   \og_H\in S_{\og}
\end{equation}
and
$$
[\sigma_{\og}\# u]=A.
$$
Again, the orbit $\og_H$ may vary for different elements of~$\widehat
\cM^A(S_{\og},\tq;H,J)$.

If $\og\neq\ug$ or $A\neq 0$, the additive group $\R$ acts freely on
$\widehat \cM^A(S_{\og},S_{\ug};H,J)$
and $\widehat \cM^A(S_{\og},\tq;H,J)$
by $s_0\cdot u(\cdot,\cdot):=
u(s_0+\cdot,\cdot)$. We define the {\bf Morse-Bott moduli spaces of
   Floer trajectories} by
$$
\cM^A(S_{\og},S_{\ug};H,J):=\widehat \cM^A(S_{\og},S_{\ug};H,J)/\R
$$
\index{$\cM^A(S_{\og},S_{\ug};H,J),\widehat\cM^A(S_{\og},S_{\ug};H,J)$|)}
and
$$
\cM^A(S_{\og},\tq;H,J):=\widehat \cM^A(S_{\og},\tq;H,J)/\R.
$$
\index{$\cM^A(S_{\og},\tq;H,J),\widehat\cM^A(S_{\og},\tq;H,J)$|)}
If $\og=\ug$ and $A=0$, the space $\widehat \cM^0(S_{\og},S_{\og};H,J)$
is diffeomorphic to $S_{\og}$, consists of constant cylinders (i.e. independent
of $s$) and the $\R$ action is trivial. In this case, we define the Morse-Bott
moduli spaces by $\cM^0(S_{\og},S_{\og};H,J):=\widehat 
\cM^0(S_{\og},S_{\og};H,J)$.
We have natural evaluation maps
$$
\oev : \cM^A(S_{\og},S_{\ug};H,J) \to S_{\og}, \qquad
\uev : \cM^A(S_{\og},S_{\ug};H,J) \to S_{\ug}
\index{$\oev$, $\uev$, evaluation maps}
$$
and
$$
\oev : \cM^A(S_{\og},\tq;H,J) \to S_{\og}
$$
defined by
$$
\oev([u]):= \lim_{s\to-\infty} u(s,\cdot), \qquad \uev([u]):=
\lim_{s\to \infty} u(s,\cdot).
$$

In the statement of the next result we denote by $\cJ'$ 
\index{$\cJ'$, time-indep. admissible a.c. structures} the set of
almost complex structures $J\in\cJ$ 
which are independent of $\theta\in S^1$.

\begin{proposition} \label{prop:Jreg} 
\begin{enum} 
  \item Given $H\in\cH'$, let $\cJ'(H)\subset \cJ'$ be the
(nonempty and open) set
of almost complex structures $J$ such that, for any $x\in \widehat W$
located on a simple $1$-periodic orbit of $X_H$, we have  
\begin{equation} \label{eq:bracket} 
  [X_H,J X_H](x) \neq 0 \quad \mbox{ and } \quad [X_H,J X_H](x)\notin
\langle X_H, J X_H \rangle. 
\end{equation}
There exists a set of second category
$\Jreg'(H)\subset \cJ'(H)$\index{$\cJ'_{\textrm{reg}}(H)$} 
consisting of almost complex structures
$J$ that are regular for all $u\in\widehat
\cM^A(S_\og,S_\ug;H,J)$ or $u\in\widehat \cM^A(S_\og,\tq;H,J)$ with
$\og$ or $\ug$ being a simple orbit, and such that $J\xi=\xi$, $J\frac
\p {\p t}=R_\lambda$ outside a fixed open neighbourhood of the
nonconstant periodic orbits of $X_H$. 
  \item Given $H\in\cH'$, there exists a set of second category
$\Jreg(H)\subset \cJ$\index{$\cJ_{\textrm{reg}}(H)$} consisting of
regular almost complex structures $J$ which, outside a fixed open 
neighbourhood of the nonconstant periodic orbits of $X_H$, are  
independent of $\theta$ and satisfy $J\xi=\xi$, $J\frac \p {\p t} = 
R_\lambda$. 
\end{enum}

\medskip 

In each of the previous cases the relevant moduli spaces 
$\cM^A(S_{\og},S_{\ug};H,J)$, $\cM^A(S_{\og},\tq;H,J)$
are smooth manifolds of dimension  
\begin{eqnarray*}
\dim \, \cM^A(S_{\og},S_{\ug};H,J) & = & \mu(\og) - \mu(\ug) + 2\langle
c_1(TW),A\rangle,\\
\dim \, \cM^A(S_{\og},\tq;H,J) & = & \mu(\og) - \mu(\gamma_{\tq}) + 2\langle
c_1(TW),A\rangle,
\end{eqnarray*}
and the evaluation maps $\oev,\uev$ are smooth.
\end{proposition}
The proof of this statement is given in
Section~\ref{sec:FredMB}. Unless the contrary is explicitly mentioned,
all statements in this section hold both for $J\in\Jreg(H)$ or
$J\in\Jreg'(H)$, provided one considers moduli spaces with at least
one simple asymptotic orbit in the latter case. 

\medskip

Let now $J\in \Jreg(H)$ and fix for each $\gamma\in \cP_\lambda^{\le\alpha}$ a
metric on $S_\gamma$ such that $R_\lambda$ has length one.
Let $\Freg(H,J)$\index{$\cF_{\textrm{reg}}(H,J)$, regular families $\{f_\gamma\}$} 
be the set of {\bf regular Morse functions},
consisting of families $\{f_\gamma\}$, $\gamma\in \cP_\lambda^{\le\alpha}$ of
perfect Morse functions $f_\gamma:S_\gamma\to \R$
\index{$f_\gamma$, Morse function on $S_\gamma$} 
such that all the
maps $\oev$ are transverse to the unstable manifolds $W^u(p)$, $p\in
\textrm{Crit}(f_\gamma)$, all the maps $\uev$ are transverse to the
stable manifolds $W^s(p)$, $p\in \textrm{Crit}(f_\gamma)$ and all
pairs
$$
(\oev,\uev) : \cM^A(S_{\og},S_{\ug};H,J)   \to  S_{\og} \times
S_{\ug}, 
$$
\begin{equation} \label{eq:transvf}
(\oev,\uev) : \cM^{A_1}(S_{\og},S_{\gamma_1};H,J) \ _{\uev}\times_{\oev} 
\cM^{A_2}(S_{\gamma_1},S_{\ug};H,J)   \to  S_{\og} \times
S_{\ug}
\end{equation} 
are transverse to products $W^u(p)\times W^s(q)$, $p\in
\textrm{Crit}(f_{\og})$, $q\in \textrm{Crit}(f_{\ug})$.
Here and in the sequel the unstable and stable manifolds are
understood with respect to $\nabla f_\gamma$. Denote by
$C^\infty_{p}(S_\gamma,\R)$ the set of perfect Morse
functions on $S_\gamma$.

\begin{lemma} \label{lem:Freg}
    The set $\Freg(H,J)$ is of the second Baire category in the space
    $\prod _{\gamma \in \cP_\lambda^{\le\alpha}}
    C^\infty_{p}(S_\gamma,\R)$.
\end{lemma}

\begin{proof}
   The first two transversality conditions on $\oev$, $\uev$ are
satisfied if and only if the maximum of
each function $f_\gamma$ is a regular
value of all the evaluation maps $\oev$ having $S_\gamma$ as target
space, and if the minimum of each $f_\gamma$ is a
regular value of all the evaluation maps $\uev$ mapping to $S_\gamma$.
The third transversality condition requires in addition that each pair
$(\overline{M},\underline{m})\in S_{\og} \times S_{\ug}$, with
$\overline{M}$ the maximum of $f_{\og}$ and $\underline{m}$ the
minimum of $f_{\ug}$, is a regular value of $(\oev,\uev)$.

By Sard's theorem the minimum and maximum of each $f_\gamma$ can be
chosen inside a set of second category in $S_\gamma$. The conclusion
follows.
\end{proof}

Let now $J\in \Jreg(H)$ and $\{f_\gamma\}\in \Freg(H,J)$. For $p\in
\textrm{Crit}(f_\gamma)$ we denote the Morse index by 
$$\ind(p):=\dim \, W^u(p;\nabla f_\gamma).
\index{$\ind(p)$, index of critical point of $f_\gamma$}  
$$
Let $\og,\ug\in\cP_\lambda^{\le\alpha}$ and $p\in
\mathrm{Crit}(f_{\og})$, $q\in \mathrm{Crit}(f_{\ug})$. For $m\ge 0$
we denote by 
\begin{equation} \label{eq:bigM}
\cM_m^A(p,q;H,\{f_\gamma\},J)
\index{$\cM_m^A(p,q;H,\{f_\gamma\},J)$, $\cM^A(p,q;H,\{f_\gamma\},J)$|(}
\end{equation}
the union for $\gamma_1,\ldots,\gamma_{m-1}\in\cP_\lambda^{\le\alpha}$ and
$A_1+\ldots +A_m=A$ of the fibered products 
\begin{eqnarray*}
&&
\hspace{-.7cm}W^u(p) 
\times_{\oev}
(\cM^{A_1}(S_{\og}\,,S_{\gamma_1})\!\times\!\R^+)
{_{\varphi_{f_{\gamma_1}}\!\circ\uev}}\!\times  
_{\oev}
(\cM^{A_2}(S_{\gamma_1},S_{\gamma_2})\!\times\!\R^+)
{_{\varphi_{f_{\gamma_2}}\!\circ\uev}\times_{\oev}} \\
&& \hspace{-.7cm}
\ldots\,
{_{\varphi_{f_{\gamma_{m-1}}}\circ\uev}}\times
_{\oev} 
\cM^{A_m}(S_{\gamma_{m-1}},S_{\ug}) {_{\uev}\times}
W^s(q),
\end{eqnarray*}
with the convention $\gamma_0=\og$. This is well defined as a smooth
manifold of dimension 
\begin{eqnarray*}
\lefteqn{\dim \, \cM_m^A(p,q;H,\{f_\gamma\},J)} \\
& = & \ind(p) - 1 + (\dim \,
\cM^{A_1}(S_{\og},S_{\gamma_1}) +1) - 1 \\
&&+ (\dim \,
\cM^{A_2}(S_{\gamma_1},S_{\gamma_2}) +1) -1 + ... \\ 
&& + \dim \,
\cM^{A_m}(S_{\gamma_{m-1}},S_\ug) - 1 + (1-\ind(q)) \\
& = & \mu(\og) + \ind(p) - \mu(\ug) - \ind(q) + 2\langle
c_1(TW),A_1+...+A_m\rangle - 1\\
& = & \mu(\og_p) - \mu(\ug_q) + 2 \langle c_1(TW),A\rangle -1.
\end{eqnarray*} 
The last equality follows from Lemma~\ref{lem:maslov}. Note that
$\cM_0^A(p,q;H,\{f_\gamma\},J)$ is naturally a submanifold of
$\cM^A(S_{\og},S_{\ug};H,J)$. We denote 
$$
\cM^A(p,q;H,\{f_\gamma\},J):=\bigcup_{m\ge 0}
\cM_m^A(p,q;H,\{f_\gamma\},J)
\index{$\cM_m^A(p,q;H,\{f_\gamma\},J)$, $\cM^A(p,q;H,\{f_\gamma\},J)$|)}
$$
and we call this {\bf the moduli space of Morse-Bott broken
trajectories}, whereas $\cM_m^A(p,q;H,\{f_\gamma\},J)$ is called {\bf
the moduli space of Morse-Bott broken trajectories with $m$
sublevels} (see also Definition~\ref{defi:MBbroken} and
Figure~\ref{fig:wtilde}).  

Similarly, given $\og\in\cP_\lambda^{\le\alpha}$,
$p\in \mathrm{Crit}(f_{\og})$, $\tq\in \mathrm{Crit}(H)$, we define
moduli spaces of Morse-Bott broken trajectories
$\cM_m^A(p,\tq;H,\{f_\gamma\},J)$, $m\ge 0$ and 
$\cM^A(p,\tq;H,\{f_\gamma\},J)$ by replacing the last 
term $\cM^{A_m}(S_{\gamma_{m-1}},S_\ug) {_{\uev}\times}
W^s(q)$ in the definition~\eqref{eq:bigM} with 
$\cM^{A_m}(S_{\gamma_{m-1}},\tq;H,J)$. This is again well defined as a
smooth manifold of dimension 
$$
\dim \, \cM^A(p,\tq;H,\{f_\gamma\},J) = \mu(\og_p) - \mu(\gamma_\tq) + 2
\langle c_1(TW),A\rangle -1. 
$$
Again, $\cM_0^A(p,\tq;H,\{f_\gamma\},J)$ is naturally a submanifold of
$\cM^A(S_{\og},\tq;H,J)$. 

\medskip 

The significance of the above moduli spaces of broken Morse-Bott
trajectories is explained by the following theorem, which describes  
the boundary of $\cM^A_{]0,\delta_0[}(\og_p,\ug_q;H,\{f_\gamma\},J)$ 
in~\eqref{eq:bigMdelta} as $\delta\to 0$. 

\begin{theorem}[{\bf Correspondence Theorem}] \label{thm:degen} 
Let $H\in \cH'$ be fixed and
   let $\alpha:=\lim_{t\to \infty} e^{-t}H(p,t)$ be the maximal slope
   of $H$. Let $J\in \Jreg(H)$ and $\{f_\gamma\}\in \Freg(H,J)$. There
   exists
$$
\delta_1:=\delta_1(H,J) \in \, ]0,\delta_0[
$$
such that, for any
$$
\og,\ug\in\cP_\lambda^{\le\alpha},\ p\in
\mathrm{Crit}(f_{\og}),\ q\in\mathrm{Crit}(f_{\ug}),
$$
or
$$
\og\in\cP_\lambda^{\le\alpha},\ p\in
\mathrm{Crit}(f_{\og}),\ q\in\mathrm{Crit}(H),
$$
and any $A\in H_2(W;\Z)$ with
$$
\mu(\og_p)-\mu(\ug_q) + 2\langle c_1(TW),A\rangle =1,
$$
the following hold:
\begin{enum}
  \item $J$ is regular for $\cM ^A(\og_p,\ug_q;H_\delta,J)$
  for all $\delta\in ]0,\delta_1[$; 
  \item the space
    $\cM^A_{]0,\delta_1[}(\og_p,\ug_q;H,\{f_\gamma\},J)$ is a
    $1$-dimensional manifold having a finite number of components which
    are graphs over $]0,\delta_1[$, i.e. the natural projection
    $\cM^A_{]0,\delta_1[}(\og_p,\ug_q;H,\{f_\gamma\},J) \to
    ]0,\delta_1[$ is a submersion;
  \item there is a bijective correspondence between
    points
$$
[\u]\in  \cM^A(p,q;H,\{f_\gamma\},J)
$$
and connected components
of $\cM^A_{]0,\delta_1[}(\og_p,\ug_q;H,\{f_\gamma\},J)$. 
\end{enum}
\end{theorem}

The proof of this statement, including a discussion of gluing and 
compactness for Morse-Bott moduli spaces, is given in
Section~\ref{sec:FredMB}.

We assume in the remainder of this section that the conclusions of
Theorem~\ref{thm:degen} are satisfied. For each $[\u]\in
\cM^A(p,q;H,\{f_\gamma\},J)$ the sign 
$\epsilon(u_\delta)$\index{$\epsilon(u)$, $\epsilon(u_\delta)$} 
is constant on the corresponding connected component $C_{[\u]}$ for
continuity reasons. We define a sign 
$\bar \epsilon(\u)$\index{$\bar \epsilon(\u)$} by 
\begin{equation} \label{eq:MBsign}
\bar\epsilon(\u):= \epsilon(u_\delta), \qquad \delta \in ]0,\delta_1[,
\quad (\delta,[u_\delta])\in C_{[\u]}.
\end{equation}

We define the {\bf Morse-Bott chain groups} by
\begin{eqnarray}  \label{eq:MBchain_a} 
\index{$BC_*^a(H)$, $BC_*^0(H)$, Morse-Bott chain groups}
   BC_*^a(H) & := &
\bigoplus _{  \gamma\in \cP_\lambda^{i^{-1}(a),\le\alpha}}
\Lambda_\om \langle \gamma_\min,\gamma_\Max\rangle, \qquad a\neq 0, \\
BC_*^0(H) & := & \bigoplus_{\scriptstyle \tp\in \mathrm{Crit}(H)}
   \Lambda_\om \langle \tp \rangle \ \oplus
\bigoplus _{ \gamma\in \cP_\lambda^{i^{-1}(0),\le\alpha} }
\Lambda_\om \langle \gamma_\min,\gamma_\Max\rangle.
\end{eqnarray}
where $\alpha:=\lim_{t\to \infty} e^{-t}H(p,t)$ and 
$\cP_\lambda^{i^{-1}(0),\le\alpha}
= \cP_\lambda^{\le\alpha} \cap \cP_\lambda^{i^{-1}(0)}$. The grading is
defined by
\begin{eqnarray*}
|e^A \tp| & := & \ind(\tp;-H) - n - 2\langle c_1(TW),A\rangle, \\
|e^A \gamma_\min | & := &  \mu(\gamma) + 1 - 2\langle c_1(TW),A\rangle,
\\
|e^A \gamma_\Max | & := &  \mu(\gamma) - 2\langle c_1(TW),A\rangle.
\end{eqnarray*}

We define the {\bf Morse-Bott differential}
$$
\p : BC_*^a(H) \to BC_{*-1}^a(H)
\index{$\p$, Morse-Bott differential}
$$
by
\begin{eqnarray} \label{eq:MBdiff_p}
   \p \tp & := & \sum_{\substack{ \tq\in \mathrm{Crit}(H) \\
      |\tp|-|\tq| =1}}
\ \sum_{\scriptstyle [\u]\in \cM^0(\tp,\tq;H,\{f_\gamma\},J)} \bar
   \epsilon(\u)\tq, \\ 
\label{eq:MBdiff_gamma}
\p \gamma_p & := & \sum_{\substack{
  \tq\in \mathrm{Crit}(H) \\
|\gamma_p| - |e^A \tq|=1}}
\ \sum_{\scriptstyle [\u]\in \cM^A(\gamma_p,\tq;H,\{f_\gamma\},J)}
\bar \epsilon(\u)e^A \tq \\
& & + \sum_{\substack{
  \ug\in \cP_\lambda^{\le\alpha}, q\in \mathrm{Crit}(f_{\ug}) \\
|\gamma_p| - |e^A \ug_q|=1}}
\ \sum_{\scriptstyle [\u]\in \cM^A(\gamma_p,\ug_q;H,\{f_\gamma\},J)}
\bar \epsilon(\u)e^A \ug_q, \qquad p \in {\rm Crit}(f_\gamma). \nonumber
\end{eqnarray}
The sums~(\ref{eq:MBdiff_p}) and~(\ref{eq:MBdiff_gamma}) clearly
involve only periodic orbits in the same free homotopy class as that
of $\tp$ or $\gamma_p$ respectively.

\begin{remark} \label{rmk:grad_traj} \rm
Since $H$ is $C^2$-small, the moduli spaces
$\cM^A(\tp,\tq;H_\delta,J)$, $\tp,\tq\in \mathrm{Crit}(H)$ of expected
dimension $\ind(\tp;-H)-\ind(\tq;-H) + 2\langle c_1(TW),A)\rangle -1=0$
are independent of $\delta$ and consist
exclusively of gradient trajectories of $H$ in
$W$~\cite[Theorem~6.1]{HS}(see also~\cite[Theorem~7.3]{SZ}). As a
consequence, these moduli spaces are empty whenever $A\neq 0$.
\end{remark}

We have, following directly from the definitions, an obvious
isomorphism of free $\Lambda_\om$-modules
$$
SC_*^a(H_\delta) \simeq BC_*^a(H), \qquad \delta \in ]0,\delta_1[.
$$
It follows now from Theorem~\ref{thm:degen} and the
definition~(\ref{eq:MBsign}) of signs in the Morse-Bott complex
that the corresponding differentials, defined by~(\ref{eq:diff})
and~(\ref{eq:MBdiff_p}-\ref{eq:MBdiff_gamma}), also coincide. Here we
use the fact that the Hamiltonian action functional decreases along Floer
trajectories, hence the differential~(\ref{eq:diff}) applied
to elements $\tp\in\mathrm{Crit}(H)$ does not involve nonconstant
elements of $\cP(H_\delta)$ and reduces to~(\ref{eq:MBdiff_p}) by
Remark~\ref{rmk:grad_traj}. As a consequence, we have
$$
H_*(BC_*^a(H),\p) = SH_*^a(H_\delta,J).
$$

We shall construct in Section~\ref{sec:ori} 
a system of coherent orientations on the Morse-Bott moduli spaces
$$
\cM^A(S_{\og},S_{\ug};H,J),\qquad \cM^A(S_{\og},\tq;H,J)
$$
whenever $\og,\ug\in \cP_\lambda^{\le\alpha}$ are
good orbits. This in turn determines signs 
$\epsilon(\u)$\index{$\epsilon(\u)$} 
via an
orientation rule for fiber products (see~\eqref{eq:signepsilon}). 

\begin{proposition} \label{prop:signs}
 Assume $\dim \, \cM^A(p,q;H,\{f_\gamma\},J) =0$. The bijective
 correspondence between elements
$[\u]\in\cM^A_m(p,q;H,\{f_\gamma\},J)$, $m\ge 1$ and 
 elements of $[u_\delta]\in \cM^A(\og_p,\ug_q;H_\delta,J)$ given by
 Theorem~\ref{thm:degen} changes the signs by the rule
$$
\epsilon(\u)=(-1)^{m-1}\epsilon(u_\delta).
$$
Moreover, if $m=0$ then $\u=u_\delta$ and
$\epsilon(\u)=\epsilon(u_\delta)$, $p$ is a minimum and $q$ is a
maximum, the moduli space 
$\cM^A_0(p,q;H,\{f_\gamma\},J)$ consists of 
the two gradient lines running from $p$ to $q$ and their signs 
are different if and only if the underlying orbit is good. 
\end{proposition}

In view of~\eqref{eq:MBsign}
and~(\ref{eq:MBdiff_p}--\ref{eq:MBdiff_gamma}), this identification of
signs between $\epsilon(\u)$ and $\bar \epsilon (\u)$ allows to define
the Morse-Bott differential exclusively in terms of Morse-Bott data.


\section{Morse-Bott moduli spaces} \label{sec:FredMB}

The structure of this section is as follows. We give
in~\S\ref{sec:transv} the proof of 
Proposition~\ref{prop:Jreg}, whereas 
Theorem~\ref{thm:degen} is proved
in~\S\ref{sec:compact}--\S\ref{sec:gluing}, which treat compactness and
gluing and correspond to assertions (i-ii) and (iii)
respectively. Finally \S\ref{sec:ori} contains a full discussion of
orientation issues and the proof of Proposition~\ref{prop:signs}. 

\subsection{Transversality} \label{sec:transv}

\begin{proof}[\bf Proof of Proposition~\ref{prop:Jreg}]
We first prove (ii).
Let $\cJ^\ell\subset\cJ$ be the space of admissible almost complex
structures of class $C^\ell$, $\ell \ge 1$, and let
$\cJ^\ell(H)\subset \cJ^\ell$ be the set of almost complex
structures $J$ which, outside a fixed neighbourhood of the nonconstant
periodic orbits of $X_H$, are independent of $\theta$ and 
satisfy $J\xi=\xi$, $J\frac \p {\p t} =R_\lambda$. 
By a standard trick of Taubes~\cite[Theorem 5.1]{FHS} it is enough to
show that there exists an open and dense set $\Jreg^\ell(H) \subset
\cJ^\ell(H)$ consisting of regular elements. We define the universal
moduli spaces 
$$
\cM^A(S_{\og},S_{\ug};H,\cJ^\ell(H)) = \{ (u,J) \ | \ J \in
\cJ^\ell(H), u \in \cM^A(S_{\og},S_{\ug};H,J) \}
$$
and
$$
\cM^A(S_{\og},\tq;H,\cJ^\ell(H)) = \{ (u,J)\  |\  J \in \cJ^\ell(H),
u \in \cM^A(S_{\og},\tq;H,J) \} .
$$
The main point is to show that these universal moduli spaces are 
Banach manifolds. Then the sets $\Jreg^\ell(H)$ consist of the regular
values of the natural projections from the universal moduli spaces to
$\cJ^\ell(H)$. We only treat the case of
$\cM^A(S_{\og},S_{\ug};H,\cJ^\ell(H))$ since the second case is
entirely similar, and we assume without loss of generality that
$\og\neq \ug$. This universal moduli space is the zero set of
a distinguished section of a Banach vector bundle $\cE \to \cB^A
\times \cJ^\ell(H)$ which we now define.  

Let $p > 2$ and $d > 0$. Let 
$\cB^A = \cB^{1,p,d}(S_{\og},S_{\ug},A;H)$
\index{$\cB^A = \cB^{1,p,d}(S_{\og},S_{\ug},A;H)$} 
be the space
of proper maps $u : \R \times S^1 \to \widehat W$ which are locally 
in $W^{1,p}$
and satisfy
\begin{enum}
\item the map $u$ converges uniformly in $\theta$ as $s \to \pm\infty$ to
$\ug(\cdot + \utheta_0)$, respectively $\og(\cdot + \otheta_0)$,
for some $\utheta_0,\otheta_0\in S ^1$, and represents the homology 
class $A\in H_2(W;\Z)$;

\item there exist tubular neighbourhoods $\oU$ and $\uU$ of $\og$ and $\ug$
respectively, together with parametrizations $\opsi : \oU \to S^1 
\times \R^{2n-1}$ and
$\upsi : \uU \to S^1 \times \R^{2n-1}$ such that
\begin{eqnarray*}
\opsi \circ \og (\theta) = \{ \theta \} \times \{ 0 \} , &\quad&
\upsi \circ \ug (\theta) = \{ \theta \} \times \{ 0 \} ,\\
\opsi \circ \og (\theta+\otheta_0) - \opsi \circ u (s,\theta) &\in&
W^{1,p}(]-\infty,-s_0],e^{d|s|} ds \, d\theta), \\
\upsi \circ \ug (\theta+\utheta_0) - \upsi \circ u (s,\theta) &\in&
W^{1,p}([s_0,\infty[,e^{d|s|} ds \, d\theta) ,
\end{eqnarray*}
for some $s_0 > 0$ sufficiently large.
\end{enum}
Then $\cB^A$ is a Banach manifold and, for $d/p$ strictly smaller than
the constant $r$ in Proposition~\ref{prop:asymptotic}, it
contains the moduli spaces $\cM^A(S_{\og},S_{\ug};H,J)$ for all $J 
\in \cJ^\ell$.
Let $\cE \to \cB^A \times \cJ^\ell(H)$ be the Banach vector bundle with fiber
$\cE_{(u,J)} = L^p(\R \times S^1, u^* T\widehat W; e^{d|s|} ds \, d\theta)$.
Let $\dbar_H : \cB^A \times \cJ^\ell(H) \to \cE$ be the section defined by
\begin{equation}  \label{eq:dbarHdelta}
 \index{$\dbar_H$}
\dbar_H(u,J) := \partial_s u + J_\theta (\partial_\theta u - X_H) .
\end{equation}

Then $\cM^A(S_{\og},S_{\ug};H,\cJ^\ell(H)) = \dbar_H^{-1}(0)$ and it
remains to show that $\dbar_H$ is transverse to the
zero section. This means that the vertical differential
$$
D\dbar_H(u,J) : T_u \cB^A \times T_J \cJ^\ell(H) \to \cE_{(u,J)}
$$
is surjective for all $(u,J) \in \dbar_H^{-1}(0)$.
We have
$$
T_u \cB^A = W^{1,p}(\R\times S^1,u^*T\widehat W;e^{d|s|}ds\,d\theta) 
\oplus \oV \oplus \uV,
$$
where $\oV$, $\uV$ are the one-dimensional real vector spaces generated by
two sections of $u^*T\widehat W$ of the form $(1-\beta(s,\theta)) X_H(\og(\theta))$
and $\beta(s,\theta)X_H(\ug(\theta))Ò$ respectively, with
$\beta(s,\theta)=\beta(s)$\index{$\beta$, cutoff function} a
smooth cutoff function which vanishes for $s\le 0$ 
and is equal to $1$ for $s\ge 1$. The space $T_J\cJ^\ell(H)$ consists of
matrix valued functions $Y:S^1\to \mathrm{End}(T\widehat W)$ of class
$C^\ell$ satisfying the conditions
\begin{equation} \label{eq:Y}
J_\theta Y_\theta + Y_\theta J_\theta = 0, \qquad \widehat \om 
(Y_\theta v,w) + \widehat \om
(v,Y_\theta w) =0, \ \forall v,w\in T\widehat W,
\end{equation}
and such that, outside fixed neighbourhoods of the nonconstant
periodic orbits of $X_H$, they are independent of $\theta$ and have
the form $\left(\begin{array}{cc} Y_\xi & 0 \\ 0 & 0
\end{array}\right)$ with respect to the splitting $\xi\oplus 
\mathrm{Span}(R_\lambda,\frac \p {\p t})$.
The operator $D\dbar_H(u,J)$ can be written
$$
D\dbar_H(u,J)\cdot (\zeta,Y) = D_u\zeta + Y_\theta(u)(\p_\theta u - X_H(u)).
$$
Here
$$
D_u:W^{1,p}(\R\times S^1,u^*T\widehat W;e^{d|s|}ds\,d\theta) \oplus 
\oV \oplus \uV
\to L^p(\R \times S^1, u^* T\widehat W; e^{d|s|} ds \, d\theta)
$$
is  the linearization of the Cauchy-Riemann operator associated to 
the pair $(H,J)$
and is explicitly given by formula~(\ref{eq:D}). It is proved
in~\cite[Proposition~4]{BM} that $D_u$
is a Fredholm operator. It is at this point that the exponential
weight plays a crucial role, due to the degeneracy of the asymptotic
orbits.  As a consequence the range of $D\dbar_H(u,J)$ is closed
and we are left  to prove that it is also dense. Let $q>1$ be such
that $1/p+1/q=1$. We show that
every $\eta\in L^q(\R \times S^1, u^* T\widehat W; e^{d|s|} ds \, 
d\theta)$ satisfying
\begin{equation} \label{eq:eta}
\int _{\R\times S^1} \langle \eta, D_u\zeta \rangle e^{d|s|} ds \, 
d\theta =0, \qquad
\int _{\R\times S^1} \langle \eta,  Y_\theta(u)(\p_\theta u - X_H(u))
\rangle e^{d|s|} ds \, d\theta =0
\end{equation}
for all $\zeta$ and $Y$ vanishes. The first equation implies,
by elliptic regularity, that $\eta$ is of class $C^\ell$ and has 
the unique continuation
property. Assume by contradiction that $\eta$ does not vanish.
Then the set $\{(s,\theta) \, : \, \eta(s,\theta)\neq 0\}$ is open 
and dense. On the other hand,
it is proved in~\cite[Theorem~4.3]{FHS} that the set
$$
R(u):=\{(s,\theta) \, : \, \p_s u(s,\theta)\neq 0, \ u(s,\theta)\neq
\ug(\theta), \og(\theta), \
u(s,\theta) \notin u(\R\setminus\{s\},\theta)\}
$$
of regular points of $u$ is open and dense (although nondegeneracy of
the asymptotic orbits is a standing assumption in~\cite{FHS}, it does
not play any role in the proof of this result).
Let $z_0=(s_0,\theta_0)$ be a point in $R(u)$ with
$\eta(z_0)\neq 0$ and $u(z_0)$ belonging to the fixed open  
neighbourhood of $\og$ (such a point exists since we have assumed 
$\og\neq\ug$). One can choose a matrix $Y_{\theta_0}(u(z_0))$  
satisfying~(\ref{eq:Y}) such that
$$
\langle \eta(z_0),Y_{\theta_0}(u(z_0))J(u(z_0))\p_s u(z_0)\rangle 
\neq 0. 
$$
Because $z_0$ is a regular point
we can choose a time-dependent cutoff function $\rho:S^1\times 
\widehat W \to [0,1]$
supported near $(\theta_0,u(z_0))$ such that
$Y:=\rho Y_{\theta_0}(u(z_0))$ satisfies
$$
\int _{\R\times S^1} \langle \eta,  Y_\theta(u)(\p_\theta u - X_H(u))
\rangle e^{d|s|} ds \, d\theta \neq 0.
$$
This contradicts~(\ref{eq:eta}) and shows that $D\dbar_H(u,J)$ is surjective,
hence the universal moduli space $\cM^A(S_{\og},S_{\ug};H,\cJ^\ell(H))$ is a
Banach manifold as claimed.

\medskip 

We now prove (i). The set $\cJ'(H)$ is obviously open. The fact that
it is nonempty can be seen as follows. The space $S^1\times \R^2$
admits the ``skating ring'' contact form $\alpha=\sin \theta dx - \cos
\theta dy$, $(\theta,x,y)\in S^1 \times \R^2$ for which 
$\frac \p {\p \theta}\in \xi=\ker\alpha$. If $J$ denotes the
almost complex structure on $\xi$ satisfying $J\frac \p {\p
\theta}=\cos \theta \frac \p {\p x} + \sin\theta \frac \p {\p y}$,
then $[\frac \p {\p \theta},J\frac \p {\p \theta}] \neq 0$ and 
$[\frac \p {\p \theta},J\frac \p {\p \theta}]\notin \xi =
\langle \frac \p {\p \theta},J\frac \p {\p \theta}\rangle$.   
This simple model can be adapted to our situation as follows. 
We can symplectically trivialize a neighbourhood of the simple orbit
$\gamma$ as $S^1\times \R^{2n-1} \ni
(\theta,t,q_2,p_2,\ldots,q_n,p_n)$ with the standard
symplectic form $d\theta\wedge dt +dq_2\wedge dp_2 +\ldots +dq_n
\wedge dp_n$, so that $X_H$ corresponds to $\frac \p {\p \theta}$. 
Let $J$ be a compatible almost complex structure such that $J \frac \p
{\p \theta}= \frac \p {\p t} + \cos \theta \frac \p {\p q_2} + \sin
\theta \frac \p {\p p_2}$. Since $\frac \p {\p \theta}$ and $\frac \p
{\p t}$ commute we have $[X_H,J X_H] = [\frac \p {\p \theta}, \cos \theta
\frac \p {\p q_2} + \sin \theta \frac \p {\p p_2}]\neq 0$ and 
$[X_H,J X_H] \notin \langle X_H,J X_H \rangle$, so that
$J\in\cJ'(H)$. 

Let $\cJ^{\prime\ell}\subset \cJ'$ be the
space of admissible almost complex structures of class $C^\ell$,
$\ell\ge 1$ which are independent of $\theta\in S^1$, and let
$\cJ^{\prime\ell}(H)\subset \cJ^{\prime\ell}$ be the space of almost
complex structures $J$ which, outside a fixed neighbourhood of the
nonconstant periodic orbits of $X_H$, satisfy $J\xi=\xi$, $J\frac\p{\p
t} = R_\lambda$. It is enough to show that there exists an open and
dense set $\Jreg^{\prime\ell}(H)\subset \cJ^{\prime\ell}(H)$
consisting of elements which are regular for Floer trajectories with
one nontrivial simple asymptote. 

We have $\cJ^{\prime\ell}\subset \cJ^\ell$ and the main point is to
show that the corresponding universal moduli spaces
$\cM^A(S_\og,S_\ug;H,\cJ^{\prime\ell}(H)) \subset
\cM^A(S_\og,S_\ug;H,\cJ^\ell(H))$ and
$\cM^A(S_\og,\tq;H,\cJ^{\prime\ell}(H)) \subset 
\cM^A(S_\og,\tq;H,\cJ^\ell(H))$ are Banach manifolds. We again treat
only $\cM^A(S_\og,S_\ug;H,\cJ^{\prime\ell}(H))$ and assume without loss
of generality that $\og$ is a simple orbit and $\og\neq \ug$. This
universal moduli space is the zero set of the section of 
the restricted bundle $\cE\to \cB^A\times \cJ^{\prime\ell}(H)$ defined
by~\eqref{eq:dbarHdelta}, and we have to show that the vertical
differential $D\dbar_H(u,J) : T_u \cB^A \times T_J \cJ^{\prime\ell}(H)
\to \cE_{(u,J)}$ is surjective. Arguing by
contradiction, we get an element $\eta\in L^q(\R\times
S^1,u^*T\widehat W; e^{d|s|}ds \, d\theta)$ of class $C^\ell$ which
does not vanish on an open and dense subset of $\R\times S^1$ and
satisfies $\int_{\R\times S^1} \langle \eta,Y(u)(\p_\theta
u-X_H(u))\rangle e^{d|s|}ds \, d\theta =0$ for any
$Y\in T_J \cJ^{\prime\ell}(H)$. The main difference with respect
to (ii) is that $T_J \cJ^{\prime\ell}(H)
\subset T_J \cJ^\ell(H)$ consists of elements
$Y\in\textrm{End}(T\widehat W)$ which are independent of $\theta$. 

Let 
$$
I(u):=\big\{ (s,\theta) \, : \, \p_s u(s,\theta)\neq 0, \
u^{-1}(u(s,\theta))=\{(s,\theta)\} \big\}
$$
be the set of {\bf injective} points, and denote $I_R(u):= I(u) \, \cap \, 
]R,\infty[\times S^1$, $R>0$. 
The main observation is that our special choice of $J\in\cJ'(H)$
implies that $I_R(u)$ is open and dense in $]R,\infty[\times S^1$ for
$R$ large enough. This is proved exactly as in~\cite[\S7]{FHS}, and
the main steps are the following. Since $\og$ is simple, every $u$ as
above is {\bf simple}, i.e. for every integer $m>1$ there exists
a point $(s,\theta)\in \R \times S^1 = \R\times \R/\Z$ such that
$u(s,\theta+\frac 1 
m) \neq u(s,\theta)$.  Let $U$ be a
neighbourhood of $\og$ in which $[X_H, JX_H]\neq 0$ and $[X_H,J
X_H]\notin \langle X_H, JX_H \rangle$. We call a point $(s,\theta)$
{\bf regular} if $\p_s u$, $\p_\theta u$, $X_H(u)$, $JX_H(u)$ are linearly
independent at $(s,\theta)$, and we denote by $R(u)$ the set of
regular points. Then~\cite[Lemma~7.6]{FHS} holds
and~\cite[Lemma~7.7]{FHS} shows that the set $\{(s,\theta) \in R(u) \,
: \, u(s,\theta) \in U\}$ is open and dense in $u^{-1}(U)$. Note that
we crucially use here our hypothesis $J\in \cJ'(H)$, which plays the
role of the hypothesis $J\in \cJ_{\textrm{ad}}(M,\om,X)$
in~\cite{FHS}. Finally~\cite[Lemma~7.8]{FHS} shows that the set of 
points which are regular and injective is open and dense in
$u^{-1}(U)$, and in particular $I_R(u)$ is open and dense in
$]R,\infty[\times S^1$ for $R$ large enough. 

We can then choose a point $z_0=(s_0,\theta_0)\in I_R(u)$ 
such that $\eta(z_0)\neq 0$ and a matrix $Y(u(z_0))$
satisfying~\eqref{eq:Y} and  $\langle
\eta(z_0),Y(u(z_0))J(u(z_0))\p_s u(z_0)\rangle \neq 0$. Since $z_0$ is
an injective point we can choose a cutoff function $\rho :\widehat
W\to \R$ supported near $u(z_0)$ such that $Y:=\rho Y(u(z_0))$
satisfies $\int_{\R\times S^1} \langle \eta,Y(u)(\p_\theta u -
X_H(u))\rangle e^{d|s|}ds\,d\theta \neq 0$. This contradiction
shows that $D\dbar_H(u,J)$ is surjective and therefore
$\cM^A(S_\og,S_\ug;H,\cJ^{\prime\ell}(H))$ is a Banach manifold
as claimed. 

\medskip 

The dimension of the moduli space $\cM^A(S_{\og},S_{\ug};H,J)$, $J 
\in \Jreg(H)$
is equal to  ind$(D_u) - 1$. The restriction of the operator $D_u$ to 
the subspace
$W^{1,p}(\R\times S^1,u^*T\widehat W;e^{d|s|}ds\,d\theta)$ is conjugated to
a Cauchy-Riemann operator
$$
\cD_u : W^{1,p}(\R\times S^1,u^*T\widehat W; ds\,d\theta) \to
L^p(\R\times S^1,u^*T\widehat W; ds\,d\theta)
$$
via multiplication by $e^{\frac{d}{p}|s|}$. If the asymptotics of $D_u$ were
nondegenerate, the Fredholm index of $\cD_u$ would be given by~\cite{S}
$$
\mu_{RS}(\og) - \mu_{RS}(\ug) + 2 \langle c_1(TW), A \rangle .
$$
Due to the one-dimensional degeneracy of $\og$ and $\ug$, the actual index
of $\cD_u$ is obtained by a calculation analogous to~\cite[Proposition 4]{BM}
(see also Lemma~\ref{lem:maslov})~:
\begin{equation} \label{eq:indMB}
(\mu_{RS}(\og)-\frac12) - (\mu_{RS}(\ug)+ \frac12) + 2 \langle 
c_1(TW), A \rangle .
\end{equation}
We have proved in Lemma~\ref{lem:maslov} that 
$\mu_{RS}(\gamma)=\mu(\gamma)+\frac 12$, hence
$$
\mathrm{ind}(D_u) = \mathrm{ind}(\cD_u) + 2
= \mu(\og) - \mu(\ug) + 2 \langle c_1(TW), A \rangle + 1.
$$

Finally note that the evaluation maps $\oev, \uev$ are well-defined
and smooth on $\cB^A$. Hence their restrictions to the moduli spaces
are smooth as well.
\end{proof}

\subsection{Compactness for Morse-Bott trajectories} 
  \label{sec:compact}

\begin{definition} \label{defi:MBbroken}
  Let $H$, 
$\{f_\gamma\}$ and $J$ be fixed as above, and let $p\in 
\mathrm{Crit}(f_{\og})$, 
  $q\in \mathrm{Crit}(f_{\ug})$. The space 
$\widehat \cM^A(p,q;H,\{f_\gamma\},J)$ 
  of {\bf parametrized 
Morse-Bott broken trajectories} consists of tuples 
  $$
 \u=(c_m,u_m,c_{m-1},u_{m-1},\ldots,u_1,c_0)
  $$ 
  such that 
\begin{enum}
   \item $u_i\in \widehat \cM^{A_i}(S_{\gamma_i},S_{\gamma_{i-1}};H,J)$, 
$i=1,\ldots,m$ 
   with $\gamma_m=\og$, $\gamma_0=\ug$ and 
$A_1+\ldots+A_m=A$;
   \item $c_0:[-1,+\infty[\to S_{\gamma_0}$, 
$c_i:[-T_i/2,T_i/2]\to S_{\gamma_i}$, 
   $i=1,\ldots,m-1$ and 
$c_m:]-\infty,1]\to S_{\gamma_m}$ satisfy 
   $\dot c_i=\nabla 
f_{\gamma_i} \circ c_i$, $i=0,\ldots,m$;
   \item 
$\oev(u_i)=\uev(c_i)$, $\uev(u_i)=\oev(c_{i-1})$, $i=1,\ldots,m$ and 
$c_0(+\infty)=q$, $c_m(-\infty)=p$.
  \end{enum}
  The space 
$\cM^A(p,q;H,\{f_\gamma\},J)$ of {\bf unparametrized Morse-Bott 
broken trajectories} consists of equivalence classes 
$$
[\u]=(c_m,[u_m],c_{m-1},[u_{m-1}],\ldots,[u_1],c_0)
$$ 
such that 
$\u\in \widehat 
\cM^A(p,q;H,\{f_\gamma\},J)$.
\end{definition}

\begin{definition} \label{def:convubar}
 Let 
$$
 \u_k=(c_{m_k,k},u_{m_k,k},c_{m_k-1,k},\ldots,u_{1,k},c_{0,k})\in 
 \widehat \cM^A(p_k,q_k;H,\{f_\gamma\},J)
$$
 with $k=1,\ldots,\ell$, 
and satisfying $q_k=p_{k-1}$ for $k=2,\ldots,\ell$. 
 We denote 
$p:=p_\ell$, $q:=q_1$. 
 A sequence $v_n\in 
 \widehat 
\cM^A(\og_p,\ug_q;H_{\delta_n},J)$ with $\delta_n\to 0$, 
 $n\to 
\infty$ is said to {\bf converge} to $\overline\u:=(\u_\ell,\ldots,\u_1)$ 
if 
 there exist shifts $(s_{i,k}^n)\in\R$, $i=1,\ldots,m_k$ such 
that 
 $$
 v_n(\cdot + s_{i,k}^n\,,\cdot)\to u_{i,k}, \quad 
n\to\infty
 $$
 uniformly on compact sets in $\R\times S^1$. We write 
in this case $v_n\to \overline\u$.  

A sequence $[\widetilde v_n]\in \cM^A(\og_p,\ug_q;H_{\delta_n},J)$ 
with $\delta_n\to 0$, $n\to\infty$ is said to {\bf converge} to
$[\overline\u]\in \cM^A(p,q;H,\{f_\gamma\},J)$ 
if there exist representatives $v_n$ and $\overline
\v$ such that $v_n\to \overline\v$ (this condition is
obviously independent on the choice of representatives). We write in
this case $[\widetilde v_n]\to[\overline\u]$. 

We call $\overline \u$ a {\bf 
broken Floer trajectory with gradient fragments}. We call each of the 
$\u_k$'s a {\bf  Floer trajectory with gradient 
 fragments}. Each 
$\u_k$ is a {\bf level} of $\overline \u$ and 
 each $u_{i,k}$ is a 
{\bf sublevel} of $\u_k$. 
 \end{definition}

\begin{definition} 
\label{defi:stable}
  An element 
  $$
 \u=(c_{m},u_{m},c_{m-1},\ldots,u_{1},c_{0})\in \widehat 
 \cM^A(p,q;H,\{f_\gamma\},J)
  $$ 
  with $m\ge 1$ is {\bf stable} 
if each $u_i$, $i=1,\ldots,m$ is a nonconstant Floer trajectory and 
if each $c_i$, $i=1,\ldots,m-1$ defined on an interval of nonzero 
length is nonconstant. An element $\u=(c_0)\in 
\cM^A(p,q;H,\{f_\gamma\},J)$ is {\bf stable} 
  if $p\neq q$. A broken Floer trajectory with gradient fragments 
  $\overline \u=(\u_\ell,\ldots,\u_1)$ is {\bf stable} 
if each $\u_k$, $k=1,\ldots,\ell$ is stable. 
\end{definition}

\begin{remark} 
{\rm 
  A convergent sequence $v_n$ of nonconstant Floer trajectories 
  has a stable limit $\overline \u$ which is unique  
 up to shifts on the $c_{i,k}$ and $u_{i,k}$. }
\end{remark}

The proofs of the next two lemmas use
the asymptotic estimates proved
in the Appendix. The relevant notation is introduced at the beginning
of the Appendix, and we briefly recall it here for the reader's
convenience. For each $\gamma \in \cP(H)$ we choose coordinates
$(\vartheta, z) \in S^1 \times \R^{2n-1}$
parametrizing a tubular neighbourhood of $\gamma$, such that
$\vartheta \circ \gamma(\theta)
= \theta$ and $z \circ \gamma(\theta) = 0$. Given a smooth function
$f_\gamma : S_\gamma \to \R$, we denote by $\varphi^{f_\gamma}_s$ the
gradient flow of $f_\gamma$ with respect to the natural metric on $S^1$.

In a neighbourhood of $\gamma\in \cP(H)$ 
the Floer equation $\p_s u+J\p_\theta u-JX_H=0$ becomes $\p_s
Z+J\p_\theta Z+J\frac \p {\p\vartheta}-JX_H=0$, where 
$Z(s,\theta):=(\vartheta\circ u(s,\theta)-\theta, z\circ 
u(s,\theta))$. Since 
$X_H=\frac \p {\p\vartheta}$ on $\{z=0\}$ this can be rewritten as   
$
\p_s Z+ J\p_\theta Z + Sz=0
$
for some matrix-valued function $S=S(\vartheta,z)$. The matrix 
$S_\infty(\theta):=S(\theta,0)$\index{$S_\infty$, asymptotic matrix} 
is symmetric.
Let $A_\infty:H^k(S^1,\R^{2n})\to H^{k-1}(S^1,\R^{2n})$ be the operator 
defined by 
$
A_\infty Z:=J\frac d {d\theta} Z+ S_\infty(\theta)z.
\index{$A_\infty$, asymptotic operator}
$
The kernel of $A_\infty$ has dimension one and is spanned  
by the constant vector $e_1:=(1,0,\ldots,0)$. We denote by $Q_\infty$
\index{$Q_\infty$, asymptotic operator} 
the orthogonal projection onto $(\ker \, A_\infty)^\perp$ and we
set
\index{$P_\infty$, asymptotic operator}
$P_\infty := \one -Q_\infty$.

\begin{lemma}  \label{lem:cornershift}
 Let $v_n \in \widehat 
\cM^A(\og_p,\ug_q;H_{\delta_n},J)$ with $\delta_n\to 0$, 
 $n\to \infty$ and $s_1^n<s_2^n$ be shifts such that
$v_n(\cdot+s_1^n,\cdot)\to u_1$, $v_n(\cdot+s_2^n,\cdot)\to u_2$
uniformly on compact sets, with 
$u_1\in \widehat \cM^{A_1}(S_{\gamma_1},S_\gamma; H,J)$ and 
$u_2\in \widehat \cM^{A_2}(S_\gamma,S_{\gamma_2}; H,J)$. Any two
sequences of shifts $s_1^n < s_+^n < s_-^n < s_2^n$ satisfying 
$s_+^n-s_1^n\to\infty$, $s_2^n-s_-^n\to\infty$ and
\begin{equation} \label{eq:deltasmall}
\delta_n(s_+^n-s_1^n)\to 0, \qquad \delta_n(s_2^n-s_-^n)\to 0,
\end{equation}
have the property that 
$$
v_n(\cdot+s_+^n,\cdot) \to \uev(u_1), \qquad v_n(\cdot+s_-^n,\cdot)\to
\oev(u_2)
$$
uniformly on compact sets. 
\end{lemma}

\begin{proof} We claim that there exists $K>0$ such that
  $v_n([s_1^n+K,s_2^n-K]\times S^1)$ is contained in a given small
  neighbourhood of $S_\gamma$. If that was not the case, we could find
  a sequence $K_n\to \infty$ and a sequence
  $(s_n,\theta_n)\in[s_1^n+K_n,s_2^n-K_n]\times S^1$
  such that $\textrm{dist}(v_n(s_n,\theta_n),S_\gamma)$ is bounded
  away from zero. Up to a subsequence, $v_n(\cdot+s_n,\cdot)$ converges
  to some Floer trajectory $v$ which must be nonconstant. On the other hand, for any
  $s\in \R$ and for any $K>0$ we have, for $n$ large enough,  
$$
\cA_{H_{\delta_n}} \big(v_n(s+s_2^n - K,\cdot)\big)
\!\le\! 
\cA_{H_{\delta_n}} \big(v_n(s+s_n,\cdot)\big)
\!\le\! 
\cA_{H_{\delta_n}} \big(v_n(s+s_1^n +K,\cdot)\big), 
$$
and in the limit $\cA_H(u_2(s-K,\cdot)) \le \cA_H(v(s,\cdot)) \le
\cA_H(u_1(s+K,\cdot))$. We let $K$ go to infinity and obtain 
$\cA_H(\gamma) \le \cA_H(v(s,\cdot)) \le \cA_H(\gamma)$. 
This holds for all $s\in\R$ and therefore the cylinder $v$ is
constant over some element of $\cP(H)$, a contradiction which proves the
claim. 

By~\eqref{eq:pointwiseintervalQdelta} in the proof of
Proposition~\ref{prop:intervaldelta} applied to $v_n$ on
$[s_1^n+K,s_2^n-K]$ we get  
\begin{equation} \label{eq:normaldir}
|Q_\infty v_n(s,\theta)|\le C \max(\|Q_\infty v_n (s_1^n+K)\|,
\|Q_\infty v_n (s_2^n-K)\|).  
\end{equation} 

Let $\gamma_+$ be the limit in $S_{\gamma}$ of $v_n(s_+^n,\cdot)$, and
let $I_n(\epsilon):=[s_+^n(\epsilon),s_+^n]\subset [s_1^n+K,s_+^n]$ be 
the maximal subinterval containing $s_+^n$ such 
that $P_\infty v_n(s)$, $s\in I_n(\epsilon)$ 
is at distance at least $\epsilon$ from the critical points of
$f_{\gamma}$, except maybe $\gamma_+$ (if the latter is a
critical point). By the second part of Proposition~\ref{prop:intervaldelta}
applied to $v_n$ on $I_n(\epsilon)$ we obtain 
$$
|\vartheta\circ v_n(s,\theta)-\theta- \varphi_{\delta
s}^{f_{\gamma}}(\theta_0)|\le C \max(\|Q_\infty v_n(
s_+^n(\epsilon))\|, \|Q_\infty v_n(s_+^n)\|) e^{M\delta_n(s_+^n-s_1^n)}.
$$ 
Since $\delta_n(s_+^n-s_1^n)\to 0$ and
taking into account~\eqref{eq:normaldir} we get 
\begin{equation} \label{eq:estimateK}
|\vartheta\circ v_n(s,\theta)-\vartheta\circ \gamma_+(\theta)| \le
C_1 \max(\|Q_\infty v_n(s_1^n+K)\|, \|Q_\infty v_n(s_2^n-K)\|). 
\end{equation}
For $K$ large enough the right hand term becomes so small that 
the distance between $P_\infty v_n(s)$, $s\in I_n(\epsilon)$ and the
critical points of $f_{\gamma}$, except possibly $\gamma_+$,
is strictly bigger than $\epsilon$, hence $I_n(\epsilon)= [s_1^n
+K,s_+^n]$ by maximality (this holds for $K$ large enough).  
Applying~\eqref{eq:estimateK} to $s=s_1^n+K$ we obtain 
$$
|\vartheta\circ v_n(s_1^n+K,\theta)-\vartheta\circ
\gamma_+(\theta)| \le C_1 \max (\|Q_\infty v_n(s_1^n+K)\|, \|Q_\infty
v_n(s_2^n-K)\|).  
$$
Passing to the limit in the above inequality we obtain 
$$
|\vartheta\circ u_1(K,\theta)-\vartheta\circ
\gamma_+(\theta)| \le C_1 \max(\|Q_\infty u_1(K)\|, \|Q_\infty
u_2(-K)\|).  
$$
Letting $K\to \infty$ we obtain $\uev(u_1)=\gamma_+$. That this
implies uniform convergence on compact sets to the constant cylinder
over $\uev(u_1)$ can be seen in two ways: either one notices that the
above estimates hold uniformly when $s_+^n$ is replaced with
$s_+^n+K_+$, where $K_+$ is a bounded constant, or one uses the fact
that a Floer trajectory passing through a periodic orbit is
necessarily a constant cylinder, by unique continuation applied to the
infinite jet at that orbit~\cite[Theorem 2.3.2]{McDS}.

A similar argument proves the assertion involving $\oev(u_2)$.
\end{proof}

\begin{lemma} \label{lem:connect}
 Let $v_n\in \widehat \cM^A(p,q;H_{\delta_n},J)$ with $\delta_n\to 0$,
$n\to \infty$. Assume we are given two sequences of shifts $s_1^n <
s_2^n$ such that $v_n(\cdot+s_i^n,\cdot)$, $i=1,2$ converge uniformly on 
compact sets to constant cylinders $u_{\gamma_i}$ over orbits
$\gamma_i$ belonging to the same family $S_\gamma$. Then
there exists a (possibly broken) gradient trajectory of $f_\gamma$
starting at $\gamma_1$ and ending at $\gamma_2$. Moreover, the length
of the gradient trajectory is $T=\lim_{n\to \infty}
\delta_n(s_2^n-s_1^n)$. 
\end{lemma}

\begin{proof}
  We claim that, for $n$ large enough,
$v_n([s_1^n,s_2^n]\times S^1)$
is entirely contained in 
an arbitrarily small neighbourhood of $\gamma$. By
contradiction, if this fails we can reparametrize the sequence $v_n$
so that it converges to a nonconstant Floer trajectory $v\in
\cM^B(\og ',\ug ';H,J)$ for some class $B\in H_2(M;\Z)$ and some $\og
',\ug '\in \cP(H)$ such that their actions satisfy
$\cA_H(\gamma) \le \cA_H(\ug ') < \cA_H(\og ') \le
\cA_H(\gamma)$, which is impossible. 

We first assume that $\gamma_1$ is not a critical point of
$f_\gamma$. Let $\epsilon >0$ be fixed and denote by 
$I_n(\epsilon)=[s_1^n,s_2^n(\epsilon)]\subset [s_1^n,s_2^n]$ the
maximal subinterval containing $s_1^n$ such that 
the distance between $P_\infty v_n(s)$, $s\in I_n(\epsilon)$ 
and $\textrm{Crit}(f_\gamma)$ is at least $\epsilon$. 
We can apply Proposition~\ref{prop:intervaldelta} to $v_n$ and 
$I_n(\epsilon)$. In particular, for some sequence
$\theta_n\in S^1$ we have 
$$
\lim_{n\to\infty} 
\sup_{(s,\theta)\in I_n(\epsilon)\times S^1} 
|\vartheta\circ v_n(s,\theta) -\theta -
\varphi_s^{\delta_n f_\gamma} (\theta_n)| =0.
$$
Since $v_n(s_1^n,\cdot)$ converges to
$\gamma_1$, we also have 
\begin{equation} \label{eq:limsup}
\lim_{n\to\infty} 
\sup_{(s,\theta)\in I_n(\epsilon)\times S^1} 
|\vartheta\circ v_n(s,\theta) -\theta -
\varphi_{\delta_n (s-s_1^n)}^{f_\gamma} (\gamma_1)| =0.
\end{equation}
Modulo passing to a subsequence we know that
$v_n(s_2^n(\epsilon),\cdot)$ converges, which together
with~\eqref{eq:limsup} implies that $\delta_n(s_2^n(\epsilon)-s_1^n)$
converges to $T(\epsilon)\in \R_+$. This holds for each $\epsilon>0$
and, since $s_2^n(\epsilon) < s_2^n(\epsilon')$ if $\epsilon >
\epsilon '$, the limit $\lim_{\epsilon \to 0} T(\epsilon)=T\in
\overline \R_+$ exists. Then $\varphi_s^{f_\gamma}(\gamma_1)$, $s\in
[0,T]$ is a gradient trajectory starting at $\gamma_1$. 

If $T$ is
finite then this trajectory, and therefore $v_n(I_n(\epsilon) \times
S^1)$ stay at a fixed distance from $\textrm{Crit}(f_\gamma)$ for $n$
large enough. Hence $I_n(\epsilon)=I_n$ for $\epsilon$ sufficiently
small and we are done. 
If $T$ is infinite and the limit $\lim_{s\to \infty}
\varphi_s^{f_\gamma}(\gamma_1)$ is equal to $\gamma_2$, we are also
done. Otherwise we are in the next case, with shifts $\widetilde s_1^n
:= \lim_{\epsilon \to 0} s_2^n(\epsilon)$ and $\widetilde
s_2^n:=s_2^n$. 

We now assume that $\gamma_1$ is a critical point of $f_\gamma$ and
$\gamma_1\neq \gamma_2$. Given $\epsilon >0$ we denote by 
$I_n(\epsilon)=[s_1^n,s_2^n(\epsilon)]\subset [s_1^n,s_2^n]$ the
maximal subinterval containing $s_1^n$ such that 
the distance between $P_\infty v_n(s)$, $s\in I_n(\epsilon)$ 
and $\textrm{Crit}(f_\gamma)\setminus \{\gamma_1\}$ 
is at least $\epsilon$. For $\epsilon>0$ small enough the loops
$P_\infty v_n(s_2^n(\epsilon))$ are at a distance bigger than
$\epsilon$ from $\gamma_1$ and, up to a subsequence,
$v_n(s_2^n(\epsilon),\cdot)$ converges to some $\widetilde \gamma_2
\in S_\gamma$ which is not a critical point of $f_\gamma$. The same
argument as in the previous case applied ``backwards'' to the shifts
$s_1^n < s_2^n(\epsilon)$ produces a negative gradient trajectory
running from $\widetilde \gamma_2$ to some critical point $\widetilde
\gamma_1$. By definition of $I_n(\epsilon)$, we must have $\widetilde
\gamma_1=\gamma_1$ and we thus obtain a gradient trajectory from
$\gamma_1$ to $\widetilde \gamma_2$. We are now in the first case with
shifts $\widetilde s_1^n := s_2^n(\epsilon)$ and $\widetilde
s_2^n:=s_2^n$. 

We successively apply the above two cases in order to produce a broken
gradient trajectory from $\gamma_1$ to $\gamma_2$. This is a finite
process since a broken trajectory has a finite number of nonconstant 
fragments. 
\end{proof}

\begin{proposition} \label{prop:compact}
 Let 
$v_n\in 
 \widehat \cM^A(p,q;H_{\delta_n},J)$ with $\delta_n\to 0$, 
$n\to \infty$. There 
 exists a broken Floer trajectory with gradient 
fragments 
 $\overline \u$ and a subsequence (still denoted
 by 
$v_n$) such that $v_n\to \overline \u$.  
\end{proposition}

\begin{proof}
 The energy $\cE(v_n):=\cE_{J,H_{\delta_n}}(v_n)$ defined 
in~\eqref{eq:E} satisfies 
 $$
 \cE(v_n)=-\int_{S^1}H_{\delta_n}(\theta,\og_p(\theta))\,d\theta + 
 \int_{S^1}H_{\delta_n}(\theta,\ug_q(\theta))\,d\theta.
 $$ 
 Since $H_{\delta_n}\to H$ we infer that $\cE(v_n)$ is uniformly 
 bounded. 

 Floer's compactness theorem~\cite[Proposition~3c]{F} applies to 
our situation and provides a collection of Floer trajectories 
$u_i$, $i=1,\ldots,m$ for the pair $(H,J)$ together with 
holomorphic spheres attached to them, as well as shifts $(s_i^n)$ such
that $v_n(\cdot+s_i^n,\cdot)$ converges to $u_i$ 
and its associated holomorphic
spheres in the sense of nodal curves. Condition~\eqref{eq:asph}
implies symplectic asphericity $\langle \om,\pi_2(\widehat W)\rangle
=0$, therefore holomorphic spheres in $(\widehat W,J)$ are constant
and the shifted $v_n$ converge to $u_i$ uniformly on compact
sets.

 Because the action spectrum of $\partial W$ was assumed to be 
discrete and injective the trajectories $u_i$ connect with each 
other, in the sense that $\oev(u_i)$ and $\uev(u_{i+1})$ belong 
to the same family of trajectories $S_{\gamma_i}$,
$i=1,\ldots,m-1$. Moreover, $\oev(u_m)$ belongs to $S_{\og}$ and  
$\uev(u_1)$ belongs to $S_{\ug}$. 

By Lemma~\ref{lem:cornershift} there exist shifts $s_{i,\pm}^n$ 
such that $v_n(\cdot +
s_{i,+}^n,\cdot)$ converges to the constant cylinder over $\uev(u_i)$,
and $v_n(\cdot + s_{i,-}^n,\cdot)$ converges to the constant cylinder
over $\oev(u_i)$. 
 Applying Lemma~\ref{lem:connect} 
with shifts $s_{i,+}^n < s_{i-1,-}^n$, $i=2,\ldots,m$ 
and $n$ large enough, we obtain 
broken gradient trajectories $c_{i-1}$ starting at
$\uev(u_i)$ and ending at $\oev(u_{i-1})$. 
Let now $s_-^n$, $s_+^n$ be shifts such that 
$v_n(\cdot+s_-^n,\cdot)\to p$ and $v_n(\cdot+s_+^n,\cdot)\to q$. 
Applying Lemma~\ref{lem:connect} with shifts
$s_-^n < s_{m,-}^n$ and with shifts $s_{1,+}^n < s_+^n$ we obtain broken 
gradient
trajectories $c_m$ starting at $p$ and ending at $\oev(u_m)$ and $c_0$
starting at $\uev(u_1)$ and ending at $q$. 
Since all $S_\gamma$ are 
circles of periodic orbits, the broken gradient trajectories $c_i$,
$i=0,\ldots,m$ consist each of a single fragment. 

The construction of a
stable broken Floer trajectory with gradient fragments out of the data
$c_i$, $u_i$ is straightforward and goes as follows. The collection of
points of the form $\uev(u_{i+1})$, $\oev(u_i)$ which are critical
points of $f_{\gamma_i}$ determine a partition 
 $$
 (c_{m_\ell,\ell}, 
u_{m_\ell,\ell}, c_{m_\ell-1,\ell}, \ldots,c_{1,\ell},
u_{1,\ell},c_{0,\ell}),\ldots,
 (c_{m_1,1},u_{m_1,1},\ldots,u_{1,1},c_{0,1})
 $$
 of the ordered tuple $(c_m,u_m,\ldots,c_1,u_1,c_0)$. Note that the
$c_{m_k,k}$ and $c_{0,k}$ may either be missing or be constant and
exactly one of $c_{0,k}$ and $c_{m_{k-1},k-1}$ is missing. In such a
situation we set $c_{m_k,k}$ or $c_{0,k}$ to be a constant trajectory
at the relevant critical point, defined on a semi-infinite interval. 
\end{proof}

\subsection{Gluing for Morse-Bott moduli spaces} \label{sec:gluing}

We prove in this subsection the assertions (i-ii) of
Theorem~\ref{thm:degen}. The following notation was introduced
in the previous subsection. For $\gamma \in \cP(H)$ we choose coordinates
$(\vartheta, z) \in S^1 \times \R^{2n-1}$ 
parametrizing a tubular neighbourhood of $\gamma$, such that
$\vartheta \circ \gamma(\theta)
= \theta$ and $z \circ \gamma(\theta) = 0$. Given a smooth function
$f_\gamma : S_\gamma \to \R$, we denote by $\varphi^{f_\gamma}_s$ the
gradient flow of $f_\gamma$ with respect to the natural metric on $S^1$.
The orthogonal projection onto the $1$-dimensional kernel of the
asymptotic operator at $\gamma\in\cP(H)$ is denoted by
$P_\infty$, and we denote $Q_\infty:=\one - P_\infty$. 

Let $p > 2, d > 0$ and $\delta > 0$. Let
$\cB^A_\delta = \cB_\delta^{1,p,d}(\og_p,\ug_q,A; H, \{ f_\gamma \})$ 
\index{$\cB^A_\delta = \cB_\delta^{1,p,d}(\og_p,\ug_q,A; H, \{ f_\gamma \})$}
be the space of proper maps $u : \R \times S^1 \to \widehat W$ which
are locally in $W^{1,p}$ and satisfy
\begin{enum}
\item the map $u$ converges uniformly in $\theta$ as $s \to 
\pm\infty$ to $\ug_q$,
respectively $\og_p$, and represents the homology class $A\in H_2(W;\Z)$;

\item there exist tubular neighbourhoods $\oU$ and $\uU$ of $\og$ and $\ug$
respectively, parametrized by $(\vartheta,z)\in S^1 \times \R^{2n-1}$ such that
\begin{eqnarray*}
\vartheta \circ u(s,\theta) - \theta - \varphi^{\delta 
f_{\og}}_s(\otheta_0) &\in& \! \!
W^{1,p}(]-\infty, -s_0]
\times S^1, \R ; e^{d|s|} ds \, d\theta) , \\
z \circ u(s,\theta) &\in& \! \! W^{1,p}(]-\infty, -s_0] \times S^1,
\R^{2n-1} ; e^{d|s|} ds \, d\theta) , \\
\vartheta \circ u(s,\theta) - \theta - \varphi^{\delta
f_{\ug}}_s(\utheta_0) &\in& \! \! W^{1,p}([s_0,\infty[
\times S^1, \R ; e^{d|s|} ds \, d\theta) , \\
z \circ u(s,\theta) &\in& \! \! W^{1,p}([s_0,\infty[ \times S^1,
\R^{2n-1} ; e^{d|s|} ds \, d\theta) ,
\end{eqnarray*}
for some $s_0 > 0$ sufficiently large and
some $\otheta_0, \utheta_0 \in S^1$ satisfying
\begin{equation} \label{eq:astheta}
\lim _{s\to -\infty} \varphi_s^{f_\gamma}(\otheta_0) = p, \qquad
\lim _{s\to +\infty} \varphi_s^{f_\gamma}(\utheta_0) = q.
\end{equation}
\end{enum}
Then $\cB^A_\delta$ is a Banach manifold and, for $d > 0$ sufficiently
small, it contains
the moduli spaces $\cM^A(\og_p,\ug_q;H_\delta,J)$ for all $J \in \cJ$ (see
Proposition \ref{prop:asymptoticdelta} in the Appendix).
Let $\cE \to \cB^A_\delta$ be the Banach vector bundle with fiber
$\cE_{(u,J)} = L^p(\R \times S^1, u^* T\widehat W; e^{d|s|} ds \, d\theta)$.
Let $\dbar_{H_\delta,J} : \cB^A_\delta \to \cE$ be the section defined by
$$
\index{$\dbar_{H_\delta,J}$}
\dbar_{H_\delta,J}(u) := \partial_s u + J_\theta (\partial_\theta u -
X_{H_\delta}) .
$$

Then $\cM^A(\og_p,\ug_q;H_\delta,J) =
\dbar_{H_\delta,J}^{-1}(0)$. From now on
we fix $J \in \Jreg(H)$. In order to prove (i) in
Theorem~\ref{thm:degen} we need to show that the vertical
differential
$D_u : T_u \cB^A_\delta \to \cE_u$ defined by~(\ref{eq:D}) is
surjective for all
$u \in \cM^A(\og_p,\ug_q;H_\delta,J)$ when $\delta > 0$ is
sufficiently small and the expected dimension
of the moduli space is zero. We have
$$
T_u \cB^A_\delta = W^{1,p}(\R\times S^1,u^*T\widehat
W;e^{d|s|}ds\,d\theta) \oplus
\oV_u \oplus \uV_u,
$$
where $\oV_u$, $\uV_u$ are real vector spaces of dimension
$$
\dim\, \oV_u = \ind(p), \qquad \dim\, \uV_u = 1-\ind(q).
$$
When their dimension is nonzero $\oV_u$ and $\uV_u$ are respectively
generated by
two sections of $u^*T\widehat W$ of the form
$$
(1-\beta(s,\theta)) \nabla f_{\og}(\vartheta \circ u(s,0)) \quad
\mathrm{and} \quad
\beta(s,\theta)\nabla f_{\ug}(\vartheta \circ u(s,0)),
$$
with $\beta(s,\theta)=\beta(s)$ a smooth cutoff
function\index{$\beta$, cutoff function}
which vanishes
for $s\le 0$ and is equal to $1$ for $s\ge 1$. The fact that
$\oV_u$ and $\uV_u$ have varying dimensions is a consequence of
condition~(\ref{eq:astheta}).

We shall prove surjectivity of $D_u$ by showing that the elements of
the moduli space
$\cM^A(\og_p,\ug_q;H_\delta,J)$ can be approximated, for $\delta>0$
small enough,
by gluing the elements of  $\cM^A(S_{\og},S_{\ug};H,J)$ with fragments
of gradient trajectories of the Morse functions
$f_\gamma$.

Given $a,b \in \overline{\R}$, $a < b$ we define intervals
$$
I(a,b) = \left\{\begin{array}{ll}
[a, b] & \mathrm{if } \ a,b \in \R, \\
]-\infty,b] & \mathrm{if } \ a = -\infty, \ b \in \R , \\
{} [a,\infty[ & \mathrm{if } \ a \in \R, \ b = \infty .
\end{array}\right.
$$
For $b-a>4$ and $|\epsilon| < 1$, we let
$h_{a,b,\epsilon}:\R\to I(a,b) \subset \R$
\index{$h_{a,b,\epsilon}$} 
be a collection of smooth increasing functions such that
$h_{a,b,\epsilon}(s)=a$ if $s\le a-\epsilon/2$, $h_{a,b,\epsilon}(s)=b$ if $s\ge
b+\epsilon/2$ and $h_{a,b,\epsilon}(s):= s$
if $a-\epsilon/2+1<s<b+\epsilon/2-1$. 
 We can of course make the family
$\{h_{a,b,\epsilon}\}$ depend smoothly on 
$a$, $b$ and $\epsilon$. We define $k_{a,b,\epsilon}(s):=\frac d
{d\sigma} |_{\sigma=0} h'_{a-\sigma,b+\sigma,\epsilon}(s)
\index{$k_{a,b,\epsilon}$} 
$. The support of
$k_{a,b,\epsilon}$ is contained in $[a-\epsilon/2,a-\epsilon/2+1] \, \cup \,
[b+\epsilon/2-1,b+\epsilon/2]$. We may assume without loss of
generality that $h'_{a,b,\epsilon}$ and $k_{a,b,\epsilon}$ are
uniformly bounded. 

\medskip 

\noindent {\bf Convention.} If $\epsilon=0$ we shall omit it from all
subsequent decorations, and we set $\epsilon=0$ if $a=-\infty$ or
$b=+\infty$.   

\medskip 

\begin{figure}[ht]
         \begin{center}
\scalebox{0.9}{\input{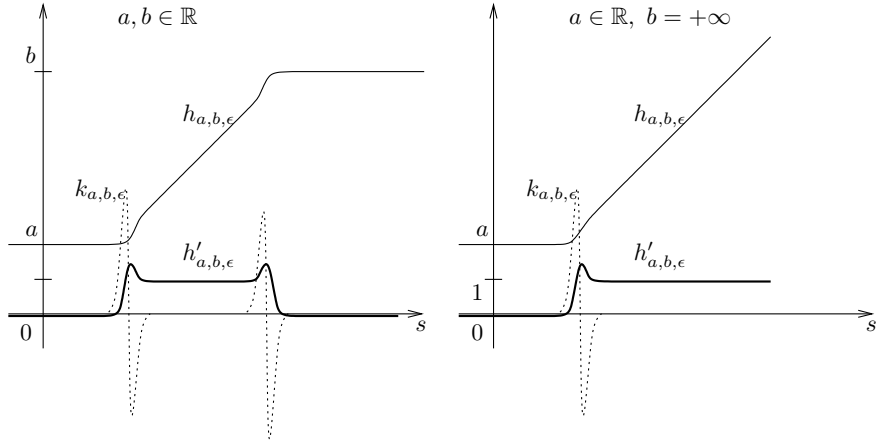}}
\caption{The reparametrization function $h_{a,b,\epsilon}$ 
and its derivatives. \label{fig:h}}
         \end{center}
\end{figure}

Let $\gamma \in \cP_\lambda$ and $c:I(a,b)\to S_\gamma \subset
\widehat W$ be a fragment 
of gradient trajectory for the function
$f_\gamma$, i.e. $\dot c = \nabla f_\gamma \circ c$.
We define the corresponding {\bf gradient cylinder}
$$
u_{\delta, \gamma, a, b,\epsilon} : \R \times S^1 \to S_\gamma \subset
\widehat W 
\index{$u_{\delta, \gamma, a, b,\epsilon}$, gradient cylinder}
$$
by the equation
\begin{equation} \label{eq:udelta}
\vartheta \circ u_{\delta, \gamma, a, b,\epsilon}(s,\theta)=
\vartheta \circ c (\delta h_{\frac a \delta,\frac b \delta,\frac
  \epsilon \delta}(s)) + \theta.
\end{equation}
Then $\displaystyle \lim_{s\to-\infty} \vartheta \circ u_{\delta, 
\gamma, a, b,\epsilon}(s,\theta) =
\vartheta \circ c (a) + \theta$ and
$\displaystyle \lim_{s\to+\infty} \vartheta \circ u_{\delta, \gamma, 
a, b,\epsilon}(s,\theta) =
\vartheta \circ c (b) + \theta$.

For $\gamma\in\cP_\lambda$
we define Banach manifolds \index{$\cB'_\delta$|(}
$\cB^{1,p,d}_\delta(S_\gamma,S_\gamma;f_\gamma)$,
$\cB^{1,p,d}_\delta(p,S_\gamma;f_\gamma)$, $p\in\mathrm{Crit}(f_\gamma)$
and $\cB^{1,p,d}_\delta(S_\gamma,q;f_\gamma)$,
$q\in\mathrm{Crit}(f_\gamma)$  consisting of maps
$u:\R\times S^1\to \widehat W$ which are locally of class $W^{1,p}$,
whose asymptotics are translates of $\gamma$, which represent the zero
homology class
and which
satisfy the following asymptotic conditions.
\begin{enum}
\item for $\cB^{1,p,d}_\delta(S_\gamma,S_\gamma;f_\gamma)$:     
there exists a neighbourhood $U$ of $S_\gamma$
together with a parametrization $(\vartheta,z) : U \to S^1 \times
\R^{2n-1}$ such that
\begin{eqnarray*}
\vartheta \circ u(s,\theta) - \theta -\otheta_0 &\in&
W^{1,p}(]-\infty, -s_0]
\times S^1, \R ; e^{d|s|} ds \, d\theta) , \\
z \circ u(s,\theta) &\in& W^{1,p}(]-\infty, -s_0] \times S^1,
\R^{2n-1} ; e^{d|s|} ds \, d\theta) , \\
\vartheta \circ u(s,\theta) - \theta - \utheta_0 &\in&
W^{1,p}([s_0,\infty[
\times S^1, \R ; e^{d|s|} ds \, d\theta) , \\
z \circ u(s,\theta) &\in& W^{1,p}([s_0,\infty[ \times S^1, \R^{2n-1} ;
e^{d|s|} ds \, d\theta) ,
\end{eqnarray*}
for some $\otheta_0, \utheta_0 \in S^1$ and some $s_0 > 0$. Moreover,
there exists $T>0$ such that
\begin{equation} \label{eq:T}
\varphi_T^{f_\gamma}(\otheta_0)=\utheta_0;
\end{equation}

\item for $\cB^{1,p,d}_\delta(p,S_\gamma;f_\gamma)$: there exists a
  neighbourhood $U$ of $S_\gamma$
parametrized by $(\vartheta,z)\in S^1 \times \R^{2n-1}$ such that
\begin{eqnarray*}
\vartheta \circ u(s,\theta) - \theta - \varphi_s^{\delta
f_\gamma}(\otheta_0)
&\in& \! \! W^{1,p}(]-\infty, -s_0]
\times S^1, \R ; e^{d|s|} ds \, d\theta) , \\
z \circ u(s,\theta) &\in& \! \! W^{1,p}(]-\infty, -s_0] \times S^1,
\R^{2n-1} ; e^{d|s|} ds \, d\theta) , \\
\vartheta \circ u(s,\theta) - \theta - \utheta_0 &\in& \! \!
W^{1,p}([s_0,\infty[
\times S^1, \R ; e^{d|s|} ds \, d\theta) , \\
z \circ u(s,\theta) &\in& \! \!
W^{1,p}([s_0,\infty[ \times S^1, \R^{2n-1} ; e^{d|s|} ds \, d\theta) ,
\end{eqnarray*}
for some $\otheta_0, \utheta_0 \in S^1$ such that
$\lim_{s\to-\infty}\varphi_s^{f_\gamma}(\otheta_0)=
\lim_{s\to-\infty}\varphi_s^{f_\gamma}(\utheta_0)=
p$ and some $s_0 > 0$;

\item for $\cB^{1,p,d}_\delta(S_\gamma,q;f_\gamma)$: there exists a
  neighbourhood $U$ of $S_\gamma$
parametrized by $(\vartheta,z)\in S^1 \times \R^{2n-1}$ such that
\begin{eqnarray*}
\vartheta \circ u(s,\theta) - \theta - \otheta_0
&\in& \! \! W^{1,p}(]-\infty, -s_0]
\times S^1, \R ; e^{d|s|} ds \, d\theta) , \\
z \circ u(s,\theta) &\in& \! \! W^{1,p}(]-\infty, -s_0] \times S^1,
\R^{2n-1} ; e^{d|s|} ds \, d\theta) , \\
\vartheta \circ u(s,\theta) - \theta - \varphi_s^{\delta f_\gamma}(\utheta_0)
&\in& \! \! W^{1,p}([s_0,\infty[
\times S^1, \R ; e^{d|s|} ds \, d\theta) , \\
z \circ u(s,\theta) &\in& \! \! W^{1,p}([s_0,\infty[ \times S^1,
\R^{2n-1} ; e^{d|s|} ds \, d\theta) ,
\end{eqnarray*}
for some $\otheta_0, \utheta_0 \in S^1$ such that
$\lim_{s\to\infty}\varphi_s^{f_\gamma}(\utheta_0)=
\lim_{s\to\infty}\varphi_s^{f_\gamma}(\otheta_0)=
q$ and some $s_0 > 0$.
\end{enum}
We will designate one of the above three spaces by $\cB'_\delta$. 
\index{$\cB'_\delta$|)} 
We define evaluation maps $\oev$ and $\uev$ on $\cB'_\delta$ by
$$
\oev(u)=\lim_{s\to-\infty} u(s,\cdot), \qquad
\uev(u)=\lim_{s\to+\infty} u(s,\cdot).
$$
Any map
$u = u_{\delta, \gamma, a, b,\epsilon}$ belongs to a suitable space
$\cB'_\delta$, depending on $a$, $b$ being finite or not.
The tangent space $T_u\cB'_\delta$ has a natural
decomposition
\begin{equation} \label{eq:TucB'delta} 
T_u\cB'_\delta = W^{1,p,d}(\R\times S^1,u^* T\widehat W) \oplus \oV'_u
\oplus \uV'_u, 
\end{equation} 
where $\oV'_u,\uV'_u$ are real vector spaces of dimensions
\begin{equation} \label{eq:dimoV'uV'}
\dim \oV'_u = \left\{\begin{array}{ll} 1& \mbox{if } a\in\R, \\
\ind(p) & \mbox{if } a=-\infty,
\end{array} \right.
\dim \uV'_u = \left\{\begin{array}{ll} 1& \mbox{if } b\in\R, \\
1-\ind(q) & \mbox{if } b=+\infty.
\end{array} \right.
\end{equation} 
When the dimensions are respectively nonzero the generators of
$\oV'_u,\uV'_u$ are sections given as follows.
\begin{enum}
\item for $\cB^{1,p,d}_\delta(S_\gamma,S_\gamma;f_\gamma)$ the sections are
$$
(1-\beta(s,\theta))X_H(\gamma(\theta+\otheta_0))
\mbox{ and }
\beta(s,\theta)X_H(\gamma(\theta+\utheta_0));
$$

\item for $\cB^{1,p,d}_\delta(p,S_\gamma;f_\gamma)$ the sections are
$$
(1-\beta(s,\theta))\nabla f_\gamma(\vartheta\circ u(s,0))
\mbox{ and }
\beta(s,\theta)X_H(\gamma(\theta+\utheta_0));
$$

\item for $\cB^{1,p,d}_\delta(S_\gamma,q;f_\gamma)$ the sections are
$$
(1-\beta(s,\theta))X_H(\gamma(\theta+\otheta_0))
\mbox{ and }
\beta(s,\theta) \nabla f_\gamma(\vartheta\circ u(s,0)).
$$
\end{enum}

We recall that $\beta(s,\theta)=\beta(s)$ is a smooth cutoff 
function\index{$\beta$, cutoff function} 
which
vanishes for $s\le0$ and is equal to $1$ for $s\ge 1$. 
The norm on $T_u\cB'_\delta$ is chosen such that
the norm of the above generators of $\oV'_u$, $\uV'_u$ is equal 
to $1$.
Let $\cE\to \cB'_\delta$ be the Banach vector bundle with fiber
$$
\cE_u = L^p(\R\times S^1,u^*T\widehat W; e^{d|s|}ds\,d\theta).
$$
We are interested in the family of sections
$\dbar_{a,b,\epsilon}:=\dbar_{H'_{a,b,\epsilon},J}:\cB'_\delta\to \cE$, with
\begin{eqnarray}
\index{$\dbar_{a,b,\epsilon}:=\dbar_{H'_{a,b,\epsilon},J}$}
\index{$H'_{a,b,\epsilon}$}
H'_{a,b,\epsilon} & = & H + h'_{\frac a \delta,\frac b \delta,\frac
  \epsilon \delta}(s) (H_\delta -H)
\nonumber \\
& = & H + \delta h'_{\frac a \delta,\frac b \delta,\frac \epsilon
  \delta}(s) \rho 
f_\gamma(\ell_\gamma \vartheta - \ell_\gamma \theta).  \label{eq:Hprime}
\end{eqnarray}
Here we use the definition~(\ref{eq:Hdelta}) of $H_\delta$. This is a
three-parameter family in case (i) and a two-parameter family in cases
(ii) and (iii). Its main feature is that
$$
\dbar_{a,b,\epsilon}(u_{\delta, \gamma, a, b,\epsilon})=0.
$$
We note that neither of the operators
$\dbar_{H,J}$ and $\dbar_{H_\delta,J}$ defines a section
$\cB'_\delta\to \cE$ if $a$ or $b$ is infinite.
The vertical differential
$D_u:= D_u^{a,b,\epsilon}: T_u\cB'_\delta\to \cE_u$ of each
of the sections $\dbar_{a,b,\epsilon}$
is given by formula~(\ref{eq:D}) and is
a Fredholm operator whose index has the following values (see
also~(\ref{eq:indMB})).
\begin{enum}
\item for $\cB^{1,p,d}_\delta(S_\gamma,S_\gamma;f_\gamma)$
$$
\mathrm{ind}(D_u) = (\mu_{RS}(\gamma)-\frac 12) -
(\mu_{RS}(\gamma)+\frac 12) +2=1,
$$
\item for $\cB^{1,p,d}_\delta(p,S_\gamma;f_\gamma)$
$$
\mathrm{ind}(D_u) = (\mu_{RS}(\gamma) - \frac 12) -
(\mu_{RS}(\gamma)+\frac 12) +\ind(p)+1= \ind(p),
$$
\item for $\cB^{1,p,d}_\delta(S_\gamma,q;f_\gamma)$
$$
\mathrm{ind}(D_u) = (\mu_{RS}(\gamma)-\frac 12) -
(\mu_{RS}(\gamma)+\frac 12) +1 + 1-\ind(q) =1-\ind(q).
$$
\end{enum}
In formulas (ii) and (iii) the asymptotics of the operator obtained 
by conjugation
with $e^{\frac d p |s|}$ do not depend on $\ind(p)$, $\ind(q)$
because, for $\delta$ small, the exponential weight $\frac d p$
overrides the contribution of the
perturbation $H_\delta-H$.

\begin{proposition}  \label{prop:Surj_udelta}
Let $u = u_{\delta,\gamma,a,b,\epsilon} \in \cB'_\delta$. 
The operator  
$$
D_u:W^{1,p}(\R\times S^1,u^*T\widehat
W;e^{d|s|})\oplus\oV_u'\oplus\uV_u'\to L^p(\R\times S^1,u^*T\widehat
W;e^{d|s|} dsd\theta) 
$$ 
is surjective for $\delta>0$ small enough. 
\end{proposition}

\begin{proof} 
 In order to compute $D_u$ we choose $\nabla$ to be the Levi-Civita
connection corresponding to a (split) metric given by
$(d\lambda+dt\wedge\lambda)(\cdot,J\cdot)$.
It is a general fact that the operator $D_u$ can be written in a
unitary trivialization of $u^*T\widehat W$ as
$$
(D_u\zeta) (s,\theta) = \p_s \zeta + J_0\p_\theta\zeta + 
S(s,\theta)\zeta(s,\theta),
$$
where $J_0$ is the standard complex structure on $\R^{2n}$ and $S$ is
asymptotically symmetric as $s\to\pm\infty$. We can choose the
trivialization so that $X_H$ and
$\p/\p t$ correspond to constant vectors in $\R^{2n}$. We denote
$\oS:=\lim_{s\to -\infty}S(s,\cdot)$. 
In this situation the matrix $S$ has the following properties:
\begin{enum}
\item $\|S(s,\theta) -
  \oS(\vartheta\circ
  u(s,\theta)-\vartheta\circ c(a))\|$ is
  bounded by a constant multiple of $\delta$. This is because, for
  $s\in \R$, the restriction of $u$ to $[s-1,s+1]\times S^1$ is
  $\delta$-close to the constant cylinder over the orbit
  $u(s,\cdot)\in S_\gamma$. 
\item the action of $S(s,\theta)$ on the (constant) vector of $\R^{2n}$ 
corresponding to
$X_H$ is multiplication by
$$
\delta k(s):=
\delta h'_{\frac a \delta,\frac b \delta,\frac \epsilon \delta}(s) 
f''_\gamma(\vartheta(u(s,\theta))-\theta) ,
$$
and this expression goes to zero with $\delta$.
\item the matrix $S(s,\theta)$ sends the subspace corresponding to $\xi$ to 
itself and sends
$\p/\p t$ on a multiple of the form $(E+\delta F(s,\theta))\p/\p t$, with
$E>0$ and $F$ a bounded function of $(s,\theta)$.
This follows from~(\ref{eq:D}), (\ref{eq:XHR}) and the fact that
$\nabla_{\p/\p t} R_\lambda=0$ and $\nabla_vR_\lambda \in\xi$,
$v\in\xi$;
\item there is a constant $C>0$ such that $\|S'(s,\theta)\|\le C\delta$ for
all $s\in\R$ and $\theta\in S^1$. This follows from~(\ref{eq:udelta})
due to the presence of the factor $\delta$ in front of the
reparametrization function $h_{\frac  a \delta,\frac b \delta,\frac
  \epsilon \delta}$. 
\end{enum}

We characterize now the kernel of $D_u$. We first show that each
$\zeta\in\ker\,D_u$ is a multiple of the (constant) vector
corresponding to $X_H$, or that its component $\zeta^\perp$ on the
orthogonal complement vanishes. Let $F(s)$ denote the self-adjoint
operator $J_0\p_\theta + S(s,\theta)$, so that $D_u=\p_s+F(s)$. If
$\zeta\in \ker\,D_u$ we have $(\p_s - F(s))(\p_s+F(s))\zeta=0$, i.e.
$$
\p_s^2 \zeta - F(s)^2\zeta + S'(s)\zeta = 0.
$$
By taking the scalar product in $L^2(S^1,\R^{2n})$
with $\zeta^\perp$ and using property (ii) for $S$ we get
$$
\langle \zeta^\perp, \p_s^2\zeta^\perp \rangle - \|F(s)\zeta^\perp\|
^2 + \langle \zeta^\perp, S'(s)\zeta^\perp\rangle =0.
$$
The Morse-Bott assumption and property (i) guarantee that
$\|F(s)\zeta^\perp\|_{L^2}\ge c\|\zeta^\perp\|_{L^2}$ for some $c>0$. We obtain
\begin{eqnarray*}
\p_s^2 \|\zeta^\perp\|^2_{L^2} & \ge & 2 \langle \zeta^\perp,
\p_s^2\zeta^\perp\rangle _{L^2} \ \ge \ 2(c^2-C\delta)
\|\zeta^\perp\|^2_{L^2} \ \ge \ c^2  \|\zeta^\perp\|^2_{L^2}
\end{eqnarray*}
if $\delta>0$ is sufficiently small. In particular
$\|\zeta^\perp\|^2_{L^2}$ can have no local maximum on $\R$. Since
$\|\zeta^\perp\|^2_{L^2}\to0$ as $s\to\pm\infty$ we deduce that
$\zeta^\perp\equiv 0$.

We now show that all elements of $\ker\,D_u$ are independent of $\theta$.
Let $\zeta\in\ker\,D_u$. Because $\zeta^\perp=0$ we have
$\p_s\zeta+J_0\p_\theta\zeta + \delta k(s)\zeta=0$, with
$\p_s\zeta+\delta k(s)\zeta$ and $\p_\theta\zeta$ pointwise colinear
with $X_H$. Hence
$\p_s\zeta+\delta k(s)\zeta=0$ and $\p_\theta\zeta=0$.

This shows that the elements of $\ker\,D_u$ also belong to the kernel
of the linearized
Morse operator
$$
\zeta\mapsto \p_s\zeta + \delta h'_{\frac a \delta,\frac b
  \delta,\frac \epsilon \delta}(s)
f''_\gamma(\vartheta\circ c(\delta h_{\frac a \delta,\frac b
  \delta,\frac \epsilon \delta}))\zeta.
$$
This is a differential equation on $\R$ for which the Cauchy problem
has a unique solution. Hence the space of solutions is one-dimensional
in $C^\infty(\R,\R)$ and, in order to determine the dimension of
$\ker\,D_u$, we just have to check whether the solutions belong or not
to its domain.

If $a$ and $b$ are finite the solutions are constant near $\pm\infty$,
hence belong to the domain of $D_u$ and $\dim\,\ker\,D_u =1$. If
$a=-\infty$ (and $b$ is finite)
we distinguish two cases: either $p$ is a maximum, in which case
$f''_\gamma(p)<0$, the solutions are unbounded near $-\infty$ and
$\ker\,D_u=0$, or $p$ is a minimum, in which case $f''_\gamma(p)>0$,
the solutions coincide near $-\infty$ with the elements of $\oV_u'$ and
$\dim\,\ker\,D_u =1$. Hence $\dim\,\ker\,D_u = \ind(p)$. A similar
argument
shows that $\dim\,\ker\,D_u = 1-\ind(q)$ if $b=+\infty$ (and $a$ is
finite). In all cases we have
$$
\dim\,\ker\,D_u = \ind(D_u),
$$
so that $D_u$ is surjective.
\end{proof} 

Up to a translation, the defining interval $I(a,b)$ of a gradient
cylinder can be considered to be $[-T/2,T/2]$, $T>0$ in case (i), or
$]-\infty,1]$, $[-1,\infty[$ in cases (ii) and (iii) respectively. 
We shall thus assume in the sequel that the parameters $a,b$ take the
values
$$
a=-T/2,\ b=T/2 \ \mbox{ for } T>0, \quad \mbox{ or } a=-\infty,\ b=1, \quad
\mbox{ or } a=-1,\ b=+\infty.
$$

We consider a tuple $(\gamma,a,b,\epsilon)$ 
and the gradient cylinder
$u:=u_\delta:=u_{\delta,\gamma,a,b,\epsilon}$ for $\delta$ small
enough. Let $(s_\delta)$ be a family of parameters such that $s_\delta
\le s_\delta^*$ and $\frac {s_\delta} {s^*_\delta}\to 1$ as $\delta
\to 0$, where 
$$
s^*_\delta := \left\{\begin{array}{ll} 
(T+\epsilon)/2\delta,& \mbox{ if } a=-T/2,\ b=T/2,\\
1/\delta,& \mbox{ otherwise}.
\end{array}\right.
$$
In particular we have $s_\delta\to\infty$ as $\delta\to 0$. Our goal
now is to define modified norms $\|\cdot\|_{1,\delta}$ and
$\|\cdot\|_\delta$ on the domain and target of the operators
$D_u=D_{u_\delta}$ such that they admit uniformly bounded right
inverses with respect to $\delta\to 0$. Let $w_\delta:\R\to\R^+$ be
the weight function defined by 
\begin{equation} \label{eq:wdelta} 
w_\delta(s)=\left\{\begin{array}{ll} 
e^{d||s|-s_\delta|}, & \mbox{ if } a \mbox { and } b \mbox{ are
finite}, \\ 
e^{d|s-s_\delta|}, & \mbox{ if } a=-\infty \mbox { and } b \mbox{ is
finite}, \\ 
e^{d|s+s_\delta|}, & \mbox{ if } a \mbox{ is
finite and } b=\infty.
\end{array}\right.
\end{equation}
\begin{figure}[ht]
         \begin{center}
\input{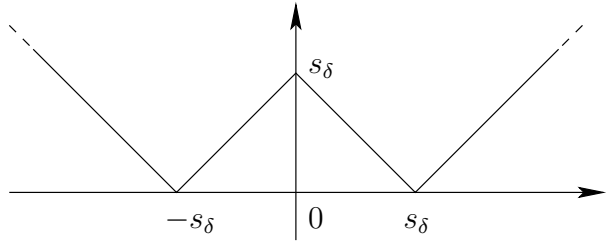}
\caption{Weight function $ ||s|-s_\delta|$ for $a,b$ finite
  (logarithmic scale). \label{fig:weight}}  
         \end{center}
\end{figure}

The new norm $\|\cdot\|_\delta$ on the target of $D_u$ is the
$L^p$-norm with weight $w_\delta$, and we emphasize it by writing the
target as 
$$
L^p(\R\times S^1,u^*T\widehat W;w_\delta(s)dsd\theta). 
$$

Let $\oV'_{u,\delta}$, $\uV'_{u,\delta}$ be vector spaces of the
same dimension as $\oV'_u$, $\uV'_u$, given by~\eqref{eq:dimoV'uV'},
and which, when their dimension is nonzero, are spanned by the
following sections.   

\begin{enum}
\item if $a,b$ are both finite the sections are
$$
(1-\beta(s+s_\delta,\theta))X_H(\gamma(\theta+\otheta_0))
\mbox{ and }
\beta(s-s_\delta,\theta)X_H(\gamma(\theta+\utheta_0));
$$

\item if $a=-\infty$ and $b$ is finite the sections are
$$
(1-\beta(s-s_\delta,\theta))\nabla f_\gamma(\vartheta\circ u(s,0))
\mbox{ and }
\beta(s-s_\delta,\theta)X_H(\gamma(\theta+\utheta_0));
$$

\item if $a$ is finite and $b=+\infty$ the sections are
$$
(1-\beta(s+s_\delta,\theta))X_H(\gamma(\theta+\otheta_0))
\mbox{ and }
\beta(s+s_\delta,\theta) \nabla f_\gamma(\vartheta\circ u(s,0)).
$$
\end{enum}

In case $a=-\infty$, $b$ finite or $a$ finite, $b=+\infty$ we define
the new norm $\|\cdot\|_{1,\delta}$ on the domain of $D_u$ by
splitting it as   
$$
\mathrm{dom}\, D_u = W^{1,p}(\R\times S^1,u^*T\widehat
W;w_\delta(s)dsd\theta)\oplus \oV'_{u,\delta} \oplus \uV'_{u,\delta}
$$
and setting the norm of the above generators of $\oV'_{u,\delta}$,
$\uV'_{u,\delta}$ to be equal to $1$. 

In case $a,b$ are finite we split the domain of $D_u$ as above and
further modify the weighted norm on the $W^{1,p}$-space. We recall
from the proof of Proposition~\ref{prop:Surj_udelta} that $\ker D_u$
is $1$-dimensional and is spanned by a section $\zeta_\delta$ which is
constant for $|s|\ge s^*_\delta$. We normalize $\zeta_\delta$ by
requiring that its value at $0$ be equal to the constant vector
corresponding to $X_H$. Let $\langle\cdot,\cdot\rangle$ be the
scalar product in $L^2(S^1)$. For an element $\zeta\in
W^{1,p}(\R\times S^1,u^*T\widehat W;w_\delta(s)ds d\theta)$ we denote  
$$
\kappa_\delta:=\frac {\langle
  \zeta(0,\cdot),\zeta_\delta(0,\cdot)\rangle}
{\langle \zeta_\delta(0,\cdot), \zeta_\delta(0,\cdot) \rangle }. 
$$
We denote 
$$
\chi_\delta(s,\theta):=
\beta(s+s_\delta)\beta(-s+s_\delta)\zeta_\delta(s,\theta) 
$$
and 
define
the norm $\|\cdot\|_{1,\delta}$ on $W^{1,p}(\R\times
S^1,u^*T\widehat W;w_\delta(s)dsd\theta)$ by 
$$
\|\zeta\|_{1,\delta}:=\|\zeta-\kappa_\delta \chi_\delta\|_{1,p,\delta}
+|\kappa_\delta|. 
$$
Here $\|\cdot\|_{1,p,\delta}$ is the weighted norm on $W^{1,p}(\R\times
S^1,u^*T\widehat W;w_\delta(s)dsd\theta)$.

\begin{proposition}  \label{prop:Surjectivity_udelta}
Let $u=u_\delta = u_{\delta,\gamma,a,b,\epsilon}$ as above. There
exists $\delta_2 \in ] 0, \delta_0]$ such that the operator 
$$
D_u: (\mathrm{dom}\, D_u,\|\cdot\|_{1,\delta}) \to
(L^p(\R\times S^1,u^*T\widehat W;w_\delta(s) dsd\theta),
\|\cdot\|_\delta) 
$$ 
is surjective and has a uniformly bounded right inverse
$Q_u=Q_{u_\delta}$ for $\delta \in ] 0, \delta_2 ]$.
\end{proposition}

\begin{proof}
We choose a unitary trivialization of $u^*T\widehat W$ as in the proof
of Proposition~\ref{prop:Surj_udelta}, so that $X_H$ and
$\partial/\partial t$ correspond to constant vectors in $\R^{2n}$, and
so that the operator $D_u$ takes the form 
$$
(D_u\zeta) (s,\theta) = \p_s \zeta + J_0\p_\theta\zeta + 
S(s,\theta)\zeta(s,\theta).
$$
Here $J_0$ is the standard complex structure on $\R^{2n}$ and $S$ is
asymptotically symmetric as $s\to\pm\infty$. The matrix $S(s,\theta)$
can be written as $S'(s,\theta)\oplus S''(s,\theta)$ with respect to the
splitting $\xi\oplus\langle \p/\p t, X_H\rangle$, so that
the operator $D_u$ is also split with respect to the decomposition
$\xi\oplus L$, where $L:=\langle \p/\p t, X_H\rangle$. It is therefore 
enough to find uniformly bounded right inverses for each of the
surjective operators 
$$
D'_u:W^{1,p}(\R\times S^1,u^*\xi;w_\delta(s)dsd\theta)\to
L^p(\R\times S^1,u^*\xi;w_\delta(s)dsd\theta),
$$
$$
D''_u: W^{1,p}(\R\times
S^1,u^*L;\|\cdot\|_{1,\delta})\oplus\oV'_{u,\delta}\oplus\uV'_{u,\delta}\to 
L^p(\R\times S^1,u^*L;w_\delta(s)dsd\theta).
$$
Here we use the fact that the norm
$\|\cdot\|_{1,\delta}$ coincides with the weighted $W^{1,p}$-norm on
sections with values in the subbundle $u^*\xi$. Note that $D'_u$ is an
isomorphism since it has index $0$, whereas
$\ind(D''_u)=\ind(D_u)$ is either $0$ or $1$.  

\medskip 

We treat $D''_u$ and consider first the case of a semi-infinite
gradient trajectory. The two possible cases are entirely similar, and
we assume without loss of generality that $a=-\infty$, $b=1$. Let 
$$
S''_0:=\left(\begin{array}{cc} E & 0 \\ 0 & 0 \end{array}\right),
$$
so that $\lim_{\delta\to 0}S''(s,\theta)=S''_0$ uniformly in
$(s,\theta)$. Consider the operator 
$$
D''_{0,\delta}: W^{1,p}(\R\times
S^1,u^*L;w_\delta(s)dsd\theta)\oplus\oV'_{u,\delta}\oplus\uV'_{u,\delta}\to  
L^p(\R\times S^1,u^*L;w_\delta(s)dsd\theta)
$$
defined by $D''_{0,\delta}:=\p_s+J_0\p_\theta +S''_0$. As in the proof of
Proposition~\ref{prop:Surj_udelta} one sees that
$D''_{0,\delta}$ is surjective, and we claim that it admits a right inverse
$Q''_{0,\delta}$ that is uniformly bounded with respect to
$\delta$. Indeed, let $Q''_0$ be a right inverse of $D''_0:=D''_{0,\delta=1}$
and consider the shift operators 
$$
(T_\delta\zeta)(s):=\zeta(s+s_\delta)
$$
acting from $\mathrm{dom}(D''_{0,\delta})\to\mathrm{dom}(D''_0)$ and
from $L^p(w_\delta(s) dsd\theta)\to L^p(e^{d|s|}dsd\theta)$. It follows
from the definitions of $\|\cdot\|_{1,\delta}$ and $\|\cdot\|_\delta$
that the operators $T_\delta$ are isometries, and we have
$D''_{0,\delta} =T_\delta^{-1} D''_0 T_\delta$ since $D''_0$ is
independent of  $s\in\R$. Hence $Q''_{0,\delta}=T_\delta^{-1}
Q''_0T_\delta$ is a right inverse for $D''_{0,\delta}$ such that 
$\|Q''_{0,\delta}\|=\|Q''_0\|$, and the claim is proved. 

Now, if $\delta$ is small enough we
have $\|S''(s,\theta)-S''_0\|\le 1/2\|Q''_0\|$, $s\in\R$
and therefore $\|D''_u - D''_{0,\delta}\|\le 1/2\|Q''_0\|$. This implies that
$$
\|D''_uQ''_{0,\delta} - \mathrm{Id}\| = \|D''_uQ''_{0,\delta} -
D''_{0,\delta}Q''_{0,\delta}\| \le \frac 12. 
$$
Thus $D''_uQ''_{0,\delta}$ is invertible and the norm of its inverse is $\le
2$. Finally a right inverse for $D''_u$ is given by $Q''_{0,\delta}
(D''_uQ''_{0,\delta})^{-1}$ and has norm $\le 2 \|Q''_0\|$.

\medskip 

We now treat the case $a=-T/2$, $b=T/2$ for $T>0$. Let
$\ou:=u_{\delta,\gamma,-(T+\epsilon)/2,0}$,
$\uu:=u_{\delta,\gamma,0,(T+\epsilon)/2}$ and  
$$
\oD'':=D''_\ou:W^{1,p}(\R\times S^1,\ou^*L;e^{d|s|}dsd\theta)\oplus 
\oV'_\ou \oplus\uV'_\ou \to L^p(e^{d|s|}dsd\theta),
$$
$$
\uD'':=D''_\uu:W^{1,p}(\R\times S^1,\uu^*L;e^{d|s|}dsd\theta)\oplus
\oV'_\uu \oplus\uV'_\uu \to L^p(e^{d|s|}dsd\theta).
$$
The same argument as above, using the constant operator $D''_0$, shows
that $\oD''$ and $\uD''$ admit right inverses which are uniformly
bounded with respect to $\delta\to 0$. Both operators have index $1$
and it follows from the description of their kernels given in the proof of
Proposition~\ref{prop:Surj_udelta} that their restrictions to 
$W^{1,p}\oplus\oV'_u$, respectively $W^{1,p}\oplus \uV'_u$ are
isomorphisms. We choose the right inverses $\oQ''$, $\uQ''$ to be the
inverses of their respective restrictions. 

Let $\uzeta\in\ker \oD''$, $\ozeta\in\ker \uD''$ be two sections
such that their values at $+\infty$ and respectively $-\infty$ are
equal to the (constant) vector corresponding to $X_H$ in the chosen
trivialization of $u^*T\widehat W$. Let $\uV'$, $\oV'$ be the
$1$-dimensional vector spaces spanned by $\beta\uzeta$ and
$(1-\beta)\ozeta$ respectively. Setting the norm of these generators
to be equal to $1$ defines a new norm on $\mathrm{dom}(\oD'')$ and
$\mathrm{dom}(\uD'')$, which we emphasize by decomposing the latter as 
$$
\mathrm{dom}(\oD'')=W^{1,p}(\R\times S^1,\ou^*L;e^{d|s|}dsd\theta)\oplus 
\oV'_u \oplus\uV',
$$
$$
\mathrm{dom}(\uD'')=D''_\uu:W^{1,p}(\R\times S^1,\uu^*L;e^{d|s|}dsd\theta)\oplus
\oV' \oplus\uV'_u.
$$
It follows from our special choice of the right inverses $\oQ''$,
$\uQ''$ that the latter are also uniformly bounded with respect to
this new norm as $\delta\to 0$. 

Let $D'':=\oD''\#_\delta \uD''$ be the operator obtained by gluing
$\oD''$ cut at $s_\delta$ and $\uD''$ cut at $-s_\delta$, with the
$\|\cdot\|_{1,\delta}$-norm on its domain and the
$\|\cdot\|_\delta$-norm on its target. It follows as
in~\cite[Proposition~5]{BM} that the right inverses $\oQ''$, 
$\uQ''$ give rise to a uniformly bounded right inverse $Q''$ for
$D''$ as $\delta\to 0$. On the other hand, we have $\|D''_u-D''\|\to 0$ as
$\delta\to 0$, and we obtain a uniformly bounded right inverse for
$D''_u$ by the previous formula 
$Q''_u:=Q''(D''_uQ'')^{-1}$. We note
that, upon gluing, the exponential 
weights at $\pm\infty$ for $\oD''$, $\uD''$ give rise to
the peak in the weight function $w_\delta$ for $D''$, and the
fibered sum operation on $\uV'$, $\oV'$, on which the norm is fixed,
is responsible for the appearance of the distinguished cutoff section
$\zeta_\delta$ leading to the modified norm $\|\cdot \|_{1,\delta}$. 

\medskip

We now treat $D'_u$ and start by making a few general remarks. For
each $s_0\in \R$ the operator 
$$
D'(s_0):=\p_s + J_0\p_\theta +S(s_0,\theta):W^{1,p}(\R\times
S^1,u^*\xi;dsd\theta)\to L^p(\R\times S^1,u^*\xi;dsd\theta)
$$
is $\delta$-close to the $\R$-invariant operator with nondegenerate 
asymptotics corresponding to the constant cylinder over the orbit $u(s_0, \cdot)$.
Hence, for $\delta > 0$ small enough, both operators are 
isomorphisms~\cite[Lemma~2.4]{Sa}. Moreover, this property also
holds in the presence of weights $e^{d|s|}$, $e^{ds}$ or $e^{-ds}$. 
For the weight $e^{d|s|}$ we argue as follows. The operator is
still Fredholm between the $W^{1,p}$ and $L^p$ spaces with weights, of
the same index $0$. Since the corresponding $W^{1,p}$ space is
contained in $W^{1,p}(\R\times S^1,u^*\xi;dsd\theta)$ we infer that
the operator is injective, hence an isomorphism. For the
weight $e^{ds}$ we argue as follows. Multiplication by 
$e^{\frac d p s}$ determines linear isomorphisms 
$M:W^{1,p}(\R\times S^1,u^*\xi;e^{ds}dsd\theta)\to 
W^{1,p}(\R\times S^1,u^*\xi;dsd\theta)$ and $M:L^p(\R\times
S^1,u^*\xi;e^{ds}dsd\theta)\to L^p(\R\times S^1,u^*\xi;dsd\theta)$.  
The operator $M^{-1}D'(s_0)M$ is an isomorphism and, for $\zeta\in
W^{1,p}(\R\times S^1,u^*\xi;e^{ds}dsd\theta)$, we have 
$$
M^{-1}D'(s_0)M\zeta = D'(s_0)\zeta+\frac d p \zeta.
$$
Since $d > 0$ is small as in Proposition \ref{prop:asymptoticdelta} 
and $p > 2$, the operator $M^{-1}D'(s_0)M$ is $\R$-invariant and
has nondegenerate asymptotics, hence is an 
isomorphism~\cite[Lemma~2.4]{Sa}. An analogous
reasoning using the multiplication by $e^{-\frac d p s}$ proves the
claim for the weight $e^{-ds}$.  

We now prove that $D'_u$ admits a uniformly bounded right inverse in
the case $a=-T/2$, $b=T/2$, $s_\delta=(T+\epsilon)/2\delta$. We recall
the notation 
$\ou:=u_{\delta,\gamma,-(T+\epsilon)/2,0}$,
$\uu:=u_{\delta,\gamma,0,(T+\epsilon)/2}$ and set  
$$
\oD':=D'_\ou:W^{1,p}(\R\times S^1,\ou^*\xi;e^{d|s|}dsd\theta)\to
L^p(\R\times S^1,\ou^*\xi;e^{d|s|}dsd\theta),
$$
$$
\uD':=D'_\uu:W^{1,p}(\R\times S^1,\uu^*\xi;e^{d|s|}dsd\theta)\to
L^p(\R\times S^1,\uu^*\xi;e^{d|s|}dsd\theta).
$$

We claim that each of the operators $\oD'$, $\uD'$ is an isomorphism
with uniformly bounded right inverse as $\delta\to 0$. We give the
proof for $\oD'$ since the proof for $\uD'$ is entirely analogous. We
choose a finite number of points $-\infty=s_{-m}<s_{-m+1}<\dots
<s_{-1}<0=s_0<s_1<\dots<s_{m+1}=+\infty$ such that
$\|S(s,\theta)-S(s',\theta)\|\le 1/4C$ for 
all $\theta\in S^1$ and $s,s'\in[s_i,s_{i+1}]$, $i=-m,\dots,m$, with
$C>0$ a constant to be chosen below. Let
$$
b_{i-1}:=a_i:=c^{-1}(u(s_i,0)), \qquad i=-m,\dots,m+1.
$$
We consider the operators
\begin{eqnarray*}
D'_i& :=& D'_{u_{\delta,\gamma,a_{i-1},b_{i-1}}}, \quad i=-m+1,\dots,-1,\\
D'_0&:=&D'_{u_{\delta,\gamma,a_{-1},b_0}},\\
D'_i&:=&D'_{u_{\delta,\gamma,a_i,b_i}}, \qquad \quad i=1,\dots,m.
\end{eqnarray*}
For each $i=-m+1,\dots,m$ we denote by $u_i=u_{i,\delta}$ the gradient
cylinder corresponding to the operator $D'_i$.
The domain and range of the operators $D'_i$ are as follows:
$$D'_i:W^{1,p}(\R\times S^1,u_i^*\xi;e^{-ds}dsd\theta)\to
L^p(\R\times S^1,u_i^*\xi;e^{-ds}dsd\theta), \quad i<0,
$$
$$
D'_0:W^{1,p}(\R\times S^1,u_0^*\xi;e^{d|s|}dsd\theta)\to
L^p(\R\times S^1,u_0^*\xi;e^{d|s|}dsd\theta),
$$
$$
D'_i:W^{1,p}(\R\times S^1,u_i^*\xi;e^{ds}dsd\theta)\to L^p(\R\times 
S^1,u_i^*\xi;e^{ds}dsd\theta), \quad i>0.
$$
We have seen that $D'(s_0)$ is an isomorphism for all $s_0\in \R$ if
one uses any of the weights $e^{d|s|}$, $e^{ds}$, $e^{-ds}$. Since 
$S(s_0,\cdot)$ belongs to a compact set of loops of matrices we infer
that the norm of the inverse $Q'(s_0):=D'(s_0)^{-1}$ is uniformly
bounded with respect to $s_0\in\R$ for each of these three weights. 
We choose
$C:=
\max_{\mathrm{weight}\in\{e^{d|s|},e^{ds},e^{-ds}\}}\max_{s_0\in\R}\|Q'(s_0)\|$.  

The same argument as for $D''_u$ shows that the inverse  
of each 
$D'_i$ is bounded by
$2C$ independently of $\delta$. We glue together the operators $D'_i$ into
$\widetilde D'$ by cutting at $a_i/\delta$ and $b_i/\delta$.
Then $\widetilde D'$ is still surjective and the norm of its inverse 
is bounded
by $2C\widetilde C^{2m-1}$, with $\widetilde C$ a universal constant
(see~\cite[Proposition~3.9]{Sa}). Note that our choice of weights for
the operators $D'_i$ is such that the resulting weight for the domain
and target of $\widetilde D'$ is still $e^{d|s|}$. On the other hand
we have 
$$
\|\widetilde D' - \oD'\|\to 0, \qquad \delta\to 0.
$$
This is because the two operators coincide outside $2m-1$ intervals of
length $2$, where the variation of $S$ tends to zero as $\delta\to
0$. As a consequence the inverse of $\oD'$ is also
uniformly bounded when $\delta$ is small enough.

We now glue the operator $\oD'$ cut at $s_\delta$ with the operator
$\uD'$ cut at $-s_\delta$, and denote the resulting operator by $D'$. 
The argument in~\cite[Proposition~5]{BM} shows that $D'$ admits
a uniformly bounded right inverse $Q'$, provided one uses the weight
$w_\delta(s)$ on its domain and target. On the other hand 
$$
\|D'_u-D'\|\to 0, \qquad \delta\to 0
$$
since the two operators differ on a segment of length $2$ where the
variation of $S$ goes to zero. We infer that
$D'_u$ also admits a uniformly bounded right inverse. 

The cases when $a=-\infty$, $b=1$ or $a=-1$, $b=\infty$ follow now
easily by combining the proof of the existence of uniformly bounded
right inverses for the operators $\oD'$ with the previous use of a
shift operator $(T_\delta\zeta)(s)=\zeta(s\pm s_\delta)$.
\end{proof}

\begin{remark}
 Note that, if $a=-T/2$, $b=T/2$, Our construction of a right inverse
 for $D''$ in the proof of Proposition~\ref{prop:Surjectivity_udelta} is 
 such that its norm is uniformly bounded as $\delta\to 0$ even if one 
 uses the ``non-compensated'' norm $\|\cdot\|_{1,p,\delta}$ instead of 
 $\|\cdot\|_{1,\delta}$. However, our choice of the norm $\|\cdot\|_{1,\delta}$
 will be essential in the proof of Proposition \ref{prop:deltageom}.
\end{remark} 

In order to describe the pregluing construction it is convenient to
work with a single section over $\cB'_\delta$ rather than with a
family of sections. We recall that,
up to a translation, the defining interval $I(a,b)$ of a gradient
cylinder can be considered to be $[-T/2,T/2]$, $T>0$ in case (i), or
$]-\infty,1]$, $[-1,\infty[$ in cases (ii) and (iii) respectively. We
are therefore led to consider the section
\begin{equation} \label{eq:dbarprime}
\dbar:\cB'_\delta\to \cE
\end{equation}
defined by
$\dbar:=\dbar_{-\infty,1}$ and $\dbar := \dbar_{-1,\infty}$ in cases
(ii) and (iii), and by
$$
\dbar (u) = \dbar_\epsilon(u) := \dbar _{-T_u/2,T_u/2,\epsilon}(u)
$$
in case (i). Here $T_u >0$ is the time needed to flow along the gradient
of $f_\gamma$ from the negative limit to the positive
limit of $u$ (see~\eqref{eq:T}).

\begin{remark} \label{rmk:dbarT}
  In case (i) the section $\dbar$ can be described as follows. The one
parameter family of sections $\dbar_T:=\dbar_{-T/2,T/2,\epsilon}$ gives rise to
a section denoted $\{\dbar_T\}$\index{$\{\dbar_T\}$} 
of the pull-back bundle
$\textrm{pr}_1^*\cE \to \cB'_\delta \times \R^+$.
There is a codimension one embedding
$\iota:\cB'_\delta\to\cB'_\delta \times \R^+$ given by $\iota(u)=
(u,T_u)$, the composition $\textrm{pr}_1\circ \iota$ is the
identity and we have
$$
\dbar = \{\dbar_T\}|_{\im\,\iota}.
$$
The situation is summarized in the following commutative diagram.
$$
\xymatrix
@C=40pt
@R=30pt@W=1pt@H=1pt
{
\cE \ar[r] \ar[d] & \textrm{pr}_1^*\cE \ar[r] \ar[d] & \cE \ar[d] \\
\cB'_\delta \ar@{^(->}[r]^{\hspace{-3mm}\iota}
\ar@/^1.3pc/ @{->}[u]^{\dbar}
&
\cB'_\delta\times \R^+ \ar[r]^{\quad \textrm{pr}_1}
\ar@/^1.3pc/ @{->}[u]^{\{\dbar_T\}}
& \cB'_\delta
}
$$
\end{remark}

Given $u\in\cB'_\delta$
we denote by $D'_u : T_u\cB'_\delta \to \cE_u$
the vertical differential of $\dbar$. In cases (ii) and (iii) we have
seen that $D'_u$ is a Fredholm operator of index $\ind(p)$ and
$1-\ind(q)$ respectively. In case (i) the vertical differential can be
computed explicitly as follows. The vertical differential of
$\{\dbar_T\}$, denoted by $D\{\dbar_T\}$, is
\begin{eqnarray*}
D\{\dbar_T\}(u,T)\cdot (\zeta,\tau) & = &
D_u^{-T/2,T/2,\epsilon}\zeta - \tau (JX_{H_\delta - H}) \frac d {dT}
h'_{-T/2\delta, T/2\delta,\epsilon/\delta} (s) \\
& = &  D_u^{-T/2,T/2,\epsilon}\zeta - \frac \tau {2\delta}
(JX_{H_\delta - H}) k_{-T/2\delta, T/2\delta,\epsilon/\delta}(s).
\end{eqnarray*}
Let us write a section $\zeta\in T_u\cB'_\delta$ as
$\zeta=\zeta^0 + a\ozeta + b\uzeta$, with $\zeta^0\in W^{1,p,d}$,
$a,b\in\R$ and $\ozeta$, $\uzeta$ being the distinguished
generators of $\oV'_u$, $\uV'_u$ respectively.
The vertical differential $D'_u$ acts by
\begin{eqnarray*}
D'_u\zeta & = & D\{\dbar_T\}(u,T_u)\cdot (\zeta,dT_u\cdot
\zeta) \\
& = & D_u^{-T_u/2,T_u/2,\epsilon}\zeta - \frac {dT_u\cdot\zeta} {2\delta}
(JX_{H_\delta - H}) k_{-T_u/2\delta, T_u/2\delta,\epsilon/\delta}(s).
\end{eqnarray*}
One can explicitly compute
$$
dT_u \cdot\zeta = dT_u\cdot(a\ozeta + b\uzeta) =
\frac {\dot c(-T_u/2) b - \dot c(T_u/2)a} {\dot c(-T_u/2) \cdot
\dot c(T_u/2)},
$$
where $c:\R\to S_\gamma$ is the gradient trajectory satisfying
$c(-T_u/2)=\otheta_0$, $c(T_u/2)=\utheta_0$ and $\dot c$ is the
derivative with respect to the $X_H$-parametrization of $S_\gamma$.

\begin{proposition} \label{prop:support_of_section}
   Let $T>0$ and
$u=u_{\delta,\gamma,-T/2,T/2,\epsilon}$. The index of $D'_u$ is equal to
$1$, its kernel has dimension $2$ and a
complement of $\im\,D'_u$ is spanned by a section supported in
$$
[-(T+\epsilon)/2\delta,-(T+\epsilon)/2\delta +1] \times S^1 \quad \bigcup \quad
[(T+\epsilon)/2\delta-1,(T+\epsilon)/2\delta] \times S^1.
$$
Morever, $D'_u$ admits a right inverse
defined on its image which is uniformly 
bounded with respect to $\delta\to0$.
\end{proposition}

\begin{proof}
  The first order differential operators
$D'_u$ and $D_u^{-T_u/2,T_u/2,\epsilon}$ differ by a term of order zero, hence
their indices are equal and $\ind\, D'_u=1$.

The operator $D\{\dbar_T\}(u,T_u)$ is surjective and has index
$2$. As a consequence
$\dim\ker D'_u \le \dim\ker D\{\dbar_T\}(u,T_u) = 2$.
Let $c:\R\to S_\gamma$ be the gradient curve
defining $u=u_{\delta,\gamma,-T/2,T/2,\epsilon}$. For
$\sigma$ close to zero we define $c^\sigma(s):=c(\sigma+s)$ and
denote by
$u_1^\sigma:=u_{\delta,\gamma,-T/2,T/2,\epsilon}^\sigma$ the gradient cylinder
defined by $c^\sigma$. Then
$\dbar(u_1^\sigma)=\dbar_T(u_1^\sigma)=0$, hence $\zeta^1:=\frac d
{d\sigma} |_{\sigma=0} u_1^\sigma \in \ker D'_u$. We also define
$u_2^\sigma := u_{\delta,\gamma,-(T+\sigma)/2,(T+\sigma)/2,\epsilon}$
to be the gradient cylinder associated to $c$. Then
$\dbar(u_2^\sigma)=\dbar_{T+\sigma}(u_2^\sigma)=0$, hence
$\zeta^2:=\frac d {d\sigma} |_{\sigma=0} u_2^\sigma \in \ker
D'_u$. Since $\zeta^1$ and $\zeta^2$ are linearly
independent, we infer that $\dim\ker D'_u=2$.

We claim that the section
$\eta:=\frac 1 {2\delta} (JX_{H_\delta - H}) k_{-T_u/2\delta,
T_u/2\delta,\epsilon/\delta}(s)$ spans a complement of $\im\,D'_u$.
This follows from~(i) in Lemma~\ref{lem:abstract} below with
$\ell:=dT_u$, $\phi:=D_u^{-T_u/2,T_u/2,\epsilon}$, $\widetilde \phi:=D'_u$,
$y:=\eta$ and $x_y:=\zeta^2$. That $D'_u$ admits a uniformly
bounded  
right inverse defined on its image follows from~(ii) in
Lemma~\ref{lem:abstract} and the fact that  
$D_u^{-T_u/2,T_u/2,\epsilon}$ has a uniformly bounded right inverse by 
Proposition~\ref{prop:Surjectivity_udelta}.   
\end{proof}

\begin{lemma} \label{lem:abstract} 
  Let $\phi:E\to F$ be a surjective map of Banach vector spaces,
$\ell:E\to \R$ be a nonzero linear functional, $y=\phi(x_y)\in F$
be fixed and $\widetilde \phi:E\to F$ be defined by
$$
\widetilde \phi (x)=\phi(x) - \ell(x)y.
$$
We assume that  $\ker \phi \subset \ker \ell$. 
Then $\im\,\widetilde \phi=\phi(\ker \ell)$ if and only if $\ell(x_y)=1$, in 
which case the following hold.
\begin{description} 
\item[(i)] The element $y$ spans a complement of $\im\,\widetilde\phi$.
\item[(ii)] If $Q:F\to E$ is a right inverse 
for $\phi$, then $Q|_{\phi(\ker\ell)}$ 
is a right inverse for $\widetilde \phi$ defined on its image.  
\end{description} 
\end{lemma}

\begin{proof}
  We first note that $\im\,\widetilde \phi \supseteq \phi(\ker \ell)$.
  Let us now assume that $\im\,\widetilde \phi=\phi(\ker\ell)$. For
$x\notin\ker \ell$ we obtain $\phi(x)-l(x)\phi(x_y)\in\phi(\ker\ell)$,
hence $x-l(x)x_y\in\ker\ell$, implying $\ell(x)-\ell(x)\ell(x_y)=0$
and $\ell(x_y)=1$. Conversely, if $\ell(x_y)=1$ we obtain
$x-\ell(x)x_y\in\ker\ell$ for any $x\in E$, hence $\widetilde \phi(x)
=\phi(x-\ell(x)x_y) \in \phi(\ker\ell)$.

   The element $y$ does not belong to $\phi(\ker\ell)$ because $y=\phi(x_y)$
with $\ell(x_y)=1$ and the preimage $x_y$ is well-defined up to an
element of $\ker\phi\subset\ker\ell$. This proves the equivalence in the 
statement of the Lemma, as well as~(i). 

To prove~(ii) we need to show that $Q(\phi(\ker\ell))\subset \ker\ell$. We prove the 
stronger statement $\im\,Q\cap \ker\ell=Q(\phi(\ker\ell))$. The inclusion 
$\im\,Q\cap \ker\ell\subset Q(\phi(\ker\ell))$ follows from the observation that, 
given $x=Qz$ with $\ell(x)=0$, we have $z=\phi(Qz)=\phi(x)\in \phi(\ker\ell)$. 
On the other hand note that $Q\phi$ is the projection to $\im\, Q$ along $\ker \phi$.
Since $\ker\phi \subset \ker \ell$, it follows that 
$Q\phi(\ker\ell) \subset \im \, Q \cap \ker \ell$.
\end{proof}

%
%
%
%

We describe now the pre-gluing construction for elements of the
Morse-Bott moduli spaces and gradient cylinders of the form
$u_{\delta,\gamma,a,b,\epsilon}$. We define the space
$$
\widetilde\cB_\delta:=
\widetilde\cB^{1,p,d}_\delta(\og_p,S_{\gamma_{m-1}},\ldots,S_{\gamma_1},
\ug_q,A;H,\{f_\gamma\})
\index{$\widetilde\cB_\delta$}
$$
consisting of tuples $\widetilde w:=(u_1,\ldots,u_m,v_0,\ldots,v_m)$
satisfying the following conditions.
\begin{enum}
\item $u_i \in\cB^{1,p,d}(S_{\gamma_i},S_{\gamma_{i-1}},A_i;H)$, 
$i=1,\ldots,m$,
with $S_{\gamma_0}:=S_{\ug}$, $S_{\gamma_m}:=S_{\og}$, 
$S_{\gamma_i} \neq S_{\gamma_{i-1}}$, $i = 1, \ldots, m$
and $A_1+\ldots+A_m=A$;
\item $v_0\in \cB^{1,p,d}_\delta(S_{\ug},q;f_{\ug})$,
$v_i\in\cB^{1,p,d}_\delta(S_{\gamma_i},S_{\gamma_i};f_{\gamma_i})$ for
$i=1,\ldots,m-1$, and
$v_m\in \cB^{1,p,d}_\delta(p,S_{\og};f_{\og})$;
\item $\oev(v_{i-1})=\uev(u_i)$ and $\uev(v_i)=\oev(u_i)$ for $i=1,\ldots,m$;
\item $\oev(v_0)$ belongs to the stable manifold of $q$, and $\uev(v_m)$
belongs to the unstable manifold of $p$. 
\end{enum}
By the definition of the spaces
$\cB^{1,p,d}_\delta(S_{\gamma_i},S_{\gamma_i};f_{\gamma_i})$ we have 
$\oev(v_i)\neq \uev(v_i)$ for $i=1,\dots,m-1$. We denote by $T_i>0$
the unique positive real number such that
$\varphi^{f_{\gamma_i}}_{T_i} (\oev(v_i))=\uev(v_i)$, where
$\varphi^{f_\gamma}_s$ is the gradient flow of $f_\gamma$.

\begin{figure}[ht]
         \begin{center}
\input{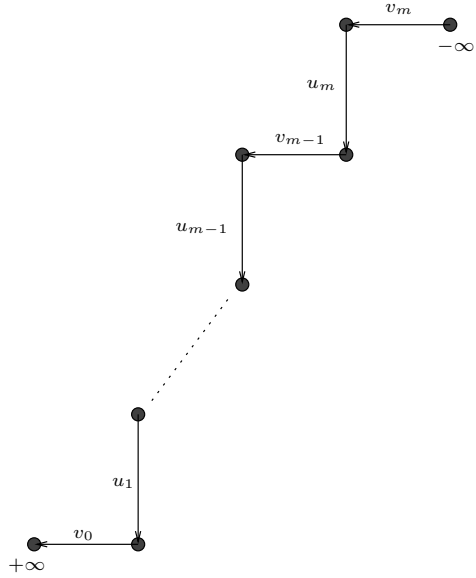}
\caption{Broken Morse-Bott trajectory $\widetilde w$. \label{fig:wtilde}}
         \end{center}
\end{figure}

Let us choose a tubular neighbourhood $U_\gamma\subset \widehat W$ for each
$\gamma\in\cP(H)$, parametrized by $(\vartheta,z)\in S^1 \times \R^{2n-1}$.
Given any subset
$$
\cK\subset \widetilde\cB^{1,p,d}_\delta
(\og_p,S_{\gamma_{m-1}},\ldots,S_{\gamma_1}, \ug_q,A;H,\{f_\gamma\})
$$
for which there exists $s_0>0$ such that, for $|s|\ge s_0$,
the components of any $\widetilde w \in \cK$ belong
to the respective tubular neighbourhoods of their asymptotics, we
construct, for $\delta>0$ small enough and $\epsilon_i\in\R$,
$i=1,\ldots,m-1$  small enough in absolute value, 
a pre-gluing map
$$
G_{\delta,\oeps}:\cK \to
\cB^{1,p,d}_\delta(\og_p,\ug_q,A;H,\{f_\gamma\}), \quad 
\oeps:=(\epsilon_1,\ldots,\epsilon_{m-1}). 
$$

Let $\beta:\R\to[0,1]$ be\index{$\beta$, cutoff function}  
a smooth increasing cutoff function vanishing on $]-\infty,0]$ and
identically equal to $1$ on $[1,\infty[$.
Define the gluing profile $R=R(\delta)$ by 
\begin{equation}
\index{$R=R(\delta)$, gluing profile}
R:=\frac 1 {2d} \ln\big(\frac 1 \delta \big).
\end{equation} 
We define for $i=1,\ldots,m$ the
maps $\widehat u_i:[-R,R]\times S^1\to \widehat W$ by
$$
\widehat u_i(s,\theta)\! :=  \! \left\{\begin{array}{lr}
\!\left\{\begin{array}{l}
  \!\! z(s,\theta) = \beta(s+R) z\circ u_i(s,\theta), \\
  \!\! \vartheta(s,\theta)  =
  \theta + \beta(s+R)(\vartheta\circ u_i(s,\theta) - \theta),
\end{array}\right.
  & \hspace{-7mm} s\in[-R,-R+1], \\
 \! u_i(s,\theta), & \hspace{-11mm} s\in[-R+1,R-1], \\
  \!\!\left\{\begin{array}{l}
  \!\! z(s,\theta) =  \beta(-s+R) z\circ u_i(s,\theta), \\
\! \vartheta(s,\theta) =  \theta + \beta(-s+R)(\vartheta\circ
u_i(s,\theta) - \theta),
\end{array}\right.
  & \hspace{-1mm} s\in[R-1,R].
\end{array}\right.
$$
We define for $i=1,\ldots,m-1$
the maps
$$
\widehat v_i:[- (T_i+\epsilon_i) / 2\delta,(T_i+\epsilon_i)
/2\delta]\times S^1\to \widehat W 
$$
by the analogous formulas in which we replace $R$ by $\frac {T_i+\epsilon_i} {2\delta}$.
We also define
$$
\widehat v_0:[-1/\delta,+\infty[ \times S^1 \to \widehat W
$$
by
$$
\widehat v_0(s,\theta):=  \left\{\begin{array}{lr}
\left\{\begin{array}{l}
  z(s,\theta) = \beta(s+\frac 1 \delta) z\circ v_0(s,\theta), \\
  \vartheta(s,\theta)  =
  \theta \!+\! \beta(s\!+\!\frac 1 \delta)(\vartheta\!\circ\! v_0(s,\theta) \!-\! \theta),
\end{array}\right.
  & \hspace{-5mm} s\in[-\frac 1 \delta,-\frac 1 \delta+1], \\
  v_0(s,\theta), & \hspace{-11mm} s\in[-\frac 1 \delta + 1, +\infty[,
\end{array}\right.
$$
as well as
$$
\widehat v_m:]-\infty,1/\delta] \times S^1 \to \widehat W
$$
by the analogous formula with $s$ replaced by $-s$ and $v_0$ replaced 
by $v_m$. 
Finally, we define
$$
G_{\delta,\oeps}(\widetilde w)
\index{$G_{\delta,\oeps}(\widetilde w)$, pre-glued curve}
$$
as the catenation
$\widehat v_m,\widehat u_m,\widehat v_{m-1},
\ldots,\widehat u_1, \widehat v_0$.
The catenation of these maps is performed in the above order and with
(obvious) shifts
$$
0=s_{v_m} < s_{u_m} < s_{v_{m-1}} < \ldots < s_{u_1} < s_{v_0}
$$
in the domain defined by 
\begin{eqnarray} \label{eq:shifts} 
s_{u_j}&=& s_{v_j} + \ell_j, \\
s_{v_{j-1}} &=& s_{u_j} + \ell_{j-1} \nonumber 
\end{eqnarray} 
for $j=1,\dots,m$. Here we denote 
\begin{equation} \label{eq:elli}
\ell_i:=R+(T_i+\eps_i)/2\delta
\end{equation} 
for $i=0,\dots,m$, with the convention $T_m=T_0=2$ and
$\eps_m=\eps_0=0$. We have in particular 
$$
\widehat v_i(s,\theta) =G_{\delta,\oeps}(\widetilde w)(s+s_{v_i},\theta), \quad
(s,\theta)\in \mathrm{dom}(\widehat v_i), \quad i=0,\dots,m,
$$
$$
\widehat u_j(s,\theta) =G_{\delta,\oeps}(\widetilde w)(s+s_{u_j},\theta), \quad
(s,\theta)\in \mathrm{dom}(\widehat u_j),\quad j=1,\dots,m.
$$
Given $\u=(c_m,u_m,\ldots,u_1,c_0)\in
\widehat \cM^A(p,q;H,\{f_\gamma\},J)$, we denote by 
$$
G_{\delta,\oeps}(\u)
$$ 
the element $G_{\delta,\oeps}(\tw)\in\cB_\delta$, where
$\tw:=(v_m,u_m,\ldots,u_1,v_0)$ and
$v_i:=u_{\delta,\gamma_i,a_i,b_i,\epsilon_i}$, $i=0,\ldots,m$ is the
gradient cylinder corresponding to the gradient trajectory
$c_i:I(a_i,b_i)\to S_{\gamma_i}$. 

The section $\dbar_{H_\delta,J}(G_{\delta,\oeps}(\widetilde w))$ belongs to the
space
$$
L^p(\R\times S^1,G_{\delta,\oeps}(\widetilde w)^*T\widehat W;
g_{\delta,\oeps}(s)dsd\theta), 
$$
where the continuous function 
$g_{\delta,\oeps}(s)$\index{$g_{\delta,\oeps}(s)$, weight function for
gluing} 
is the catenation of the following functions:
\begin{enum}
\item $g_{\delta,u_i}(s):=e^{d|s|}$ 
on the domain $[-R,R]$ of $\widehat u_i$;
\item
$g_{\delta,\epsilon_i,v_i}(s) = e^{d\left| |s|- s_{i,\delta} \right|}$ 
on the domain $[- (T_i+\epsilon_i)/{2\delta},(T_i+\epsilon_i) /{2\delta}]$
of $\widehat v_i$, where $s_{i,\delta} = \frac{T_i+\epsilon_i}{2\delta} - R
\le s^*_{i,\delta} = \frac{T_i+\epsilon_i}{2\delta}$, $i=1,\ldots,m-1$;
\item $g_{\delta,v_0}(s):=e^{d|s+s_{0,\delta}|}$
on the domain $[-1/\delta,+\infty[$ of $\widehat v_0$,
where $s_{0,\delta} = 1/\delta - R \le s^*_{0,\delta} = 1/\delta$;
\item $g_{\delta,v_m}(s):=e^{d|s-s_{m,\delta}|}$
on the domain $]-\infty, 1/\delta]$ of $\widehat v_m$,
with $s_{m,\delta} = 1/\delta - R \le s^*_{m,\delta} = 1/\delta$. 
\end{enum}
We denote the norm on the above $L^p$ space with weight
$g_{\delta,\oeps}$ by $\|\cdot\|_\delta$, omitting in the notation the 
dependence on the numbers $T_i+\epsilon_i$, $i=1,\ldots,m-1$. 
We 
define
a norm $\|\cdot\|_{1,\delta}$ on the space 
$$
W^{1,p}(\R\times S^1,G_{\delta,\oeps}(\widetilde w)^*T\widehat W;
g_{\delta,\oeps}(s)dsd\theta)
$$
as follows.  For $j=1\dots,m$ let
\begin{equation} \label{eq:kappa} 
\okappa_j = \frac{\langle \zeta(s_{u_j}-R, \cdot), X_H \rangle}
{\langle X_H, X_H \rangle}, \qquad 
\ukappa_j = \frac{\langle \zeta(s_{u_j}+R, \cdot),
  X_H \rangle}
{\langle X_H, X_H \rangle},
\end{equation} 
where $\langle \cdot, \cdot \rangle$ is the inner product in
$L^2(S^1)$. Here $s_{u_j}-R$ and $s_{u_j}+R$ are the coordinates of
the catenation circles between $\widehat u_j$ 
and $\widehat v_j$, respectively $\widehat u_j$ and $\widehat
v_{j-1}$. For $i = 1, \ldots, m-1$ let
$$
\kappa_i = \frac{\langle \zeta(s_{v_i}, \cdot),
\zeta_{i,\delta}(0,\cdot) \rangle} 
{\langle  \zeta_{i,\delta}(0,\cdot),  \zeta_{i,\delta}(0,\cdot) \rangle},
$$
where the section $\zeta_{i,\delta}$ generates the kernel of the operator
$D_{v_i}$ as in Proposition \ref{prop:Surj_udelta}.
The norm $\|\cdot\|_{1,\delta}$ is then defined by
\begin{eqnarray}
\lefteqn{\| \zeta \|_{1,\delta}:=} \nonumber \\ 
\! && \! \big\| \zeta - \sum_{j=1}^m 
\okappa_j
\beta(-s+s_{u_j}) \beta(s-s_{u_j}+2R)  X_H \label{eq:norm1delta} \\
&& \quad \qquad - \ \ukappa_j
\beta(s-s_{u_j}) \beta(-s+s_{u_j}+ 2R)  X_H  \nonumber \\
&& -\sum_{i=1}^{m-1} \kappa_i
\beta(s-s_{v_i}+\ell_i-2R) \beta(-s + s_{v_i} + \ell_i -2R)
\zeta_{i,\delta}(\cdot-s_{v_i},\cdot)
\big\|_{W^{1,p}(g_{\delta,\oeps})} \nonumber \\
&& + \sum_{j=1}^m \big(|\okappa_j | + | \ukappa_j |\big) + 
\sum_{i=1}^{m-1} | \kappa_i | . \nonumber 
\end{eqnarray}
Here $\ell_j$ is defined by~\eqref{eq:elli}, 
$\beta:\R\to [0,1]$ is\index{$\beta$, cutoff function} 
the smooth cutoff function which vanishes
on $]-\infty,0]$ and is equal to $1$ on $[1,\infty[$, and
$\|\cdot\|_{W^{1,p}(g_{\delta,\oeps})}$ is the $W^{1,p}$-norm with
weight $g_{\delta,\oeps}$ on $W^{1,p}(e^{d|s|}dsd\theta)$. The graph
of the function 
\begin{eqnarray*}
\lefteqn{\beta(-s+s_{u_j}) \beta(s-s_{u_j}+2R) +
\beta(s-s_{u_j}) \beta(-s+s_{u_j}+2R) } \\
&& + \beta(s-s_{v_j}+\ell_j-2R) \beta(-s + s_{v_j} + \ell_j -2R) \\
&& + \beta(s-s_{v_{j-1}}+\ell_{j-1}-2R) \beta(-s + s_{v_{j-1}} +
\ell_{j-1} -2R) 
\end{eqnarray*}
is depicted in Figure~\ref{fig:norm1delta}.   

\begin{remark} \label{rmk:norm1delta} \rm The definition of
  $\|\cdot\|_{1,\delta}$ is such that the norm of the gluing map $G$
  constructed in the proof of Proposition~\ref{prop:gluing} below is
  uniformly bounded with respect to $\delta\to 0$.  
\end{remark} 

\begin{figure}
         \begin{center}
\scalebox{1}{\input{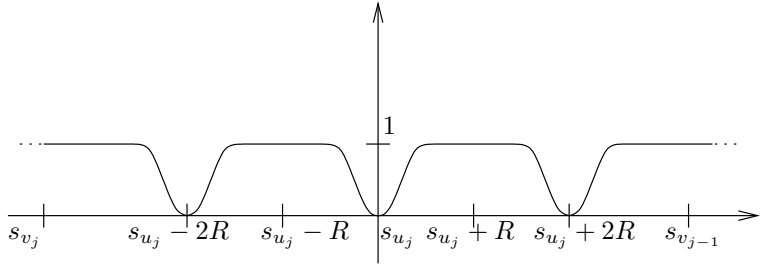}}
\caption{The definition of $\|\cdot\|_{1,\delta}$. \label{fig:norm1delta}}
         \end{center}
\end{figure}

\begin{proposition} \label{prop:catenation}
Let $\widetilde w\in \widetilde \cB_\delta$ and 
$\oeps(\delta):=(\epsilon_1(\delta),\ldots,\epsilon_{m-1}(\delta))$
be such that 
\begin{enum}
\item $\epsilon_i(\delta)\to 0$, $\delta\to 0$ for $i=1,\ldots,m-1$;
\item $u_i\in \cM^{A_i}(S_{\gamma_i},S_{\gamma_{i-1}};H,J)$, $i=1,\ldots,m$;
\item the components $v_i$ are of the form $u_{\delta,
\gamma_i,a_i,b_i,\epsilon_i}$, with $b_i=-a_i=T_i/2$ for $i=1,\ldots,m-1$, 
$b_0=+\infty$, $a_0=-1$, $\epsilon_0=0$ and $b_m=1$, $a_m=-\infty$,
$\epsilon_m=0$. 
\end{enum}
Then
$$
\lim_{\delta\to 0} \ \|\dbar_{H_\delta,J}(G_{\delta,\oeps}(\widetilde w))\|_\delta=0.
$$
\end{proposition}

\begin{proof}
We must check that $\|\dbar_{H_\delta,J}(G_{\delta,\oeps}(\widetilde
w))|_{I\times S^1}\|_\delta\to 0$ as $\delta\to 0$ when $I\subset \R$
is an interval of the following type.
\begin{enum}
\item $I=[-R+1,R-1]$ is contained in the domain of $\widehat u_i$. Then
$\dbar_{H_\delta,J}(\widehat u_i)=-J(X_{H_\delta}-X_H)\circ \widehat
u_i$. The norm of this map is pointwise bounded by a constant multiple
of $\delta$. Hence its $\delta$-norm is bounded by a constant multiple
of $\delta e^{dR} \to 0$, $\delta\to 0$;
\item $I=[-R,-R+1]$ or $I=[R-1,R]$ is contained in the domain of 
$\widehat u_i$. We have
$\dbar_{H_\delta,J}(\widehat u_i) = \dbar_{H,J}(\widehat u_i)
-J(X_{H_\delta}-X_H)\circ \widehat u_i$. The second term is bounded as
in (i).
The term $\dbar_{H,J}(\widehat u_i)$ is pointwise bounded by
the norms of $z\circ \widehat u_i$, $\vartheta\circ \widehat u_i -
\theta$ and of their derivatives. By Proposition~\ref{prop:asymptotic}
their $\delta$-norm is bounded
by a constant multiple of $e^{(d-r)R}\to 0$, $\delta\to 0$;
\item $I=[- (T_i+\epsilon_i) / {2\delta}+1,(T_i+\epsilon_i)
  /{2\delta} -1]$ for $i=1,\ldots m-1$, or
$I=[-1/\delta+1,+\infty[$ or $I=]-\infty,1/\delta-1]$ and is contained
in the domain of some $\widehat v_i$. Since  
  $\dbar_{H'_{-T_i/2,T_i/2,\epsilon_i},J}(\widehat v_i) =0$ 
  and $H'_{-T_i/2,T_i/2,\epsilon_i} = H_\delta$ for $s \in I$,
  we already have $\|\dbar_{H_\delta,J}(G_{\delta,\oeps}(\widetilde
w))|_{I\times S^1}\|_\delta = 0$;  
\item $I=[- (T_i+\epsilon_i) / 2\delta,-(T_i+\epsilon_i)
  /2\delta +1]$ or $I=[(T_i+\epsilon_i) / 2\delta -1,(T_i+\epsilon_i) /2\delta]$
for $i=1,\ldots m-1$, or $I=[-1/\delta, -1/\delta+1]$, or
$I=[1/\delta-1, 1/\delta]$ and is contained in the domain of some
$\widehat v_i$. Then $\dbar_{H_\delta,J}(\widehat v_i)$ involves only
$\vartheta\circ \widehat v_i - \theta$, its derivative with respect to
$s$ and $\delta\nabla f_{\gamma_i}$. By formula~(\ref{eq:udelta}) the
norm of these expressions is pointwise bounded by a constant multiple
of $\delta$, therefore their $\delta$-norms are bounded by $\delta
  e^{dR}\to 0$ as $\delta\to 0$. 
\end{enum}
\end{proof}


\begin{proposition} 
\label{prop:deltageom}
 Let $[\widetilde v_n] \in \cM^A(\og_p,\ug_q;H_{\delta_n},J)$ with
 $\delta_n\to 0$, $n\to \infty$ and let $[\u]\in
 \cM^A(p,q;H,\{f_\gamma\},J)$ be a broken Floer trajectory
 of level $\ell=1$ whose intermediate gradient fragments
 $c_1,\ldots,c_{m-1}$ are nonconstant. Then $[\widetilde v_n]\to[\u]$,
 $n\to\infty$ if and only if there exist 
\begin{itemize}
\item representatives $ v_n\in
 [\widetilde v_n]$, $\v\in[\u]$, 
\item real parameters
 $\oeps^n=(\epsilon_1^n,\ldots,\epsilon_{m-1}^n)$ with
 $\epsilon_i^n\to 0$, $n\to \infty$, 
\item vector fields 
$\zeta_n\in T_{G_{\delta_n,\oeps^n}(\v)}\cB_\delta$
  with
$\zeta_n=(\zeta_n^0,\overline \zeta_n,\underline \zeta_n)$, such that
$$
\|\zeta_n\|_{1,\delta_n} := \|\zeta_n^0\|_{1,\delta_n} + \|\overline
\zeta_n\| + \|\underline \zeta_n\| \to 0, \qquad n\to \infty,
$$
\end{itemize}
satisfying 
$$
v_n:=\exp_{G_{\delta_n,\oeps^n}(\v)}(\zeta_n).
$$
\end{proposition}

\begin{proof}
 We first prove the converse implication, namely that convergence in
 norm implies geometric convergence. We define shifts $(s_i^n)$,
 $i=1,\ldots,m$ inductively by
$$
s_m^n:=1/\delta_n + R_n, \qquad s_i^n:= s_{i+1}^n + 2R_n +
(T_i+\epsilon_i^n)/\delta_n. 
$$
We claim that $v_n(\cdot + s_i^n,\cdot)\to u_i$, $n\to \infty$ 
uniformly on compact sets. Let $R_0>0$ be fixed. By assumption
$$
\|\zeta_n^0(\cdot +
s_i^n,\cdot)|_{[-R_0,R_0]\times
  S^1}\|_{1,\delta_n}\to 0,   
\quad n\to \infty.
$$
By the Sobolev embedding theorem this implies 
$$
\|\zeta_n^0(\cdot +
s_i^n,\cdot)|_{[-R_0,R_0]\times
  S^1}\|_{C^0}\to 0, \quad n\to \infty. 
$$
Since
$$
G_{\delta_n,\oeps^n}(\v)(\cdot+s_i^n,\cdot)|_{[-R_0,R_0]\times
  S^1}=u_i|_{[-R_0,R_0]\times S^1}
$$ 
for $n$ sufficiently large, the conclusion follows. 
 
We now prove the direct implication. Let us pick a representative 
$$\v=(c_m,u_m,c_{m-1},\ldots,u_1,c_0)\in[\u]$$
and let $T_i$, $i=0,\ldots,m$ be the lengths of the intervals of
definition of $c_i$, with the convention $T_0=T_m=+\infty$.
We also choose arbitrary representatives $v_n\in[\widetilde v_n]$.  
By assumption there exist shifts $(s_i^n)$ such that
$v_n(\cdot+s_i^n,\cdot)$ converges to $u_i$ uniformly on compact
sets. We define 
\begin{equation}
 \epsilon_i^n := \delta_n(s_i^n-s_{i+1}^n-2R_n) - T_i, \qquad
 i=1,\ldots,m-1. 
\end{equation}
By Lemma~\ref{lem:connect} we have $\epsilon_i^n\to 0$, $n\to\infty$. 
We define partitions of the real line 
$$
-\infty=a_m^n \le b_m^n \le a_{m-1}^n \le \ldots \le
a_0^n \le
b_0^n = +\infty  
$$
by $b_m^n:=1/\delta_n$ and 
$$
a_{i-1}^n:=b_i^n + 2R_n, \quad b_{i-1}^n:= a_{i-1}^n +
(T_{i-1}+\epsilon_{i-1}^n)/\delta_n, \quad  i=1,\ldots,m.  
$$
We define a sequence of shifts $(s^n)$ by 
$$
s^n:=s_m^n - 1/\delta_n - R_n
$$
and we still denote by $v_n$ the shifted sequence
$v_n(\cdot+s^n,\cdot)$. 

\medskip 

We first show the existence of a unique
vector field $\zeta_n$ satisfying
$v_n=\exp_{G_{\delta_n,\oeps^n}(\v)}(\zeta_n)$. For that it is enough
to prove 
\begin{equation} \label{eq:ucI}
\lim_{n\to \infty} \sup_{s\in I_n,\theta\in S^1}
\mathrm{dist}(v_n(s,\theta), G_{\delta_n,\oeps^n}(\v)(s,\theta)) =0,
\end{equation}
where $I_n$ is an interval of the following form:
\begin{enum}
 \item $[b_i^n,a_{i-1}^n]$, $i=1,\ldots,m$;
 \item $[a_i^n,b_i^n]$, $i=1,\ldots,m-1$;
 \item $[b_m^n-K/\delta_n,b_m^n]$ or $[a_0^n,a_0^n+K/\delta_n]$, for
 any $K>0$.
\end{enum} 
The asymptotic behaviour of $v_n$ and $G_{\delta_n,\oeps^n}(\v)$
ensures that $\zeta_n$ is an element of the relevant $W^{1,p}$-space. 

We prove case (i) by contradiction. Assume that there exists
$\epsilon>0$ and a sequence $(\widetilde s_n,\widetilde \theta_n)\in
[b_i^m,a_{i-1}^n]\times S^1$ such that 
$$
\textrm{dist}(v_n(\widetilde s_n,\widetilde
\theta_n),G_{\delta_n,\oeps^n}(\v)(\widetilde s_n,\widetilde
\theta_n)) \ge \epsilon.
$$
Since~\eqref{eq:ucI} is satisfied if one replaces $v_n$ by
$u_i(\cdot-s_i^n,\cdot)$ (by definition of
$G_{\delta_n,\oeps^n}(\v)$), we also have 
\begin{equation} \label{eq:vntoui}
\textrm{dist}(v_n(\widetilde s_n,\widetilde
\theta_n),u_i(\widetilde s_n-s_i^n,\widetilde
\theta_n)) \ge \epsilon/2
\end{equation}
for $n$ large enough. By the assumption of uniform convergence on
compact sets $v_n(\cdot+s_i^n,\cdot)\to u_i(\cdot,\cdot)$, up to
passing to a subsequence we can assume that 
$\widetilde s_n-s_i^n\to\pm\infty$. 
We treat the case $\widetilde s_n-s_i^n\to\infty$, the other case
being similar. Since $\widetilde s_n\in [b_i^n,a_{i-1}^n]$ and
$\delta_n (a_{i-1}^n-b_i^n)=2\delta_n R_n\to 0$, we have
$\delta_n(\widetilde s_n-s_i^n)\to 0$. By Lemma~\ref{lem:cornershift}
we infer that $v_n(\cdot+\widetilde s_n,\cdot)\to \uev(u_i)$, which
means 
$$
\lim_{n\to\infty} v_n(\widetilde s_n,\cdot) =
\lim_{n\to\infty}u_i(\widetilde s_n-s_i^n,\cdot)
$$
and this contradicts~\eqref{eq:vntoui}. 

Note that the above proof shows that $v_n(\cdot+a_{i-1}^n,\cdot)\to
\uev(u_i)$ and $v_n(\cdot+b_i^n,\cdot)\to \oev(u_i)$, $i=1,\ldots,m$
uniformly on compact sets. 

We now prove case (ii). Let us fix $1\le i \le m-1$. An action
argument as the one in the proof of Lemma~\ref{lem:connect} shows that
$v_n(I_n\times S^1)$ is entirely contained in a small neighbourhood of
$S_{\gamma_i}$. We apply Proposition~\ref{prop:intervaldelta} to $v_n$
and $I_n\times S^1$ to obtain 
$$
\lim_{n\to\infty} \sup_{(s,\theta)\in I_n\times S^1} |z\circ v_n
(s,\theta)| =0
$$
and 
$$
\lim_{n\to\infty} \sup_{(s,\theta)\in I_n\times S^1} |\vartheta\circ v_n
(s,\theta) - \theta -
\varphi_{\delta_n(s-a_i^n)}^{f_{\gamma_i}}(\uev(u_{i+1}))| =0.
$$
The same two equations hold, by definition, if one replaces $v_n$ by
$G_{\delta_n,\oeps^n}(\v)$, and the conclusion follows. 

We now prove (iii). We treat only the case 
$I_n=[a_0^n,a_0^n+K/\delta_n]$, the other case being similar. 
An action argument as above shows that
$v_n(\cdot+a_0^n+K/\delta_n,\cdot)$ converges uniformly on compact
sets to a constant cylinder over some orbit $\gamma\in
S_{\gamma_0}$. By Lemma~\ref{lem:connect} we know that
$\gamma=\varphi_K^{f_{\gamma_0}} (\uev(u_1))$, and in particular is
not a critical point of $f_{\gamma_0}$. Now the conclusion follows in
the same way as in case (ii). 

\medskip 

We now show that
\begin{equation} \label{eq:In}
\lim _{n\to \infty} \| \zeta_n|_{I_n\times S^1} \|_{1,\delta_n} =0
\end{equation} 
in each of the cases (i)-(iii). We denote in the sequel
$$
|\zeta(s,\theta)|_1 := |\zeta(s,\theta)| + |\nabla _s
\zeta(s,\theta)| + |\nabla_\theta \zeta(s,\theta)| .
$$

We first consider case (i). Let us fix $K>0$ large enough. For $n$
large enough we can write  
$$
I_n=[s_i^n-R_n,s_i^n-K] \cup [s_i^n-K,s_i^n+K] \cup
[s_i^n+K,s_i^n+R_n]. 
$$

We first note that 
\begin{eqnarray*} \label{eq:estimatemiddle}
\int_{s_i^n-K}^{s_i^n+K} |\zeta(s,\theta)|_1^p \ 
  g_{\delta_n,\oeps^n}(s) dsd\theta &=&  \int_{s_i^n-K}^{s_i^n+K}
  |\zeta(s,\theta)|_1^p \ e^{d|s-s_i^n|}  dsd\theta \\ 
&\le & \sup_{\stackrel{\theta\in S^1} {s\in[s_i^n-K,s_i^n+K]}} 
 |\zeta(s,\theta)|_1^p \cdot e^{dK}.
\end{eqnarray*}
Since $v_n(\cdot+s_i^n,\cdot)$ and
$G_{\delta_n,\oeps^n}(\v)(\cdot+s_i^n,\cdot)$ converge uniformly on
compact sets together with their derivatives to $u_i$, the last term
goes to zero as $n\to\infty$. 

In order to estimate the integral on the interval
$[s_i^n-R_n,s_i^n-K]$ we apply Proposition~\ref{prop:intervaldelta}
on $[s_{i+1}^n+K,s_i^n-K]$ to $v_n$ to obtain 
$$
|z\circ v_n(s,\theta)|_1 \le  C(K) 
\frac {\cosh(\rho(s-\frac
{s_{i+1}^n+s_i^n}2))} {\cosh(\rho( \frac {s_i^n - s_{i+1}^n} 2
-K))}  \le  C_1C(K)e^{\rho(s-s_i^n+K)} 
$$
and
$$
|\vartheta\circ
v_n(s,\theta)-\theta-\varphi_{\delta_n(s-b_i^n)}^{f_{\gamma_i}}(p_i^n)|_1
\le  C_1C(K)e^{\rho(s-s_i^n+K)}, 
$$
where $|\cdot|_1$ stands for the pointwise $C^1$-norm, for some
$p_i^n \in S_{\gamma_i}$ such that $p_i^n \to \oev(u_i)$, $n \to \infty$.
 Similar estimates hold, by definition, if one replaces $v_n$ by
$G_{\delta_n,\oeps^n}(\v)$ and $p_i^n$ with $\oev(u_i)$.
Hence we obtain 
\begin{equation} \label{eq:zetan}
|\zeta_n(s,\theta) - \okappa_i^n X_H |_1
\le  C_1C(K)e^{\rho(s-s_i^n+K)},
\end{equation}
where $\okappa_i^n \to 0$ as $n \to \infty$ and
\begin{eqnarray}
\lefteqn{C(K) = C\max(\|Q_\infty v_n(s_{i+1}^n+K)\|, 
\|Q_\infty v_n(s_i^n-K)\|,} \nonumber \\
&& \qquad \qquad \|Q_\infty \widetilde v_n(s_{i+1}^n+s_n+K)\|, 
\|Q_\infty \widetilde v_n(s_i^n+s_n-K)\|). \label{eq:CK} 
\end{eqnarray}
We obtain 
\begin{eqnarray*} \label{eq:estimateup}
\lefteqn{ \int_{s_i^n-R_n}^{s_i^n-K} |\zeta_n(s,\theta)  
- \okappa_i^n X_H |_1^p \ g_{\delta_n}(s) dsd\theta} \\
& = & \int_{s_i^n-R_n}^{s_i^n-K} |\zeta_n(s,\theta)  - \okappa_i^n X_H |_1^p \ 
e^{-d(s-s_i^n)} dsd\theta \\
& \le & C_2C(K)^pe^{dK}.  
\end{eqnarray*}
A similar estimate holds when the interval of integration is 
$[s_i^n+K,s_i^n+R_n]$, with $C(K)$ replaced with $C'(K)$. 
Letting $n\to \infty$ we obtain 
\begin{eqnarray*}
 \lim _{n\to \infty} && \hspace{-1cm} \int_{I_n\times S^1}
|\zeta_n(s,\theta)  - \okappa_i^n \beta(-s+s_i^n) X_H  
- \ukappa_i^n \beta(s-s_i^n) X_H |_1^p \ 
g_{\delta_n}(s) dsd\theta  \\
&\le& C_2(C(K)^p+C'(K)^p) e^{dK}. 
\end{eqnarray*}

We let now $K\to \infty$. Proposition~\ref{prop:intervaldelta} implies
that, for $K>K'$, we have $C(K')\le C_3C(K)e^{-\rho(K'-K)}$, hence
$C(K)^p e^{dK}\to 0$ as $K\to \infty$ because $d<\rho p$. 
The equality \eqref{eq:In} follows.

We now consider case (ii). We fix $K>0$
large enough and apply Proposition~\ref{prop:intervaldelta} on the
interval 
$[s_{i+1}^n+K,s_i^n-K] \supset [s_{i+1}^n+R_n,s_i^n-R_n]=I_n$ to
obtain as in case (i) 
\begin{equation*}
|\zeta_n(s,\theta) - \kappa_i^n \zeta_{i,\delta}(s,\theta) |_1 \le
C(K) \frac {\cosh(\rho(s-\frac
{s_{i+1}^n+s_i^n}2))} {\cosh(\rho( \frac {s_i^n - s_{i+1}^n} 2
-K))},
\end{equation*}
where $C(K)$ is given by~\eqref{eq:CK},
$\zeta_{i,\delta}(s-\frac {s_{i+1}^n+s_i^n} 2)$ generates the kernel
of the linearized operator 
corresponding to gradient trajectory $c_i$ as in Proposition 
\ref{prop:Surj_udelta} and $\kappa_i^n \zeta_{i,\delta}(b_i^n,\cdot) 
= \okappa_i^n X_H$. In particular, we have
$\kappa_i^n \to 0$, $n \to \infty$. We get 
$$
\int_{s_{i+1}^n+R_n}^{s_i^n-R_n} 
|\zeta_n(s,\theta) - \kappa_i^n \zeta_{i,\delta} |_1^p \
g_{\delta_n}(s) dsd\theta 
\le  C_2C(K)^p e^{(d-\rho p)(R_n-K)}.  
$$  
The last term goes to zero as $n\to \infty$.
Equality~\eqref{eq:In} follows now as in case~(i). Case (iii) is
entirely similar to case (ii).  

\medskip 

In order to complete the proof of $\|\zeta_n\|_{1,\delta_n}\to 0$,
$n\to \infty$, it is enough to show that $\|\zeta_n|_{I_n\times
  S^1}\|_{1,\delta_n}\to 0$ if $I_n=]-\infty,b_m^n-K/\delta_n]$ or
$I_n=[a_0^n+K/\delta_n,+\infty[$, for any $K>1$. The two
cases are entirely similar and we give the argument only for
$I_n=]-\infty,b_m^n-K/\delta_n]$. By
Proposition~\ref{prop:asymptoticdelta}, for $n$ sufficiently large we
have  
$v_n(s,\theta)=\exp_{u_{\delta_n,\gamma_m,-\infty,1}(s,\theta)}
(\eta_n(s,\theta))$, with $\eta_n=(\eta_n^0,\overline \eta_n)$, 
$\eta_n^0 \in W^{1,p}(I_n \times
S^1,u_{\delta_n,\gamma_m,-\infty,1}^*T\widehat
W;e^{r|s|}ds\;d\theta)$, $\overline \eta_n\in \oV'$. Since
$v_n(b_n^m,\cdot)\to \oev(u_m)$ we have $\|\overline \eta_n\|_\infty
\to 0$. Since
$G_{\delta_n,\oeps^n}(\v)=u_{\delta_n,\gamma_m,-\infty,1}$ on $I_n$,
we obtain $\zeta_n=\eta_n$, so that $\|\overline \zeta_n\|\to 0$. The
fact that $\|\zeta_n^0\|_{1,\delta_n}\to 0$ follows from the fact that
$d<r$.  
\end{proof}

\medskip

We explain now how to construct a right inverse for
$D_{G_{\delta,\oeps}(\widetilde w)}$ which is uniformly bounded with respect
to $\delta\to 0$. The space $\widetilde \cB_\delta$ is a Banach
manifold whose tangent space at $\widetilde w$ is
\begin{equation} \label{eq:Ttw}
T_{\widetilde w} \widetilde\cB_\delta
=  T_{v_m} \cB'_\delta  \ _{d\uev}\oplus_{d\oev} T_{u_m} \cB \
_{d\uev}\oplus_{d\oev} T_{v_{m-1}} \cB'_\delta \ _{d\uev}\oplus
\ldots \oplus_{d\oev} T_{v_0} \cB'_\delta.
\end{equation}
Recall that the fibered sum of two vector spaces $W_1$, $W_2$ with
respect to linear maps $f_i:W_i\to W$ is the vector space
$$
W_1 \ _{f_1}\oplus_{f_2} W_2 := \{(w_1,w_2)\in W_1\oplus W_2  :
f_1(w_1)=f_2(w_2)
\}.
$$
If $(W_1,\|\cdot\|_1)$, $(W_2,\|\cdot\|_2)$ and $W$ are normed
vector spaces, and $f_1$, $f_2$ are continuous linear maps, then 
$W_1 \ _{f_1}\oplus_{f_2} W_2$ is a closed subspace of $W_1\oplus
W_2$ and inherits the norm $\|\cdot\|_1 +\|\cdot\|_2$ from $W_1\oplus W_2$.
In our case
$$
d\uev: T_{v_m} \cB'_\delta = W^{1,p,d}\oplus\oV'\oplus \uV' \to
T_{\uev(v_m)}S_{\gamma_m}
$$
factors through the projection on $\uV'$, and similarly for the other
evaluation maps.
Therefore the above fibered sum only affects the summands $\oV$,
$\uV$, $\oV'$, $\uV'$, so that $T_{\widetilde w} \widetilde\cB_\delta$
is a subspace of codimension $2m$ in
$$
T_{v_m} \cB'_\delta  \oplus T_{u_m} \cB
\oplus T_{v_{m-1}} \cB'_\delta\oplus
\ldots \oplus T_{v_0} \cB'_\delta.
$$
As above, the norm on $T_{\tw}\widetilde \cB_\delta$ is induced from the
ambient space. Recall that the $W^{1,p}$-component has  
weight $e^{d|s|}$ for each $T_{u_j}\cB$, weight $e^{d\left| |s|-s_{i,\delta}\right|}$ 
for each $T_{v_i}\cB'_\delta$, $i=1, \ldots, m-1$, 
weight $e^{d\left| s+s_{0,\delta}\right|}$
for $i=0$ and weight $e^{d\left| s-s_{m,\delta}\right|}$
for $i=m$, with $s_{i,\delta}$ as in the definition of $g_{\delta,\oeps}$. 

The sections $\dbar_{H,J}:\cB \to \cE$ and $\dbar:\cB'_\delta\to \cE$
defined by~\eqref{eq:dbarHdelta} and~\eqref{eq:dbarprime} give rise to
a section over $\widetilde \cB_\delta$. We denote its
vertical differential
by
$$
D_{\widetilde w}:T_{\widetilde w} \widetilde\cB_\delta \to
L^{p,d}(v_m^*T\widehat W) \oplus L^{p,d}(u_m^*T\widehat W) \oplus
\ldots \oplus
L^{p,d}(v_0^*T\widehat W),
\index{$D_{\widetilde w}$, fibered sum linearized operator}
$$
where
$$
L^{p,d}(v_i^*T\widehat W):= L^p(\R\times S^1,v_i^*T\widehat
W;g_{\delta,\epsilon_i,v_i}(s)dsd\theta),
$$
$$
L^{p,d}(u_i^*T\widehat W):= L^p(\R\times S^1,u_i^*T\widehat
W;g_{\delta,u_i}(s)dsd\theta).
$$


\begin{lemma} \label{lem:Dwtilde}
Let $J\in\Jreg(H)$ and
  $\{f_\gamma\}\in\Freg(H,J)$. Let
  $\oeps=\!(\epsilon_1,\ldots,\epsilon_{m-1})$ and let $\widetilde w \in
  \widetilde \cB_\delta$ be as in
  Proposition~\ref{prop:catenation}. 
The image of the operator $D_{\widetilde w}$ has codimension
$m-1$ and admits a complement spanned by sections
$\eta_i \in
L^{p,d}(v_i^*T\widehat W)$, $i=1,\ldots,m-1$ which
are respectively supported in
$$
[-(T_i+\epsilon_i)/2\delta, -(T_i+\epsilon_i)/2\delta +1] \times S^1 \cup 
[(T_i+\epsilon_i)/2\delta-1,(T_i+\epsilon_i)/2\delta]\times S^1\!.
$$
The
operator $D_\tw$ admits a right inverse $Q_\tw$ defined on its 
image and whose norm is uniformly bounded with respect to 
$\delta\to 0$.
\end{lemma}

\begin{proof}
We show that
\begin{equation} \label{eq:imDwtilde}
\im\,D_{\widetilde w} = \im\,D'_{v_m} \oplus \im\, D_{u_m} \oplus
\im\, D'_{v_{m-1}} \oplus \ldots \oplus \im\,D'_{v_0}=:E.
\end{equation}
By definition we have $\im\,D_{\widetilde w} \subset E$.
Let us now choose $(x_m,y_m,\ldots,x_0)\in E$ and $\widetilde x_i$
and $\widetilde y_j$ such that $D'_{v_i}(\widetilde x_i)=x_i$, $
D_{u_j}(\widetilde y_j)=y_j$.
We need to modify $\widetilde x_i$ and $\widetilde y_j$ by elements
lying in the kernels of the corresponding operators so that
\begin{equation} \label{eq:matching}
d\oev(\widetilde y_j)=d\uev(\widetilde x_j), \quad
d\uev(\widetilde y_j)=d\oev(\widetilde x_{j-1}), \qquad j=1,\ldots,m.
\end{equation}

Let us first assume $m>1$. We have
$$
T_{v_m}\cM'_{\delta,1,-\infty}(S_{\og},S_{\og};H,J)
\times T_{u_m}\cM^{A_m}(S_{\og},S_{\gamma_{m-1}};H,J) = 
\ker D'_{v_m} \times \ker D_{u_m}
$$
and, because $\{f_\gamma\}\in\Freg(H,J)$, the map
$$
(d\uev,d\oev):\ker D'_{v_m} \times \ker D_{u_m} \to
T_{\uev(v_m)}S_{\og} \times T_{\oev(u_m)}S_{\og}
$$
is transverse to the diagonal. We can therefore modify $\widetilde
x_m$ and $\widetilde y_m$ so that
$d\oev(\widetilde y_m)=d\uev(\widetilde x_m)$. Similarly
the map
$$
(d\uev,d\oev):\ker D_{u_1} \times \ker D'_{v_0} \to
T_{\uev(u_1)}S_{\ug} \times T_{\oev(v_0)}S_{\ug}
$$
is transverse to the diagonal and we can modify $\widetilde y_1$,
$\widetilde x_0$ in order to achieve $d\uev(\widetilde
y_1)=d\oev(\widetilde x_1)$. For $i=1,\ldots,m-1$ the maps
$$
(d\oev,d\uev):\ker D'_{v_i} \to T_{\oev(v_i)}S_{\gamma_i} \times
T_{\uev(v_i)}S_{\gamma_i}
$$
are surjective and we can modify $\widetilde x_i$ so
that~\eqref{eq:matching} is satisfied.

If $m=1$ the regularity hypothesis on ${f_\gamma}$ ensures that
the map
\begin{eqnarray*}
(d\uev,d\oev,d\uev,d\oev) & : &
\ker D'_{v_1} \times \ker D_{u_1} \times \ker D'_{v_0} \\
&& \to T_{\uev(v_1)}S_{\og} \times T_{\oev(u_1)}S_{\og}
\times T_{\uev(u_1)}S_{\ug} \times T_{\oev(v_0)}S_{\ug}
\end{eqnarray*}
is transverse to the product of the diagonals in the first two and in
the last two factors. We can therefore modify simultaneously
$\widetilde x_1$, $\widetilde y_1$, $\widetilde x_0$ in order to
achieve~\eqref{eq:matching}. Therefore~\eqref{eq:imDwtilde} is proved.
It then follows from
Proposition~\ref{prop:support_of_section} that the image of $D_\tw$
has codimension $m-1$ and is spanned by sections $\eta_i\in
L^{p,d}(v_i^*T\widehat W)$ supported in the desired
intervals. 

We now prove that $D_\tw$ admits a uniformly bounded right inverse
defined on its image. We observe that $D_\tw$ is the restriction
to $\mathrm{dom}(D_\tw)$ of the direct sum of operators
$D:=D'_{v_m}\oplus D_{u_m} \oplus D'_{v_{m-1}} \oplus \dots \oplus
D_{u_1} \oplus D'_{v_0}$.  
Let $\zeta_m$, $\zeta_0$ be generators of $\ker D'_{v_m}$, $\ker
D'_{v_0}$ and, for $i=1,\dots,m-1$, let $\zeta_i^1$,
$\zeta_i^2$ be the basis of $\ker D'_{v_i}$ constructed in
Proposition~\ref{prop:support_of_section}. We denote by $K$ the vector 
space spanned by these $2m$ sections, viewed as elements of
$\mathrm{dom}(D)$. Then $\dim\, K=2m$ and $K$ is a
complement of $\mathrm{dom}(D_\tw)$. Let
$P:\mathrm{dom}(D)\to\mathrm{dom}(D_\tw)$ be the projection parallel
to $K$, let $Q_{u_j}$, $j=1,\dots,m$ be
uniformly bounded right inverses for $D_{u_j}$, let $Q_{v_i}$,
$i=0,\dots,m$ be uniformly bounded right inverses for $D'_{v_i}$
defined on their images as in
Proposition~\ref{prop:support_of_section}, and denote 
$Q:=Q_{v_m}\oplus Q_{u_m} \oplus Q_{v_{m-1}} \oplus \dots \oplus
Q_{u_1}\oplus Q_{u_0}$. Since $K\subset \ker D$ the operator $P\circ 
Q:\im(D)=\im(D_\tw)\to\mathrm{dom}(D_\tw)$ is a right inverse for
$D_\tw$ defined on its image, and we claim that its norm is uniformly
bounded for $\delta\to 0$. The norm of $Q$ is uniformly
bounded for $\delta\to0$, so that it is enough to prove
that the norm of $P$ is uniformly bounded for $\delta\to 0$. 

The sections $\zeta_0$, $\zeta_m$ and $\zeta_i^1$, $\zeta_i^2$ for
$i=1,\dots,m-1$ have the property that their respective asymptotic
values (obtained by applying $d\oev$ and $d\uev$) are not
simultaneously zero. Moreover, the same is true for any linear
combination of $\zeta_i^1$ and $\zeta_i^2$ for $i=1,\dots,m-1$. 
As a consequence, there exists a uniform constant $C>0$ such
that, for any $\x=(x_m,0,x_{m-1},\dots,0,x_0)\in K$, we have 
\begin{equation} \label{eq:oevuev}
\|\x\|_{1,\delta}
\le C \big( 
|d\uev(x_m)|+|d\oev(x_0)| + 
\sum_{i=1}^{m-1}
|d\uev(x_i)|+|d\oev(x_i)|
\big). 
\end{equation} 
Given $v\in\mathrm{dom}(D)$ we have $P(v)=v+w$ for some vector $w\in
K$ which is uniquely
determined by the asymptotic values of the components of $v$, and it
follows from~\eqref{eq:oevuev} that 
$$
  \|w\|_{1,\delta}\le C\|v\|_{1,\delta}.
$$
We obtain 
$$
\frac {\|P(v)\|_{1,\delta}} {\|v\|_{1,\delta}} =\frac
{\|v+w\|_{1,\delta}} {\|v\|_{1,\delta}} \le 1+C, 
$$
so that the norm of $P$ is uniformly bounded by $1+C$. This proves the
Lemma. 
\end{proof}

\begin{proposition} \label{prop:gluing}
  Let $J\in\Jreg(H)$ and
  $\{f_\gamma\}\in\Freg(H,J)$. Let $\widetilde w \in \widetilde
\cB_\delta$ and
  $\oeps(\delta)=(\epsilon_1(\delta),\ldots,\epsilon_{m-1}(\delta))$
  be as in Proposition~\ref{prop:catenation}. The operator 
  \begin{eqnarray*}
  D_{G_{\delta,\oeps}(\widetilde w)}&: &
  W^{1,p}(\R\times S^1,G_{\delta,\oeps}(\widetilde w)^*T\widehat
  W;g_{\delta,\oeps}(s)dsd\theta) \oplus \oV_{v_m}' \oplus \uV_{v_0}'
  \\
  && \qquad \qquad \to \
  L^p(\R\times S^1,G_{\delta,\oeps}(\widetilde w)^*T\widehat W;g_{\delta,\oeps}(s)dsd\theta)
  \end{eqnarray*}
  is surjective and admits a right inverse
  $Q_\delta=Q_{\delta,\oeps,\widetilde w}$\index{$Q_\delta$, right
  inverse for gluing theorem}
  whose $\delta$-norm is uniformly bounded with respect to $\delta\to 0$.
\end{proposition}

\begin{proof} Our proof is modelled on the proof of the gluing theorem
for holomorphic spheres by McDuff and
Salamon~\cite[Ch.~10]{McDS}
. Let
$$
v_m^\delta,u_m^\delta,v_{m-1}^\delta,
\ldots,u_1^\delta, v_0^\delta
$$
be the extensions of $\widehat v_m,\widehat u_m,\widehat v_{m-1},
\ldots,\widehat u_1, \widehat v_0$ to $\R\times S^1$ defined by the
same formulas. Note that
\begin{eqnarray*}
u_j^\delta(s,\theta)=u_j(s,\theta), &  s\in[-R+1,R-1], \\
v_m^\delta(s,\theta)=v_m(s,\theta), & s \notin
[1/\delta-1,1/\delta], \\
v_0^\delta(s,\theta)=v_0(s,\theta), &
s \notin [-1/\delta,-1/\delta+1]
\end{eqnarray*}
and $v_i^\delta(s,\theta)=v_i(s,\theta)$ for $s$ outside
$[-(T_i+\epsilon_i)/2\delta,-(T_i+\epsilon_i)/2\delta +1] \cup
[(T_i+\epsilon_i)/2\delta-1,(T_i+\epsilon_i)/2\delta]$ 
and $i=1,\ldots,m-1$. The difference between $v_i^\delta$ and $v_i$ on
the one hand, and that between $u_j^\delta$ and $u_j$ on the other
hand is exponentially small as $\delta\to 0$. This implies that the
operators $D_{u_j^\delta}$, $D'_{v_0^\delta}$ and $D'_{v_m^\delta}$
are surjective for $\delta$ small enough and admit uniformly bounded
right inverses, while the operators
$D'_{v_i^\delta}$, $i=1,\ldots,m-1$ have a codimension one image with
a supplement spanned by a smooth section $\eta_i$ supported in
$[-(T_i+\epsilon_i)/2\delta,-(T_i+\epsilon_i)/2\delta +1]\times S^1 \cup
[(T_i+\epsilon_i)/2\delta-1,(T_i+\epsilon_i)/2\delta] \times S^1$, and
admit uniformly bounded ``right inverses'' defined on their image. It
follows that the vertical differential $D_{\widetilde w^\delta}$
satisfies the conclusions of Lemma~\ref{lem:Dwtilde}, where
$\widetilde w^\delta :=
(u_1^\delta,\ldots,u_m^\delta,v_0^\delta,\ldots,v_m^\delta)$. In
particular, it admits a uniformly bounded right inverse defined on
its image, which we denote by $Q_{\widetilde w ^\delta}$
(see~\cite[Lemma~10.6.1]{McDS} for a similar statement
in the case of holomorphic spheres). This means that there exists a
constant $c_0>0$ such that
$$
\|Q_{\widetilde w ^\delta} \x \|_{W^{1,p,d}} \le c_0 \|
\x \|_{L^{p,d}}
$$
for all $\x\in \im\, D_{\widetilde w^\delta}$ and $\delta>0$.

We define an operator $T_\delta$ by the commutative diagram
$$
\xymatrix
@C=60pt
@R=30pt@W=1pt@H=1pt
{T_{\widetilde w^\delta}\widetilde \cB _\delta \ar[d]_G &
  L^{p,d}(\widetilde w^{\delta *}T\widehat W)
\ar[l]_-{Q_{\widetilde w ^\delta}
\circ P} \\
\mathrm{dom}(D_{G_{\delta,\oeps}(\tw)})
&
L^p(\R\times S^1,G_{\delta,\oeps}(\widetilde w)^*T\widehat
W;g_{\delta,\oeps}(s)dsd\theta)
\ar[u]_{S} \ar[l]^-{T_\delta}
}
$$
where 
$$
L^{p,d}(\widetilde w^{\delta *}T\widehat W) 
:=
L^{p,d}({v_m^\delta}^*T\widehat W) \oplus 
L^{p,d}({u_m^\delta}^*T\widehat W) 
\oplus\ldots \oplus 
L^{p,d}({v_0^\delta}^*T\widehat W).
$$
In the rest of the proof we shall omit the subscript $\oeps$
from $G_{\delta,\oeps}$ and $g_{\delta,\oeps}$. An element of 
$L^{p,d}(\widetilde w^{\delta *}T\widehat W)$ is denoted by 
$$
\x = 
(x_m,y_m,\ldots,x_0).
$$
The {\it mixing map} $P$, the {\it splitting map} $S$ and the {\it
  gluing map} $G$ are defined below, and we shall prove that $P,S,G$
are uniformly bounded with respect to $\delta\to
0$. 
We shall also prove that $T_\delta$ is an approximate right inverse
for $D_{G_\delta(\widetilde w)}$, i.e.
\begin{equation} \label{eq:Tdelta}
  \|D_{G_\delta(\widetilde w)}T_\delta \eta - \eta \|_\delta   \le
\frac 1 2 \|\eta\|_\delta
\end{equation}
for $\delta$ sufficiently small and $\eta \in
L^p(\R\times S^1,
G_\delta(\widetilde w)^*T\widehat W;g_\delta(s)dsd\theta)$.
This implies that $D_{G_\delta(\widetilde w)}T_\delta$ is
invertible (with the norm of its inverse bounded by $2$), 
and $T_\delta (D_{G_\delta(\widetilde w)}T_\delta)^{-1}$ is a 
right inverse for $D_{G_\delta(\widetilde w)}$. Since $P,S,G$ are
uniformly bounded, the norm of $T_\delta (D_{G_\delta(\widetilde w)}T_\delta)^{-1}$
is bounded by a constant multiple of $\|Q_{\widetilde
w^\delta}\|$, hence is uniformly bounded and the conclusion of the
Proposition follows.

\medskip

For every $L>0$ we fix a smooth function\index{$\beta_L$, cutoff function}
$$
\beta_L:\R\to [0,1]
$$
which vanishes for $s\le 0$, which is constant
equal to $1$ for $s \ge L$ and whose derivative is bounded
by $2/L$. We moreover require that, for $L$ large enough, the 
function $\beta_L$ vanishes for $s\le 1$. 

We define the mixing map $P\index{$P$, mixing map}$. Let
$$
p_i :   L^{p,d}(\widetilde w^{\delta *}T\widehat W) \to \im \,
D'_{v_i^\delta}, \qquad i=0,\dots,m
$$
be the projection on $L^{p,d}((v_i^\delta)^*T\widehat W)$ followed 
by the projection 
on $\im \, D'_{v_i^\delta}$ parallel to $\eta_i$. Recall the
definition~\eqref{eq:elli} of $\ell_i$ for $i=0,\dots,m$ and let
\begin{eqnarray*}
& & \qquad \qquad q_j :   L^{p,d}(\widetilde w^{\delta *}T\widehat W) 
\to \im \, D_{u_j^\delta},\\
q_j(\x)(s,\theta) & := & y_j(s,\theta)\\
 & & + \ 
\beta_1(s-\ell_j)\cdot\Big( (\one-p_j)(x_j)\Big)(s-\ell_j,\theta) \\
& & + \ (1- \beta_1(s-\ell_{j-1}))\cdot\Big( 
(\one-p_{j-1})(x_{j-1})\Big)(s-\ell_{j-1},\theta)
\end{eqnarray*}
for $j=1,\dots,m$. We define 
$$
P:L^{p,d}(\widetilde w^{\delta *}T\widehat W)\to \im \, 
D_{\widetilde w^\delta}
$$
by
$$
P:= p_m + q_m + p_{m-1} + \ldots + 
q_1 + p_0.
$$
The norm of $P$ is uniformly bounded with respect to $\delta\to
0$
since the norm of each $p_i$ is uniformly bounded by $1$. 

\medskip 

We define now the splitting map
$$
S(\eta) := \x = (x_m,y_m,\ldots,x_0).
\index{$S$, splitting map}
$$
We recall the definition~\eqref{eq:shifts} of the catenation shifts
$$
0=s_{v_m} < s_{u_m} < s_{v_{m-1}} < \ldots < s_{u_1} < s_{v_0},
$$
and set
$$
x_m(s,\theta) := \beta_1(1/\delta-s)\eta(s,\theta),
$$
$$
x_0(s,\theta) :=
\beta_1(1/\delta+s)
\eta(s+ s_{v_0},\theta),
$$
and, for $i=1,\ldots,m-1$, $j=1,\ldots,m$,
$$
y_j(s,\theta) := \left\{
\begin{array}{ll}
(1-\beta_1(-R-s))\eta(s+s_{u_j},\theta), &
s\le 0, \\
(1-\beta_1(-R+s))\eta(s+s_{u_j},\theta), &
s\ge 0,
\end{array}\right.
$$
$$
x_i(s,\theta) := \left\{
\begin{array}{ll}
\beta_1((T_i+\epsilon_i)/2\delta+s)\eta(s+s_{v_i},\theta), &
s\le 0, \\
\beta_1((T_i+\epsilon_i)/2\delta-s)\eta(s+s_{v_i},\theta), &
s\ge 0.
\end{array}\right.
$$
It follows from the definition that the norm of $S$ is
uniformly bounded by $1$.

\medskip

We define now the gluing map 
$\zeta:=G(\widetilde \x)\index{$G$, gluing map}$, 
$\widetilde 
\x=(\widetilde x_m,\widetilde y_m,\widetilde
x_{m-1},\ldots,\widetilde x_0)\in T_{\widetilde w^\delta}\widetilde
\cB_\delta$ by ``slowly interpolating'' the components of $\widetilde \x$.
For $j=1,\ldots,m$, $i=1,\ldots,m-1$ we put 
\begin{equation} \label{eq:zeta}  
\zeta(s,\theta) := \left\{ \begin{array}{lrcl}
\widetilde x_m(s,\theta), & -\infty < & s & \le 1/\delta - R/2, \\
\widetilde y_j(s-s_{u_j},\theta), & s_{u_j} - R/2 \le & s & \le s_{u_j} +
R/2,\\
\widetilde x_i(s-s_{v_i},\theta), 
& s_{v_i} - \ell_i + 3R/2 \le & s & \le s_{v_i} + \ell_i - 3R/2, \\
\widetilde x_0(s-s_{v_0},\theta), & s_{v_0}-1/\delta + R/2 \le & s & < +\infty.
\end{array}\right.
\end{equation}
The above formula leaves out two types of intervals, on which the
actual interpolation takes place (see Figure~\ref{fig:gluing}).
\begin{itemize}
\item If $s_{v_j} + \ell_j - 3R/2  \le s \le s_{u_j} -R/2$ (interval of
length $R$), we define
\begin{eqnarray*}
\lefteqn{\zeta(s,\theta)  :=  \widetilde x_j(+\infty,\theta)}  \\
& + & \!\!\!(1-\beta_{\frac{R} 2}(s-s_{v_j}
- \ell_j  + R))
\big(
\widetilde x_j(s-s_{v_j},\theta)-\widetilde x_j(+\infty,\theta)
\big) \\
& + &  \!\!\!(1-\beta_{\frac{R}2}(-s+s_{u_j}-R))
\big(
\widetilde y_j(s-s_{u_j},\theta)-\widetilde y_j(-\infty,\theta)
\big).
\end{eqnarray*}
\item If $s_{u_j}+R/2\le s \le s_{v_{j-1}} - \ell_{j-1} + 3R/2$ (interval
of length $R$), we define
\begin{eqnarray*}
\lefteqn{\zeta(s,\theta)  :=  \widetilde x_{j-1}(-\infty,\theta)} 
\\
&\!\!+&\!\!\!\!\!(1\!-\!\beta_{\frac{R} 2}(-s\!+\!s_{v_{j-1}} 
\!-\! \ell_{j-1} + R))
\big(
\widetilde x_{j-1}(s-s_{v_{j-1}},\theta)-\widetilde x_{j-1}(-\infty,\theta)
\big) \\
&\!\!+&\!\!\!\!\!(1-\beta_{\frac{R}2}(s-s_{u_j}-R))
\big(
\widetilde y_j(s-s_{u_j},\theta)-\widetilde y_j(+\infty,\theta)
\big).
\end{eqnarray*}
\end{itemize}

\begin{figure}
         \begin{center}
\scalebox{0.95}{\input{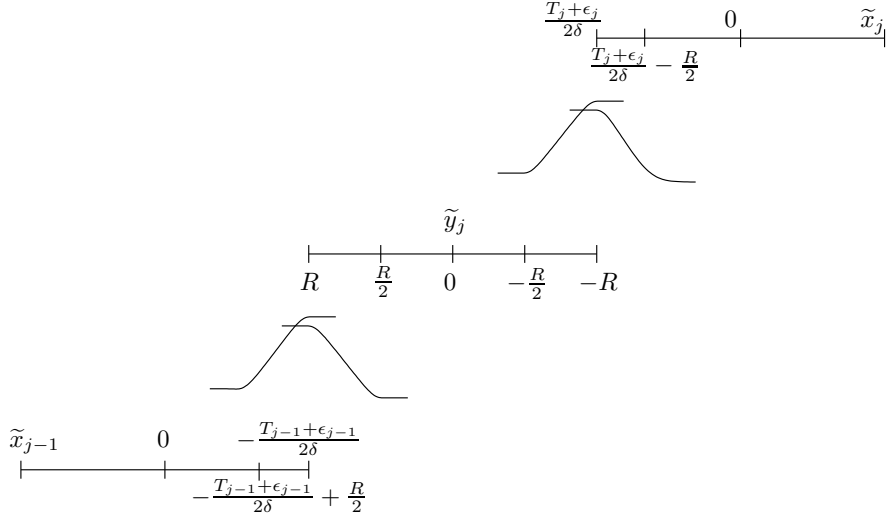}}
\caption{The gluing map $G$. \label{fig:gluing}}
         \end{center}
\end{figure}

\noindent The section $\zeta$ is indeed of class $W^{1,p}$ because
$$
\widetilde y_j(-\infty,\theta) = \widetilde x_j(+\infty,\theta),
\qquad \widetilde
y_j(+\infty,\theta)=\widetilde x_{j-1}(-\infty,\theta).
$$
That the norm of $G$ is uniformly bounded
with respect to $\delta\to
0$ follows directly from the definition~\eqref{eq:Ttw} of the norm on
$T_{\tw^\delta}\widetilde \cB_\delta$, 
as well as from the definition~\eqref{eq:norm1delta} of the norm
$\|\cdot\|_{1,\delta}$ on $\mathrm{dom}(D_{G_{\delta,\oeps}(\tw)})$
(see also Remark~\ref{rmk:norm1delta}).

\medskip 

Let us now prove the estimate~\eqref{eq:Tdelta}. On each of the 
intervals appearing 
in~\eqref{eq:zeta} we have 
$(D_{G_\delta(\widetilde w)}T_\delta\eta)(s,\theta)
=\eta(s,\theta)$ 
and we are therefore left to examine intervals of the type 
$[s_{v_j} 
+ \ell_j - 3R/2, s_{u_j} -R/2]$ and $[s_{u_j}+R/2, s_{v_{j-1}} - 
\ell_{j-1} + 3R/2]$. 
We treat only the first case since the second one 
is entirely similar.

Upon applying the operator $D_{G_\delta(\widetilde w)}$ to $\zeta$ 
we obtain five types of terms as following. 
\begin{itemize}
 \item $D_{G_\delta(\widetilde w)} \widetilde x_j(+\infty,\theta)$. 
Since 
 $\widetilde x_j(+\infty,\theta)$ does not depend on $s$ we 
can view 
 $D_{G_\delta(\widetilde w)}$ as a family of operators on 
$S^1$. Then we have 
 \begin{eqnarray*}
  \|D_{G_\delta(\widetilde 
w)} \widetilde x_j(+\infty,\theta)\|_\delta & = &  
\|(D_{G_\delta(\widetilde w)} - D_{v_j(+\infty,\theta)})\widetilde 
x_j(+\infty,\theta)\|_\delta \\
  & \le & \|D_{G_\delta(\widetilde 
w)} - D_{v_j(+\infty,\theta)}\|_\delta
  \|\widetilde 
x_j(+\infty,\theta)\| \\
    & \le & C(\delta)\|\eta\|_\delta. 
\end{eqnarray*}
 Here $D_{v_j(+\infty,\theta)}$ denotes the 
linearized operator at the constant cylinder $v_j(+\infty,\theta)$,  
the norm $\|\widetilde x_j(+\infty,\theta)\|$ 
is induced
from the ($1$-dimensional) space $\uV'_{v_j}$, and 
$$
C(\delta)\to 0, \qquad \delta\to 0.
$$
This last statement 
and the last inequality follow from 
\begin{eqnarray*}
\lefteqn{\|D_{G_\delta(\widetilde w)} \!-\! 
D_{v_j(+\infty,\theta)}\|_\delta } \\
& \!\!\!\le & \!\!\!C( \| \widehat v_j \!-\! 
v_j(+\infty,\theta)\|_{L^{1,p,d}([(T_j+\epsilon_j)/2\delta - R/2,(T_j+\epsilon_j)/2\delta]\times   
S^1)} \\
& & \quad + \ 
\| \widehat u_j \!-\! 
u_j(-\infty,\theta)\|_{L^{1,p,d}([-R,-R/2]\times 
S^1)})
\end{eqnarray*}
and the fact that the intervals of integration 
migrate to $\pm\infty$. The above 
inequality makes crucial use of 
the fact that the weight $g_\delta$ on the necks is 
given by the 
exponential weight of the ambient spaces $\cB_\delta$, $\cB'_\delta$. 
Moreover, we have $\|\widetilde x_j(+\infty,\theta)\|\le 
\|\widetilde \x\|\le C\|\eta\|_\delta$ 
because $Q_{\widetilde w}$, 
$P$ and $S$ are uniformly bounded with respect to $\delta$.
\item 
$-\beta'_{R/2}(s-s_{v_j} -\ell_j + R)\big(\widetilde 
x_j(s-s_{v_j},\theta)-
\widetilde x_j(+\infty,\theta)\big)$, as well as 
$\beta'_{R/2}(-s+s_{u_j}-R))
\big(\widetilde y_j(s-s_{u_j},\theta)-\widetilde 
y_j(-\infty,\theta)\big)$. The $\delta$-norm 
of each of these two 
terms is bounded by $C(\delta)\|\eta\|_\delta$, with 
$C(\delta)\to 
0$ as $\delta\to 0$. To see this we first use that 
$|\beta'_{R/2}| 
\le 4/R \to 0$, $\delta\to 0$. Secondly we use that 
$\|\widetilde x_j(s-s_{v_j},\theta)-\widetilde 
x_j(+\infty,\theta)\|\le \| \widetilde \x\|
\le C\|\eta\|_\delta$ and 
$\|\widetilde y_j(s-s_{u_j},\theta)-\widetilde 
y_j(-\infty,\theta)\|\le \| \widetilde \x\|
\le C\|\eta\|_\delta$.
\item $(1-\beta_{R/2}(s-s_{v_j} 
-\ell_j + R))D_{G_\delta(\widetilde w)}
\big(\widetilde 
x_j(s-s_{v_j},\theta)-\widetilde x_j(+\infty,\theta)\big)$ 
and 
$(1-\beta_{R/2}(-s+s_{u_j}-R))D_{G_\delta(\widetilde w)}
\big(\widetilde y_j(s-s_{u_j},\theta)-\widetilde 
y_j(-\infty,\theta)\big)$. The parts 
involving $\widetilde 
x_j(+\infty,\theta)=\widetilde y_j(-\infty,\theta)$ are 
bounded by 
$C(\delta)\|\eta\|_\delta$ as above. On the other hand 
we write 
$$
D_{G_\delta(\widetilde w)}\widetilde x_j= 
(D_{G_\delta(\widetilde w)}-D_{\widetilde w^\delta})\widetilde x_j + 
D_{\widetilde w^\delta}\widetilde x_j
$$
and similarly for 
$D_{G_\delta(\widetilde w)}\widetilde y_j$. The first term  
of such 
a sum is bounded by $C(\delta)\|\eta\|_\delta$ as above, with 
$C(\delta)\to  0$, $\delta\to 0$. We are left with
\begin{eqnarray*}
(1-\beta_{R/2})D_{\widetilde w^\delta}
\widetilde 
x_j(s-s_{v_j},\theta) + 
(1-\beta_{R/2})D_{\widetilde 
w^\delta}
\widetilde y_j(s-s_{u_j},\theta)&& \\
=\ \Big((P\circ 
S)_{v_j}\eta\Big)(s-s_{v_j},\theta) + 
\Big((P\circ 
S)_{u_j}\eta\Big)(s-s_{u_j},\theta)&=& \eta.
\end{eqnarray*}
Here we 
denote by $(P\circ S)_{v_j}$, $(P\circ S)_{u_j}$ the components 
of $P\circ S$ in $L^{p,d}({v_j^\delta}^*T\widehat W)$ and 
$L^{p,d}({u_j^\delta}^*T\widehat W)$ respectively. 
The first equality uses the fact that $1-\beta_{R/2}\equiv 1$ on the
support of $(P\circ S)_{v_j}\eta$ and 
on the support of $(P\circ S)_{u_j}\eta$, as well as 
$D_{\widetilde w^\delta}\circ Q_{\widetilde w^\delta}=\one$.
\end{itemize}  
As a conclusion we have 
\begin{equation*} 
\|D_{G_\delta(\widetilde 
w)}T_\delta \eta - \eta \|_\delta   \le
C(\delta) \|\eta\|_\delta, \qquad C(\delta)\to 0, \ \delta\to 0,
\end{equation*}
and the estimate~\eqref{eq:Tdelta} holds for $\delta$ small enough.
\end{proof}

We shall use the following quantitative form of the implicit 
function theorem from 
McDuff and 
Salamon~\cite[A.3.4]{McDS}.

\begin{theorem} \label{thm:implicit}
Let $X$ and $Y$ be Banach spaces, $U\subset X$ be an open set, and 
$f:U\to Y$ be a continuously differentiable map. Let $x_0\in U$ be such 
that $D:=df(x_0):X\to Y$ is surjective and has a bounded right 
inverse $Q:Y\to X$. Choose positive constants $\eps$ and $c$ such 
that $\|Q\|\le c$, $B_\eps(x_0)\subset U$, and 
\begin{equation} 
\label{eq:implicit1}
  \|x-x_0\|<\eps \quad \Longrightarrow \quad 
\|df(x)-D\|\le 1/2c.
\end{equation}
 Then, for any $x_1\in X$ 
satisfying 
 \begin{equation} \label{eq:implicit2}
 \|f(x_1)\|<\eps/4c, \qquad \|x_1-x_0\|<\eps/8,
 \end{equation}
 there 
exists a unique $x\in X$ such that 
 \begin{equation} 
\label{eq:implicit3}
  f(x)=0,\quad x-x_1\in\im\, Q, \quad \|x-x_0\| 
\le \eps.
 \end{equation}
 Moreover, $\|x-x_1\|\le 
2c\|f(x_1)\|$.
\end{theorem}

The above theorem will be used within the following setup. 
Consider an element $[\u]\in \cM^A(p,q;H,\{f_\gamma\},J)$ and denote
$u_0:=G_{\delta,\oeps}(\u)$. 
Given $\eps>0$ we
denote by $B_\eps(0)$ the ball of radius $\eps$ centered at
$0$ in $W^{1,p}(\R\times S^1,u_0^*T\widehat W;\|\cdot\|_{1,\delta})$,
where $\|\cdot\|_{1,\delta}$ is defined by~\eqref{eq:norm1delta}. 
For $\zeta\in W^{1,p}(\R\times S^1,u_0^*T\widehat
W;\|\cdot\|_{1,\delta})$ we write
\begin{eqnarray*}
\zeta &= &\zeta_1 + \sum_{j=1}^m \okappa_j
\beta(-s+s_{u_j})\beta(s-s_{u_j}+2R)X_H  \\
&& + \ \sum_{j=1}^m \ukappa_j
\beta(s-s_{u_j}) \beta(-s+s_{u_j}+ 2R)  X_H  \\
&&+ \ \sum_{i=1}^{m-1} \kappa_i
\beta(s-s_{v_i}+\ell_i-2R) \beta(-s + s_{v_i} + \ell_i -2R)
\zeta_{i,\delta}(\cdot-s_{v_i},\cdot)
\end{eqnarray*}
with $\ell_i=R+(T_i+\epsilon_i)/2\delta$ and $\zeta_{i,\delta}$ the
generator of $\ker \, D_{v_i}$ whose value at $0$ is the vector field
$X_H$ along $\gamma_i$. Then 
$$
\|\zeta\|_{1,\delta}=\|\zeta_1\|_{W^{1,p}(g_{\delta,\oeps})} + \sum_{j=1}^m
(|\okappa_j|+|\ukappa_j|) + \sum_{i=1}^{m-1} |\kappa_i|.
$$
We denote 
$$
\tzeta := \zeta _1 + \sum_{j=1}^m \big(\okappa_j
\beta(-s+s_{u_j})\beta(s-s_{u_j}+2R)X_H + \ukappa_j
\beta(s-s_{u_j}) \beta(-s+s_{u_j}+ 2R)  X_H\big),
$$
so that $\tzeta(s_{v_i},\cdot)$ is $L^2$-orthogonal to
$\zeta_{i,\delta}(0,\cdot)$. For each $i=1,\dots,m-1$ 
we consider the smooth cutoff function 
$$
\rho_{i,\delta,\oeps}(s):= 
\beta(s-s_{v_i}+\ell_i-2R)
\beta(-s+s_{v_i}+\ell_i-2R), 
$$
so that $\rho_{i,\delta,\oeps}$ 
vanishes outside $[ s_{v_i} -\frac {T_i+\epsilon_i} {2\delta}
,s_{v_i} + \frac {T_i+\epsilon_i} {2\delta}]$ and 
$\rho_{i,\delta,\oeps}\equiv 1$ on the interval $[ s_{v_i} -\frac
{T_i+\epsilon_i} {2\delta} + 
1,s_{v_i} + \frac {T_i+\epsilon_i} {2\delta} -1]$. 

We define $\varphi_\zeta(u_0):\R\times
S^1 \to \widehat W$ by 
$$
\varphi_\zeta(u_0)(s,\theta):= \left\{\begin{array}{ll}
u_0(s,\theta),& s_{u_j}-R\le s \le s_{u_j}+R,\\
\varphi_{\rho_{i,\delta,\oeps}(s)\kappa_i}^{f_{\gamma_i}}
(u_0(s,\cdot))(\theta),
& s_{v_i} - 
\frac {T_i+\epsilon_i} {2\delta} \le s \le s_{v_i} + \frac {T_i+\epsilon_i}
{2\delta}. 
\end{array}\right.
$$
Note that the last formula can also be written in the chart
$(\vartheta,z)$ around $S_{\gamma_i}$ as $\vartheta\circ
\varphi_\zeta(u_0)(s,\theta) = \vartheta \circ
\varphi_{\rho_{i,\delta,\oeps}(s)\kappa_i}^{f_{\gamma_i}}(u_0(s,0))
+\theta$. Given a vector field $\xi$ along $u_0$ we define the vector
field $\varphi_{\zeta *}\xi$ along $\varphi_\zeta(u_0)$ by 
$$
\varphi_{\zeta *}\xi(s,\theta)  := \left\{\begin{array}{ll}
\xi(s,\theta),& s_{u_j}-R\le s \le s_{u_j}+R,\\
\varphi_{\rho_{i,\delta,\oeps}(s)\kappa_i *}^{f_{\gamma_i}}
(u_0)\xi(s,\theta),
& s_{v_i} - 
\frac {T_i+\epsilon_i} {2\delta} \le s \le s_{v_i} + \frac {T_i+\epsilon_i}
{2\delta}. 
\end{array}\right. 
$$

We define a map
\begin{equation} \label{eq:chart}
\Phi:B_\eps(0)\to
\cB_\delta=\cB_\delta^{1,p,d}(\og_p,\ug_q,A;H,\{f_\gamma\}) 
\end{equation}
by 
$$
\Phi(\zeta):=\exp_{\varphi_\zeta(u_0)}(\varphi_{\zeta *}\tzeta). 
$$
Since $\rho_{i,\delta,\oeps}$ is precisely the coefficient of
$\zeta_{i,\delta}$ in our splitting for $\zeta$, it follows that
$d\Phi(0)=\mathrm{Id}$. Hence, for $\eps>0$ small enough the map
$\Phi$ is a diffeomorphism onto its image, i.e. a chart.

We denote $X:=W^{1,p}(\R\times S^1,u_0^*T\widehat
W;\|\cdot\|_{1,\delta})$, $U:=B_\eps(0) \subset X$, $Y:=L^p(\R\times
S^1,u_0^*T\widehat W; g_{\delta,\oeps}dsd\theta)$, $x_0=0$. 
For $\eps>0$ small enough the Banach bundle $\cE\to
\cB_\delta$ can be trivialized over the image of $\Phi$ as
$B_\eps(0)\times Y$, and we denote by $f:B_\eps(0)\to Y$ the
section $\dbar_{H_\delta,J}\circ \Phi$ read
in this trivialization. Then $df(0)=D_{u_0}$ is surjective and has a
right inverse $Q_\delta$ whose $\delta$-norm is uniformly bounded with
respect to $\delta\to 0$ by Proposition~\ref{prop:gluing}. In order for
the hypotheses of Theorem~\ref{thm:implicit} to be satisfied we need
to check that~\eqref{eq:implicit1} holds.

\begin{lemma} \label{lem:2der}
There exists a constant $C > 0$ independent of $\delta$ such that,
for all $x \in B_\eps(0)$, we have
\begin{equation*} 
\| df(x) - df(0) \| \le C \| x \|_{1,\delta} .
\end{equation*}
\end{lemma}

\begin{remark}
The motivation for introducing the chart $\Phi$ is that we must use
the ``compensated'' norm $\|\cdot\|_{1,\delta}$. 
The lemma would fail if one used the usual exponential chart
$\zeta\mapsto \exp_{u_0}(\zeta)$ instead of $\Phi$, because the
estimate for the expression~\eqref{eq:crucial} in the proof below
would not hold. 
\end{remark}

\begin{proof} We need to prove the existence of a uniform constant
$C>0$ such that 
\begin{equation} \label{eq:der2}
\|D(\dbar_{H_\delta,J}\circ \Phi)(x)\cdot \zeta - D(\dbar_{H_\delta,J}\circ
\Phi)(0)\cdot \zeta\|_\delta \le C \|x\|_{1,\delta}
\|\zeta\|_{1,\delta}
\end{equation} 
for all $\zeta\in X$. We recall the decomposition
$\zeta=\tzeta+\sum_{i=1}^{m-1}\kappa_i \rho_{i,\delta,\oeps}
\zeta_{i,\delta}$, which satisfies
$\|\zeta\|_{1,\delta}=\|\tzeta\|_{1,\delta} +
\sum_{i=1}^{m-1}|\kappa_i|$. It is therefore enough to
prove~\eqref{eq:der2} separately for $\zeta=\tzeta$ and for $\zeta =
\rho_{i,\delta,\oeps} \zeta_{i,\delta}$, $i=1,\dots,m-1$. 
We abbreviate in the following computations $\dbar=\dbar_{H_\delta,J}$.

We first assume $\zeta = \rho_{i,\delta,\oeps}
\zeta_{i,\delta}$. Given $x=\tx + \sum_{j=1}^{m-1} \kappa_j
\rho_{j,\delta,\oeps}\zeta_{j,\delta}$ we have
\begin{eqnarray}
\lefteqn{D(\dbar \circ \Phi)(x)\zeta - D(\dbar \circ \Phi)(0)\zeta}
\nonumber \\
&=&D(\dbar \circ \Phi)(x)\zeta - D(\dbar \circ \Phi)(\sum_{j=1}^{m-1}
\kappa_j \rho_{j,\delta,\oeps}\zeta_{j,\delta})\zeta \label{eq:term1}\\
& & + D(\dbar \circ \Phi)(\sum_{j=1}^{m-1} \kappa_j
\rho_{j,\delta,\oeps}\zeta_{j,\delta})\zeta - D(\dbar \circ
\Phi)(0)\zeta \label{eq:term2}.
\end{eqnarray}
The term~\eqref{eq:term1} is further equal to 
\begin{eqnarray}
\lefteqn{\frac d {dt}\Big|_{t=0}
\dbar(\exp_{\varphi_{x+t\zeta}(u_0)}(\varphi_{x+t\zeta *}\tx)) -  
\frac d {dt}\Big|_{t=0} \dbar(\exp_{\varphi_{x+t\zeta}(u_0)}(0))} 
\nonumber \\
&=& \hspace{-2mm} 
D_{\exp_{\varphi_x(u_0)}(\varphi_{x *}\tx)} \cdot D_2
\exp_{\varphi_x(u_0)}(\varphi_{x *}\tx) \cdot
\nabla _t \varphi_{x+t\zeta *}\tx
\nonumber  \\
&& + \ D_{\exp_{\varphi_x(u_0)}(\varphi_{x *}\tx)} \cdot D_1
\exp_{\varphi_x(u_0)}(\varphi_{x *}\tx) 
\cdot \rho_{i,\delta,\oeps} \nabla f_{\gamma_i} (\varphi_x(u_0)) 
\nonumber  \\
&& - \ D_{\varphi_x(u_0)} \cdot D_1
\exp_{\varphi_x(u_0)}(0) \cdot \rho_{i,\delta,\oeps} \nabla
f_{\gamma_i} (\varphi_x(u_0)) \nonumber \\
&=& \hspace{-2mm}  \label{eq:term3}
D_{\exp_{\varphi_x(u_0)}(\varphi_{x *}\tx)} \cdot D_2
\exp_{\varphi_x(u_0)}(\varphi_{x *}\tx) \cdot
\nabla _t \varphi_{x+t\zeta *}\tx \\
&& + \  D_{\exp_{\varphi_x(u_0)}(\varphi_{x*}\tx)} \cdot (D_1
\exp_{\varphi_x(u_0)}(\varphi_{x*}\tx) - \cT\cdot D_1
\exp_{\varphi_x(u_0)}(0)) \nonumber \\
&& \hspace{5.5cm} \cdot 
\rho_{i,\delta,\oeps} \nabla f_{\gamma_i} (\varphi_x(u_0)) 
\nonumber \\
& & + \ 
(D_{\exp_{\varphi_x(u_0)}(\varphi_{x*}\tx)}\cdot \cT -
D_{\varphi_x(u_0)}) \cdot  
D_1 \exp_{\varphi_x(u_0)}(0) \cdot \rho_{i,\delta,\oeps} \nabla
f_{\gamma_i} (\varphi_x(u_0)).  \nonumber 
\end{eqnarray}
Here $\cT$ is the parallel transport in $\widehat W$ along the
geodesic $\tau\mapsto \exp_{\varphi_x(u_0)} (\tau \varphi_{x*}\tx)$,
$\tau\in [0,1]$, and we have $\rho_{i,\delta,\oeps} \nabla
f_{\gamma_i} (\varphi_x(u_0)) = \rho_{i,\delta,\oeps}
(\varphi^{f_{\gamma_i}}_{\rho_{i,\delta,\oeps}\kappa_i})_*
\zeta_{i,\delta}$. 

We study the first term in~\eqref{eq:term3}. We have pointwise bounds
$$
|\nabla _t \varphi_{x+t\zeta *}\tx|\le
C(1+|\kappa_i|)|\tx|,
$$
$$
|\nabla \nabla _t \varphi_{x+t\zeta *}\tx|\le
C(1+|\kappa_i|)(|\tx|+|\nabla\tx|)
$$
for some universal constant $C>0$. In particular 
$$
\|\nabla _t \varphi_{x+t\zeta *}\tx\|_{W^{1,p}(g_{\delta,\oeps})}\le
C\|\tx\|_{1,\delta}
$$
if $|\kappa_i|\le\|x\|_{1,\delta}\le \eps$, with $C>0$ a universal
constant. On the other hand the operators $D_2
\exp_{\varphi_x(u_0)}(\varphi_{x *}\tx):W^{1,p}(g_{\delta,\oeps})\to
W^{1,p}(g_{\delta,\oeps})$ and  $D_{\exp_{\varphi_x(u_0)}(\varphi_{x 
*}\tx)}:W^{1,p}(g_{\delta,\oeps})\to L^p(g_{\delta,\oeps})$ are
uniformly bounded if $\|x\|_\infty\le C\|x\|_{1,\delta}\le C\eps$
(we use here the Sobolev inequality). This implies that the
$\delta$-norm of the first term in~\eqref{eq:term3} is bounded by a
constant multiple of $\|\tx\|_{1,\delta}$.  

We now study the second term in~\eqref{eq:term3}. Let
$\tpar\cdot\tpar$ be the operator norm for 
continuous linear maps 
$$
W^{1,p}(\varphi_x(u_0)^*T\widehat W;
\|\cdot\|_{1,\delta}) \to
W^{1,p}(\exp_{\varphi_x(u_0)}(\varphi_{x*}\tx)^*T\widehat W;
g_{\delta,\oeps}ds d\theta). 
$$
We claim that $\tpar D_1 \exp_{\varphi_x(u_0)}(\varphi_{x*}\tx) - \cT\cdot D_1
\exp_{\varphi_x(u_0)}(0)\tpar \le C\|\tx\|_{1,\delta}$ for some uniform
constant $C>0$, provided $\|x\|_{1,\delta}\le\eps$. Indeed, since the
metric on $\widehat W$ varies smoothly, for any 
$\xi=\txi+\sum_{\ell=1}^{m-1}\kappa'_\ell
\rho_{\ell,\delta,\oeps}\nabla f_{\gamma_\ell}(\varphi_x(u_0))$ we
have pointwise bounds  
$$
\big|(D_1 \exp_{\varphi_x(u_0)}(\varphi_{x*}\tx) - \cT\cdot D_1
\exp_{\varphi_x(u_0)}(0))\txi\big| \le C |\tx||\txi|,
$$
$$
\big|\nabla(D_1 \exp_{\varphi_x(u_0)}(\varphi_{x*}\tx) - \cT\cdot D_1
\exp_{\varphi_x(u_0)}(0))\txi\big| \le C
(|\nabla\tx||\txi|+|\tx||\nabla\txi|), 
$$
$$
\big|(D_1 \exp_{\varphi_x(u_0)}(\varphi_{x*}\tx) -\cT\cdot D_1
\exp_{\varphi_x(u_0)}(0))\rho_{\ell,\delta,\oeps}\nabla
f_{\gamma_\ell}(\varphi_x(u_0)) \big|
\le C |\tx|, 
$$
$$
\big|\nabla(D_1 \exp_{\varphi_x(u_0)}(\varphi_{x*}\tx) -\cT\cdot D_1
\exp_{\varphi_x(u_0)}(0))\rho_{\ell,\delta,\oeps} \nabla
f_{\gamma_\ell}(\varphi_x(u_0))\big| \le C (|\tx|+|\nabla\tx|). 
$$
The claim then follows by integration with respect to the weight
$g_{\delta,\oeps}$ and by using the Sobolev inequalities
$\|\tx\|_{L^\infty}\le C\|\tx\|_{1,\delta}$ and
$\|\txi\|_{L^\infty}\le C\|\txi\|_{1,\delta}$.   
On the other hand, as already seen above, the operator
$D_{\exp_{\varphi_x(u_0)}(\varphi_{x*}\tx)}$ 
acting from the space
$W^{1,p}(g_{\delta,\oeps})$ to $L^p(g_{\delta,\oeps})$ is uniformly
bounded for $\|x\|_{1,\delta}\le \eps$, since its coefficients are
bounded. We infer that the $\delta$-norm of the second term
in~\eqref{eq:term3} is bounded by a constant multiple of
$\|\tx\|_{1,\delta}$. 

We finally study the third term in~\eqref{eq:term3}. We claim that 
$\|D_{\exp_{\varphi_x(u_0)}(\varphi_{x*}\tx)}\cdot \cT -
D_{\varphi_x(u_0)}\|\le C \|\tx\|_{1,\delta}$ for some uniform
constant $C>0$, provided $\|x\|_{1,\delta}\le\eps$. This follows from
the pointwise bounds  
$$
\big|(D_{\exp_{\varphi_x(u_0)}(\varphi_{x*}\tx)}\cdot \cT -
D_{\varphi_x(u_0)})\txi\big| 
\le C |\tx|(|\txi|+|\nabla \txi|),
$$
$$
\big|(D_{\exp_{\varphi_x(u_0)}(\varphi_{x*}\tx)}\cdot \cT -
D_{\varphi_x(u_0)})\rho_{\ell,\delta,\oeps}\nabla
f_{\gamma_\ell}(\varphi_x(u_0))\big|  \le C |\tx|
$$
by integrating with respect to the weight
$g_{\delta,\oeps}$ and by using the previous Sobolev inequalities. 
Since $D_1\exp_{\varphi_x(u_0)}(0)=\mathrm{Id}$, we infer that the
$\delta$-norm of the third term in~\eqref{eq:term3} is 
bounded by a constant multiple of $\|\tx\|_{1,\delta}$. 

As a conclusion, the $\delta$-norm of the expression
in~\eqref{eq:term1} is bounded by a constant multiple of
$\|\tx\|_{1,\delta}$. 

\medskip 

We now consider the expression in~\eqref{eq:term2}, which can be
written as
\begin{eqnarray}
\lefteqn{D(\dbar \circ \Phi)(\kappa_i
\rho_{i,\delta,\oeps}\zeta_{i,\delta})\zeta - D(\dbar \circ
\Phi)(0)\zeta}   \label{eq:crucial} \\
&=& \frac d {dt} \Big|_{t=0} \dbar(\varphi_{\kappa_i\zeta +
t\zeta}(u_0)) -\frac d {dt} \Big|_{t=0} \dbar(\varphi_{t\zeta}(u_0))
\nonumber \\
&=& \frac d {dt} \Big|_{t=0}
\dbar(u_0(\cdot+(\kappa_i+t)\rho_{i,\delta,\oeps},\cdot))
-  \frac d {dt} \Big|_{t=0}
\dbar(u_0(\cdot+t\rho_{i,\delta,\oeps},\cdot)). \nonumber 
\end{eqnarray}
Each term in the above difference is supported in the intervals
$[s_{v_i}-\frac {T_i+\epsilon_i} {2\delta},s_{v_i}-\frac
{T_i+\epsilon_i} {2\delta}+1]$ and $[s_{v_i}+\frac {T_i+\epsilon_i}
{2\delta}-1,s_{v_i}+\frac {T_i+\epsilon_i} {2\delta}]$. Moreover,
their difference is pointwise bounded by $C|\kappa_i|$ for some
uniform constant $C>0$. Since the weight $g_{\delta,\oeps}$ is
uniformly bounded on the above intervals of length $1$, we infer that
the $\delta$-norm of the expression in~\eqref{eq:term2} is bounded by
$C|\kappa_i|$, hence by $C\|x\|_{1,\delta}$ for some uniform constant
$C>0$. 

\medskip 

We now assume $\zeta=\tzeta$ and we again decompose $D(\dbar \circ
\Phi)(x)\zeta - D(\dbar \circ \Phi)(0)\zeta$ as the sum of the
expressions in~\eqref{eq:term1} and~\eqref{eq:term2}. 

The expression in~\eqref{eq:term1} can be written
\begin{eqnarray}
\lefteqn{\frac d {dt} \Big|_{t=0}
\dbar(\exp_{\varphi_x(u_0)}(\varphi_{x*} (\tx+t\tzeta)))
- \frac d {dt} \Big|_{t=0}
\dbar(\exp_{\varphi_x(u_0)}(\varphi_{x*}t\tzeta))} \nonumber \\
& = &  D_{\exp_{\varphi_x(u_0)}(\varphi_{x*}\tx)} \cdot 
D_2 \exp_{\varphi_x(u_0)}(\varphi_{x*}\tx) \cdot \varphi_{x*}\tzeta
\nonumber \\ 
& & \hspace{3cm} - \  
D_{\varphi_x(u_0)}\cdot D_2\exp_{\varphi_x(u_0)}(0) \cdot
\varphi_{x*}\tzeta \nonumber \\
&=& D_{\exp_{\varphi_x(u_0)}(\varphi_{x*}\tx)} \cdot 
(D_2 \exp_{\varphi_x(u_0)}(\varphi_{x*}\tx) - \cT \cdot
D_2\exp_{\varphi_x(u_0)}(0)) \cdot \varphi_{x*}\tzeta \nonumber \\
&& + \ (D_{\exp_{\varphi_x(u_0)}(\varphi_{x*}\tx)} \cdot \cT -
D_{\varphi_x(u_0)}) \cdot D_2\exp_{\varphi_x(u_0)}(0) \cdot
\varphi_{x*}\tzeta. \label{eq:term4}
\end{eqnarray}
Here $\cT$ denotes the same parallel transport map as above.
 
We claim that the $\delta$-norm of the first term in the
expression~\eqref{eq:term4}  
is bounded by $C\|\tx\|_{1,\delta}\|\tzeta\|_{1,\delta}$ when
$\|x\|_{1,\delta}\le\eps$, for some uniform constant $C>0$. 
We have the pointwise estimates
$$
|\varphi_{x*}\tzeta| \le C(1+\sum_{j=1}^{m-1}|\kappa_j|)|\tzeta|,
$$
$$
|\nabla\varphi_{x*}\tzeta| \le
C(1+\sum_{j=1}^{m-1}|\kappa_j|)(|\tzeta|+ |\nabla\tzeta|),
$$
which imply $\|\varphi_{x*}\tzeta\|_{W^{1,p}(g_{\delta,\oeps})} \le
C \|\tzeta\|_{1,\delta}$ for some uniform constant $C>0$, provided
$\|x\|_{1,\delta}\le\eps$. On the other hand, the pointwise estimates 
$$
\big|(D_2 \exp_{\varphi_x(u_0)}(\varphi_{x*}\tx) - \cT\cdot D_2
\exp_{\varphi_x(u_0)}(0))\xi\big| \le C |\tx||\xi|,
$$
$$
\big|\nabla(D_2 \exp_{\varphi_x(u_0)}(\varphi_{x*}\tx) - \cT\cdot D_2
\exp_{\varphi_x(u_0)}(0))\xi\big| \le C
(|\nabla\tx||\xi|+|\tx||\nabla\xi|) 
$$
show that the norm of the operator $D_2
\exp_{\varphi_x(u_0)}(\varphi_{x*}\tx) - 
\cT\cdot D_2 \exp_{\varphi_x(u_0)}(0)$ acting from
$W^{1,p}(g_{\delta,\oeps})$ to itself is bounded by
$C\|\tx\|_{1,\delta}$. Finally, we have already seen that the operator
$D_{\exp_{\varphi_x(u_0)}(\varphi_{x*}\tx)}$ acting between 
$W^{1,p}(g_{\delta,\oeps})$ and $L^p(g_{\delta,\oeps})$ is uniformly 
bounded, and the claim follows. 

We now claim that the $\delta$-norm of the second term in the
expression~\eqref{eq:term4}  
is also bounded by $C\|\tx\|_{1,\delta}\|\tzeta\|_{1,\delta}$ when
$\|x\|_{1,\delta}\le\eps$, for some uniform constant $C>0$. We have
the pointwise estimate 
$$
\big|(D_{\exp_{\varphi_x(u_0)}(\varphi_{x*}\tx)}\cdot \cT -
D_{\varphi_x(u_0)})\xi\big| 
\le C |\tx|(|\xi|+|\nabla \xi|),
$$
which implies that the norm of the operator 
$D_{\exp_{\varphi_x(u_0)}(\varphi_{x*}\tx)}\cdot \cT -
D_{\varphi_x(u_0)}$ acting from $W^{1,p}(g_{\delta,\oeps})$ to
$L^p(g_{\delta,\oeps})$ is bounded by $C\|\tx\|_{1,\delta}$ for some
uniform constant $C>0$. Since
$\|\varphi_{x*}\tzeta\|_{W^{1,p}(g_{\delta,\oeps})} \le C
\|\tzeta\|_{1,\delta}$ and $D_2\exp_{\varphi_x(u_0)}(0)=\mathrm{Id}$,
the claim follows.  

\medskip 

We finally study the term~\eqref{eq:term2} in the decomposition of 
$D(\dbar \circ \Phi)(x)\zeta - D(\dbar \circ \Phi)(0)\zeta$, which can
be written 
\begin{eqnarray*}
\lefteqn{\frac d{dt}\Big|_{t=0} \dbar(\exp_{\varphi_x(u_0)}\varphi_{x*}t\tzeta)
- \frac d{dt}\Big|_{t=0} \dbar (\exp_{u_0}t\tzeta)} \\
&=& D_{\varphi_x(u_0)}\cdot D_2\exp_{\varphi_x(u_0)}(0) \cdot
\varphi_{x*}\tzeta - D_{u_0}\cdot D_2\exp_{u_0}(0)\cdot \tzeta\\
&=& D_{\varphi_x(u_0)}\cdot\varphi_{x*}\tzeta - D_{u_0}\cdot\tzeta.
\end{eqnarray*}
This last expression is pointwise bounded by
$C(\sum_{j=1}^{m-1}|\kappa_j|)(|\tzeta |+|\nabla\tzeta|)$, which
implies that its $\delta$-norm is bounded by
$C\|x\|_{1,\delta}\|\tzeta\|_{1,\delta}$ for some uniform constant
$C>0$. 

This proves the lemma.
\end{proof}

\begin{proposition} \label{prop:family}
 Let $[\u]\in \cM^A(p,q;H,\{f_\gamma\},J)$. There exists 
$\delta_1>0$ and 
 a one-parameter family $[u_\delta]\in 
\cM^A(\og_p,\ug_q;H_\delta,J)$, $0<\delta <\delta_1$
 such that 
 $$
 [u_\delta]\to [\u], 
\quad \delta\to 0. 
 $$
 Here convergence is understood in the sense 
of Definition~\ref{def:convubar}.
 Moreover, if 
$\dim\,\cM^A(p,q;H,\{f_\gamma\},J)=0$ then 
the intermediate gradient fragments in $[\u]$ are nonconstant
 and the above one-parameter family is unique. 
\end{proposition}

\begin{remark} \label{rmk:whyT>0} {\rm 
The fact that the intermediate gradient fragments in $[\u]$ are
nonconstant is the reason why we had to prove the gluing theorem only
in the case where the intermediate lengths of gradient trajectories
are strictly positive: $T_i>0$, $i=1,\ldots,m-1$, where $m$ is the
number of sublevels in $[\u]$. 
}
\end{remark}

\begin{proof}
 We choose a representative $\u=(c_m,u_m,\ldots,u_1,c_0)$ of $[\u]$
 and we apply Theorem~\ref{thm:implicit} in a chart of $\cB_\delta$
 as above. By Proposition~\ref{prop:gluing} the operator $D$ admits a 
 right inverse $Q_\delta$ which is uniformly bounded 
 with respect to $\delta$ by some constant $c$. By
 Lemma~\ref{lem:2der} there exists 
$\eps>0$ independent of $\delta$ such that 
condition~\eqref{eq:implicit1} is satisfied. We set
$x_0:=G_\delta(\u)$. By 
Proposition~\ref{prop:catenation} we have 
 $$
 \lim_{\delta\to 0} \|f(x_0)\| =0
 $$
 and therefore condition~\eqref{eq:implicit2} is 
satisfied on some open neighbourhood of $x_0$ if $\delta$ is small 
enough. Taking $x_1:=x_0$ in the statement of 
Theorem~\ref{thm:implicit} provides us with an element $x\in X$ 
satisfying~\eqref{eq:implicit3}. We set 
 $$
 u_\delta:=x. 
 $$
Then 
$[u_\delta]\in \cM^A(\og_p,\ug_q;H_\delta,J)$. Because 
$\|x-x_0\|\le 2c\|f(x_0)\|\to 0$ and $x_0=G_\delta(\u)\to \u$ by
construction, we infer by Proposition~\ref{prop:deltageom} 
that $[u_\delta]\to [\u]$, $\delta\to 0$. 

\medskip 

We now assume that the dimension of 
$\cM^A(p,q;H, \{f_\gamma\} ,J)$ is zero. We have
\begin{eqnarray*} 
\lefteqn{\dim \cM^A(\og_p,\ug_q;H_\delta,J) = } \\ 
& = & \mu(\og_p)-\mu(\ug_q) + 2\langle c_1(TW), A \rangle -1 \\
& = & \mu(\og) - \mu(\ug) + 2\langle c_1(TW), A \rangle -1 +\ind(p) -
\ind(q), 
\end{eqnarray*}
hence $\mu(\og) - \mu(\ug) + 2\langle c_1(TW), A \rangle =1 -\ind(p) +
\ind(q)\le 2$.  
On the other hand $\mu(\og) - \mu(\ug) + 2\langle c_1(TW), A \rangle =
\sum_{i=1}^m \mu(\gamma_i) - \mu(\gamma_{i-1}) + 2\langle c_1(TW), A_i \rangle$, 
where $m\ge 0$ is the number of sublevels of $\u$.   
Each of the summands is nonnegative by transversality, and the only possibilities 
occuring are the following:
\begin{enum}
\item each summand is zero, which means that all the Floer trajectories involved in $\u$ are rigid;
\item one of the summands is $1$ and the others vanish. Since $[\u]$ is rigid, the only 
nonrigid summand must be $u_0$ or $u_m$, while $c_0$, respectively $c_m$ have to be constant.;
\item two of the summands are $1$, and the others vanish. As above, the nonrigid summands must be 
$u_0$ and $u_m$, while $c_0$ and $c_m$ are constant;
\item one of the summands is $2$ and the others vanish. Since $[\u]$ is rigid we must have $m=1$ and $c_0$, $c_1$ have to be constant.
\end{enum}
In each of the cases (i-iii) the intermediate gradient trajectories have to be nonconstant by transversality of the evaluation maps~\eqref{eq:transvf}.

Let now $[\widetilde v_\delta]\to [\u]$, $\delta\to 0$. Since the
only possible intermediate gradient trajectory in $[\u]$ is
nonconstant, we can apply Proposition~\ref{prop:deltageom}. We obtain
representatives $v_\delta\in[\widetilde v_\delta]$, $\v\in[\u]$ and
functions $\oeps=\oeps(\delta)=\!(\epsilon_1(\delta), \ldots,
\epsilon_{m-1}(\delta))$ such that $v_\delta$, $G_{\delta,\oeps}(\v)$
belong to some $\|\cdot\|_{1,\delta,\oeps}\ $-chart in $\cB_\delta$ and 
\begin{equation} \label{eq:t=1}
\| v_\delta - G_{\delta,\oeps}(\v)\|_{1,\delta,\oeps}\to 0, \quad
\delta\to 0.
\end{equation}
We have to prove that $u_\delta$ and $v_\delta$ differ by a shift for
$\delta>0$ sufficiently small. 

Let us choose a continuous path $\v_t\in \widehat \cM^A(p,q;H,
\{f_\gamma\} ,J)$, $t\in [0,1]$ with $\v_0=\u$, $\v_1=\v$. We denote
$y_t=y_t(\delta):=G_{\delta,t\oeps}(\v_t)$ and
$\|\cdot\|_t:=\|\cdot\|_{1,\delta,t\oeps}$.
Note that, for each $\delta>0$, there exists a continuous function
$C_\delta:[0,1]\times [0,1]\to\R^+$ such that 
$$
\frac 1 {C_\delta(t,t')} \|\cdot\|_{t'} \le \|\cdot\|_t \le
C_\delta(t,t')\|\cdot\|_{t'}, \quad  t,t'\in[0,1]
$$
satisfying 
$C_\delta(t,t)=1$, $t\in[0,1]$. 
 
As $y_t(\delta)$ and $D_{y_t(\delta)}$ vary continuously with $t$ and
$\delta$, we can choose $\eps>0$ and $c>0$ so that the hypotheses
of the implicit function theorem~\ref{thm:implicit} are satisfied for
each $y_t(\delta)$, $t\in[0,1]$, $0<\delta<\delta_1$ and some suitable
constant $\delta_1>0$. After further shrinking $\delta_1$ we can also
assume that $\|f(y_t)\|_t<\eps/4c$ for all $t\in [0,1]$ and
$0<\delta<\delta_1$. Finally, in view of~\eqref{eq:t=1}, for some
smaller $\delta_1$ we can achieve $\|v_\delta - y_1(\delta)\|_1\le
\eps$. 

We define 
$$
I_\delta:=\{t\in [0,1] \, : \, \exists \, x\in[v_\delta], \ 
\|x-y_t(\delta)\|_t\le\eps\}, \qquad 0<\delta<\delta_1.
$$
We prove that $I_\delta=[0,1]$ by showing that it is a nonempty open
and closed subset of $[0,1]$. Note that $I_\delta$ is nonempty since
$1\in I_\delta$. We prove now that $I_\delta$ is closed. Assume that
$t_n\in I_\delta$ is such that $t_n \to t$. Let $x_n\in[v_\delta]$ be
such that $\|x_n-y_{t_n}(\delta)\|_{t_n}\le \eps$. By the
triangular inequality we see that $\|x_n-y_t(\delta)\|_t$ stays
bounded, hence the sequence of shifts defining $x_n$ is also bounded
and, up to a subsequence, we may assume that $x_n\to
x\in[v_\delta]$. Then 
\begin{eqnarray*}
 \|x_n - y_t(\delta)\|_t & \le & C_\delta(t,t_n)\|x_n -
 y_t(\delta)\|_{t_n} \\
& \le & C_\delta(t,t_n)(\|x_n - y_{t_n}(\delta)\|_{t_n} + 
\|y_{t_n}(\delta)-y_t(\delta)\|_{t_n}) \\
& \le & C_\delta(t,t_n)(\eps + 
\|y_{t_n}(\delta)-y_t(\delta)\|_{t_n}). 
\end{eqnarray*}
We pass to the limit $n\to \infty$ and obtain $\|x -
y_t(\delta)\|_t\le \eps$, hence $t\in I_\delta$. We prove now 
that $I_\delta$ is open. Let $t\in I_\delta$ and choose an open
interval $J$ containing $t$ such that $C_\delta(t',t)<2$ and
$\|y_{t'}(\delta)-y_t(\delta)\|_t < \eps/8$ for all $t'\in J$.  
Theorem~\ref{thm:implicit} applied to $x_0:=y_t(\delta)$ and 
$x_1:=y_{t'}(\delta)$ yields $x$ such that $f(x)=0$ and
$\|x-y_t(\delta)\|_t\le \eps$. The uniqueness statement in the
implicit function theorem ensures that the intersection of the space
of solutions with the $\|\cdot\|_t$-ball of radius $\eps$ centered
at $x_0$ is a graph over $\ker \, D$. Since   
$\dim\ker\,D=1$ and since translation in the $s$-variable already 
provides a $1$-parameter family of solutions, we infer that
$x\in[v_\delta]$. Moreover, the last statement in
Theorem~\ref{thm:implicit} gives $\|x-y_{t'}(\delta)\|_t\le
\eps/2$. Then $\|x-y_{t'}(\delta)\|_{t'}\le
C_\delta(t',t)\eps/2 < \eps$, so that $t'\in I_\delta$ and
$I_\delta$ is open. 

The upshot is that there exists $x\in [v_\delta]$ such that 
$\|x-y_0(\delta)\|_0\le\eps$, $0<\delta<\delta_1$. But
$y_0(\delta)= G_\delta(\u)$ and, again by the uniqueness statement in
the implicit function theorem, we get that $x$ and $u_\delta$ differ
by a shift. Hence $[u_\delta]=[v_\delta]$. 
\end{proof}

\begin{proof}[\bf Proof of Theorem~\ref{thm:degen}]
We first prove (i) and show the existence of $\delta_1$. 
Assume by contradiction that there exists a
sequence $\delta_n\to 0$ and Floer trajectories $v_n\in \widehat
\cM^A(\og_p,\ug_q;H_{\delta_n},J)$ such that $J$ is not regular for
$v_n$. By Proposition~\ref{prop:compact} we may assume, up to shifting
and passing to a subsequence, that $v_n\to \u\in\widehat
\cM^A(p,q;H,\{f_\gamma\},J)$. As seen in the proof of
Proposition~\ref{prop:family}, the limit $\u$ has nonconstant
intermediate gradient trajectories since $J$ is regular for $\u$. 
We can therefore apply
Proposition~\ref{prop:deltageom} and get parameters $\oeps^n$ and
vector fields $\zeta_n$ such that
$v_n=\exp_{G_{\delta_n,\oeps^n}(\u)}(\zeta_n)$ and
$\|\zeta_n\|_{1,\delta_n,\oeps^n} \to 0$. By
Proposition~\ref{prop:gluing} the operator
$D_{G_{\delta_n,\oeps^n}(\u)}$ is surjective and admits a right
inverse which is uniformly bounded with respect to $\delta$. We infer
that the operator $D_{v_n}$ is also surjective for $n$ large enough, a
contradiction. 

Let us prove (ii). Let $(\delta,v_\delta) \in \widehat
\cM^A_{]0,\delta_1[}(\og_p,\ug_q; H,\{f_\gamma\},J)$ and let
$I(\delta) \subset ]0,\delta_1[$ be a small relatively compact open
interval containing $\delta$. Since the norms $\|\cdot\|_{1,\delta'}$
are equivalent for $\delta'\in I(\delta)$, the space
$\cB_{I(\delta)}:=\bigcup_{\delta'\in I(\delta)} \{\delta'\} \times
\cB_{\delta'}$ is a Banach manifold. Similarly, there is a Banach
vector bundle $\cE_{I(\delta)}\to \cB_{I(\delta)}$ endowed with an
obvious section $\dbar_{H_{I(\delta)},J}$ whose restriction to
$\cB_{\delta'}$ is $\dbar_{H_{\delta'},J}$. The restriction of its
linearization $D_{(\delta,v_\delta)}$ at $(\delta,v_\delta)$ to
$T_{v_\delta}\cB_\delta$ is the surjective operator $D_{v_\delta}$ of
index $1$, hence $D_{(\delta,v_\delta)}$ is surjective and has index
$2$. Therefore $\ker \, D_{(\delta,v_\delta)}$ projects surjectively
onto $T_\delta I(\delta)=\R$ and the projection in (ii) is a submersion. 

We now prove (iii). Let us note that, by
Proposition~\ref{prop:family}, we have a map 
\begin{eqnarray*}
\cM^A(p,q;H,\{f_\gamma\},J) & \to & \pi_0
(\cM^A_{]0,\delta_1[}(\og_p,\ug_p;H,\{f_\gamma\},J)), \\
{}[\u] & \mapsto & C_{[\u]}:= \bigcup _{\delta\in ]0,\delta_1[}
\{(\delta,[u_\delta])\},  
\end{eqnarray*}
where $[u_\delta]$ is the uniquely defined one-parameter family of
Proposition~\ref{prop:family} such that $[u_\delta]\to[\u]$. This map
is injective because the limit of such a family $[u_\delta]$ as
$\delta\to 0$ is unique. In order to prove surjectivity, let
$C=\{(\delta,[v_\delta])\}$ be a connected component of
$\cM^A_{]0,\delta_1[}(\og_p,\ug_p;H,\{f_\gamma\},J)$. By
Proposition~\ref{prop:compact} there exists a sequence $\delta_n\to 0$
and $[\u] \in \cM^A(p,q;H,\{f_\gamma\},J)$ such that
$[v_{\delta_n}]\to [\u]$. By the uniqueness statement in
Proposition~\ref{prop:family} we get that $C=C_{[\u]}$. 
\end{proof}

\subsection{Coherent orientations} \label{sec:ori}

The structure of this section is as follows. We first present the
construction of coherent orientations in the usual Floer setting for
$(H_\delta,J)$ by adopting the point of view of~\cite{BM}. 
We construct coherent orientations on
the moduli spaces of Morse-Bott trajectories, out of which we get
orientations on the space of Morse-Bott trajectories with gradient
fragments. Finally, we prove Proposition~\ref{prop:signs}.

We denote $S^1:=\R/\Z$ and, for a path of symmetric matrices $S:S^1\to
M_{2n}(\R)$, we denote by $\Psi_S$ the unique solution of the Cauchy
problem  
\begin{equation} \label{eq:PsiS}
\dot \Psi (\theta) =
J_0 S(\theta) \Psi(\theta), \quad \Psi(0)= \one, \quad \theta\in
[0,1],
\end{equation} 
where $J_0$ is the standard complex structure on $\R^{2n}$. 
Then $\Psi_S$ is a path of symplectic matrices and we denote 
$$
\cS:=\{S:S^1\to M_{2n}(\R) \, : \, ^tS=S \mbox{ and }
\det(\one-\Psi_S(1))\neq 0\}.
\index{$\cS$, loops of symmetric matrices without degenerate directions}
$$

Let us denote by $E$ a symplectic vector bundle of rank $2n$ over $\C
P^1$, or $\R\times S^1$, or $\C$, with fixed trivializations in
neighbourhoods of infinity in the case of $\R\times S^1$ and $\C$. 
We denote by\index{$\cO$, space of CR operators|(} 
$$
\cO(\C P^1,E)
$$ 
the space of linear operators $D:W^{1,p}(\C
P^1,E)\to L^p(\C P^1,\Lambda^{0,1}E)$ of the form $(\p_x + J_0 \p_y +
S(z))d\bar z$ in a local trivialization of $E$, where $z=x+iy$ is a
local coordinate on $\C P^1$. Given $\oS,\uS\in\cS$ we denote by 
$$
\cO(\R\times S^1,E;\oS,\uS)
$$ 
the space of linear operators $D:W^{1,p}(\R\times
S^1,E)\to L^p(\R\times S^1,\Lambda^{0,1}E)$ of the form $(\p_s + J_0 \p_\theta +
S(s,\theta))(ds-id\theta)$ in a local trivialization of $E$, such that
$\lim_{s\to-\infty} S(s,\cdot) = \oS$ and $\lim_{s\to\infty}
S(s,\cdot)=\uS$ in the given trivializations of $E$. Given $S_0\in\cS$
we denote by 
$$
\cO_\pm(\C,E;S_0)
$$
the space of linear operators $D:W^{1,p}(\C,E)\to
L^p(\C,\Lambda^{0,1}E)$ of the 
form $(\p_x + J_0 \p_y + S(z))d\bar z$ in a local trivialization of
$E$ and such 
that, when expressed in holomorphic cylindrical coordinates
$(s,\theta)$ with $e^{\pm 2\pi(s+i\theta)}=z$ as 
$(\p_s+J_0\p_\theta+S(s,\theta))(ds-id\theta)$, we
have $\lim_{s\to \pm\infty} S(s,\theta) = S_0(\theta)$  
in the given trivialization of $E$. Intuitively, the space $\cO_+$
corresponds to the sphere with one positive puncture, while $\cO_-$
\index{$\cO$, space of CR operators|)} 
corresponds to the sphere with one negative puncture.  

It is a standard fact in the literature that each of the above spaces
$\cO$ is contractible and consists of Fredholm operators. Moreover, they
each come equipped with a canonical real line bundle $\Det(\cO)$ whose
fiber at $D$ is $\Det(D):=(\Lambda^{\max} \ker \, D) \otimes
(\Lambda^{\max} \coker \, D)^*$. Each of the bundles $\Det(\cO)$ is
trivial since the base is contractible. 

We now define gluing operations between elements of the above
spaces (see Figure~\ref{fig:ori}). Let $K\in \cO_+(\C,E;S_0)$ or
$K\in\cO(\R\times S^1,E;\oS,S_0)$, 
and $L\in \cO_-(\C,F;S_0)$ or $L\in\cO(\R\times S^1,F;S_0,\uS)$. 
Let us choose a cutoff function $\beta:\R\to [0,1]$ 
such that\index{$\beta$, cutoff function} 
$\beta(s)=0$ if $s\le 0$ and $\beta(s)=1$ if $s\ge 1$. 
Given $R>0$ large we define operators $K_R$ and $L_R$ by replacing $S$
in the asymptotic expressions of $K$ and $L$ by
$S_0+\beta(R-s)(S-S_0)$ and $S_0+\beta(R+s)(S-S_0)$ respectively. 
We cut out semi-infinite cylinders $\{s>R\}$ from the base of
$E$, $\{s<-R\}$ from the base of $F$, then identify their boundaries
using the coordinate $\theta$. We glue the vector bundles $E$ and $F$
using their given trivializations near infinity and denote the
resulting vector bundle by $E\# F$. We define $K\#_R L$ by
concatenating $K_R$ and $L_R$, so that $K\#_R L$ belongs to one of the
spaces $\cO(\C P^1,E\# F)$, $\cO_+(\C,E\# F; \uS)$, $\cO_-(\C,E\#
F;\oS)$, or $\cO(\R\times S^1,E\# F; \oS,\uS)$. 

\begin{figure}
         \begin{center}
\input{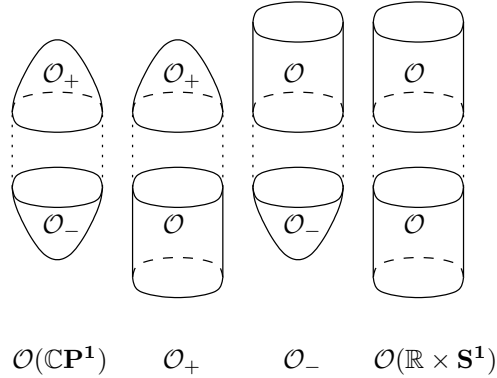}
\caption{The four possibilities of gluing ($\cO=\cO(\R\times S^1)$).
 \label{fig:ori}}
         \end{center}
\end{figure}

Following~\cite[Corollary~7]{BM}, for $R$ large enough there is a
natural isomorphism $\Det(K)\otimes \Det(L) \stackrel \sim \to
\Det(K\#_R L)$ defined up to homotopy. In particular, given
orientations $o_K$ of $\Det(K)$ and $o_L$ of $\Det(L)$, we induce a
canonical orientation $o_K\# o_L$ of $\Det(K\#_R L)$. Moreover, this
operation on orientations is associative~\cite[Theorem~10]{FH}. 

\medskip 

We describe now, following~\cite{BM}, a procedure for constructing
orientations on the spaces $\cO(\R\times S^1,E;\oS,\uS)$ which are
coherent with respect to the gluing operation, in the sense
of~\cite[Definition~11]{FH}. We denote by $\theta_n$ a trivial
symplectic vector bundle of rank $2n$. We first note that
each determinant bundle 
$\Det(\cO(\C P^1,E))$ is naturally oriented since $\cO(\C P^1,E)$ 
contains the connected space of complex linear operators and the latter
have kernels and cokernels which are canonically oriented as complex
vector spaces. We now choose arbitrary orientations of the determinant
bundles $\Det(\cO_+(\C,E;S_0))$ such that the trivialization of $E$ at
infinity extends to $\C$. 

\begin{remark} \label{rmk:Clin}
Note that, if $S_0$ commutes with $J_0$,
the set of $\C$-linear operators in $\cO_+(\C,E;S_0)$ forms a nonempty
convex set, hence $\Det(\cO_+(\C,E;S_0))$ has a canonical orientation. 
\end{remark} 

\noindent We induce orientations on the determinant
bundles $\Det(\cO_-(\C,E;S_0))$ such that the trivialization of $E$ at
infinity extends to $\C$ by requiring that the orientation induced by
gluing on $\Det(\cO(\C P^1,\theta_n))$ is the canonical one. Finally,
we induce orientations on $\Det(\cO(\R\times S^1,E;\oS,\uS))$ by
requiring that the orientation induced on $\Det(\cO(\C P^1,\theta_n
\# E \# \theta_n))$ by the gluing operation 
$$
\cO_+(\C,\theta_n;\oS)\times \cO(\R\times S^1,E;\oS,\uS)\times
\cO_-(\C,\theta_n;\uS) \to \cO(\C P^1,\theta_n \# E \# \theta_n)
$$
is the canonical one. It is proved in~\cite{BM} that this defines a
system of coherent orientations. 

The general procedure for inducing orientations of the spaces
of Floer trajectories $\widehat \cM^A(\og_p,\ug_q;H_\delta,J)$ out of
a system 
of coherent orientations goes as follows. Let $\overline
\Psi_p$, $\underline \Psi_q$ denote the linearizations of the
Hamiltonian flow of $H_\delta$ along $\og_p$, $\ug_q$ in their fixed
respective trivializations and let $\oS_p,\uS_q\in\cS$ be the
corresponding paths of symmetric matrices as in~\eqref{eq:PsiS}. Let
$E$ be a symplectic vector bundle over $\R\times S^1$ with fixed
trivializations at infinity and relative first Chern class equal to
$\langle c_1(T\widehat W),A\rangle$. For each $u\in \widehat
\cM^A(\og_p,\ug_q;H_\delta,J)$ there is an isomorphism of symplectic
vector bundles $\Phi_u:u^*T\widehat W\stackrel \sim \to E$,
chosen to depend continuously on $u$. There is a map 
$$
\widehat \cM^A(\og_p,\ug_q;H_\delta,J) \to \cO(\R\times
S^1,E;\oS_p,\uS_q), \quad u\mapsto \Phi_u\circ \widetilde D_u
\circ \Phi_u^{-1},
$$
where $\widetilde D_u$ has the same analytical expression as the
linearized operator 
$D_u:W^{1,p}(\R\times S^1,u^*T\widehat W;e^{d|s|}ds\,
d\theta)\oplus \oV_u \oplus \uV_u \to L^p(\R\times S^1,u^*T\widehat
W;e^{d|s|}ds\, d\theta)$ considered in Section~\ref{sec:gluing}. 
Under the assumption that $J$ is regular and because of elliptic
regularity, the operators $\widetilde D_u$ and $D_u$ have the same
kernel, consisting of smooth elements. Hence their determinant lines
are naturally isomorphic. It follows that the pull-back of $\Det(\cO)$
under the above map is naturally isomorphic to $\Lambda^{\max}T\widehat
\cM = \Lambda^{\max}\ker \, D_u$, and we get an orientation on
$\widehat \cM$.  

If the dimension is one, the space $\widehat \cM$ has a canonical
orientation given at each point $u$ by the vector field $\p_s
u$. Comparing this with the orientation constructed above associates
to each connected component $[u]$ of $\widehat \cM$ a sign
$\epsilon(u)$. 

\begin{lemma} \label{lem:cover}
 Let $S_1\in \cS$ and define $S_m(\theta):=S_1(m\theta)$. Assume
 $S_m\in\cS$ and define an
 automorphism $\phi_m$ of $\cO_+(\C,E;S_m)$ 
 by conjugation with the map $z\mapsto e^{2i\pi/m} z$. Then $\phi_m$
 is orientation reversing for $\Det(\cO_+(\C,E;S_m))$ if and only if 
 $m$ is even and the difference of Conley-Zehnder indices
 $\mu_{CZ}(\Psi_{S_m})-\mu_{CZ}(\Psi_{S_1})$ is odd.  
\end{lemma} 

\begin{proof}
 We start by explaining how
 $\phi_m$ acts on the orientations of the determinant bundle. The
 operators $K\in \cO_+(\C,E;S_m)$ which are 
 invariant under conjugation by $\phi_m$, i.e. $K(\zeta\circ
 \phi_m)\circ \phi_m^{-1}=K(\zeta)$ for all $\zeta$, form a convex
 and, in particular, connected set. Since $\phi_m$ acts on $\ker\, K$
 and $\coker \, K$, it also 
 acts on $\Det(K)$ and the induced action on orientations extends to
 $\Det(\cO_+(\C,E;S_m))$.

 There is a bijective correspondence between operators
 $K_1\in\cO_+(\C,E;S_1)$ and operators $K\in \cO_+(\C,E;S_m)$ which are
 invariant under conjugation by $\phi_m$, in which case the pull-back
 of $\ker \,  K_1$ under $z\mapsto z^m$ is the $1$-eigenspace of
 $\phi_m$ acting on $\ker\, K$. Since $\ker \, K$ splits as a
 direct sum of eigenspaces corresponding to the $m$-th roots of
 unity and since imaginary roots give rise to even-dimensional
 eigenspaces, we infer that the dimension of the $-1$-eigenspace
 has the parity of $\dim\ker\, K - \dim\ker\, K_1$. This fact is
 relevant in our situation since $\phi_m$
 reverses the orientation of $\ker\, K$ if and only if this dimension
 is odd. Similarly, $\phi_m$ reverses the orientation of $\coker\, K$
 if and only if $\dim\coker\, K - \dim\coker\, K_1$ is odd. As a
 conclusion, $\phi_m$ reverses the orientation of $\Det(K)$ if and only if
 $\ind(K)-\ind(K_1)=\mu_{CZ}(\Psi_{S_m})-\mu_{CZ}(\Psi_{S_1})$ is
 odd. This can happen of course only if $-1$ is an $m$-th root of
 unity, i.e. $m$ is even. 
\end{proof} 

\begin{remark} \label{rmk:modd} 
 The proof of Lemma~\ref{lem:cover} shows that, if $m$ is odd, the
 difference of Conley-Zehnder indices is automatically even.
\end{remark}

\begin{lemma} \label{lem:T}
 Let $S_1\in \cS$, define $S_m(\theta):=S_1(m\theta)$ and assume
 $S_m\in\cS$. Let $T\in\cO(\R\times
 S^1,\theta_n;S_m,S_m)$ be an element of the form 
$$
 T:= \p_s + J_0\p_\theta + S_m(\theta-\beta(s)/m),
$$
with $\beta:\R\to [0,1]$ a smooth function  
satisfying $\beta(s)=0$ near $-\infty$, $\beta(s)=1$ near $+\infty$ and with
derivative uniformly bounded by some small constant $c$. We denote by
$\cO$ one of the spaces $\cO_+(\C,E;S_m)$ or $\cO(\R\times
S^1,E;S,S_m)$, $S\in\cS$. The family $\psi=\{\psi_R\}$,
$R>0$ of automorphisms of $\cO$ defined by  
$$
 \psi_R(D):=D\#_R T
$$ 
induces an action on the orientations of $\Det(\cO)$
which is reversing if and
only if $m$ is even and the difference of Conley-Zehnder indices
 $\mu_{CZ}(\Psi_{S_m})-\mu_{CZ}(\Psi_{S_1})$ is odd. 
\end{lemma} 

\begin{proof}
 Note that $T$ is an isomorphism if $c$ is small
 enough, by the same argument as the one for $D_u''$ in the proof of
 Proposition~\ref{prop:Surjectivity_udelta}. 

 We now explain what is the action of $\psi$ on the orientations
 of $\Det(\cO)$. Let $D\in\cO$ and let $V\subset L^p$ be a finite dimensional
 vector space spanned by smooth sections with compact support, such
 that $V+\im\, D=L^p$. We define the stabilization of $D$ by $V$ as 
$$
D^V:V\oplus W^{1,p}\to L^p, \qquad (v,\zeta)\mapsto v+D\zeta.
$$
 Then $D^V$ is a surjective Fredholm operator and there is a canonical
 isomorphism $\Det(D)\simeq \Lambda^{\max} \ker\, D^V \otimes
 \Lambda^{\max} V^*$. For $R$ large enough the glued operator 
 $D_R=D^V\#_R T :V\oplus W^{1,p}\to L^p$ is surjective with a uniformly
 bounded right inverse $Q_R$, and moreover the projection onto $\ker
 \, D_R$ given by $\one-Q_R D_R$ is an isomorphism when restricted to
 $\ker \, D^V$ (see~\cite[Corollary~6]{BM}, as well
 as~\cite[Proposition~9]{FH} for a slightly different setup). 
 Since $D^V\#_R T = (D\#_R T)^V$, this induces a
 natural isomorphism between $\Det(D)$ and $\Det(\psi_R(D))$. 
 
 The gluing of orientations is associative, hence it is enough to
 prove the statement for $\cO=\cO_+(\C,\theta_n;S_m)$. We claim that
 the action induced by $\psi$ is the same as the one induced by
 $\phi_m$ in Lemma~\ref{lem:cover}. Let us choose $D\in\cO$ which is
 $s$-independent for $s$ large enough, and let $D^V$ be a 
 surjective stabilization. We construct a continuous path in
 $\cO$ from $\psi_R(D^V):=\psi_R(D)^V$ to
 $\phi_m(D^V):=\phi_m(D)^V$ as follows. Let $D^V_t$ be the conjugation
 of $D^V$  by $r_t:\C\to\C$, $z\mapsto e^{-2i\pi t/m}z$, and let $T_t$
 be the operator 
 $\p_s+J_0\p_\theta+S_m(\theta-(t+(1-t)\beta(s))/m)$. Then $D^V_t\#_R
 T_t$ interpolates between $\psi_R(D^V)$ and $\phi_m(D^V)$ as $t$
 varies from $0$ to $1$. This is a path of surjective operators
 admitting a continuous family of right inverses $Q_t$. Given a basis
 $(\zeta_1,\ldots,\zeta_k)$ of 
 $\ker\,D^V$, a basis of $\ker\, D^V_t$ is given by
 $(\zeta_1,\ldots,\zeta_k)\circ r_t$. By projecting along $\im\, Q_t$
 we obtain a basis of $\ker\, D^V_t\#_R T_t$. For $t=1$, since
 $D^V_1=D^V_1\#_R T_1$, the elements $\zeta_i\circ r_1$ are preserved
 by the projection and form a basis of $\phi_m^{-1}(D^V)$ which is exactly
 the one giving the action of $\phi_m^{-1}$ (or $\phi_m$) on
 orientations, as explained in Lemma~\ref{lem:cover}. 
\end{proof} 

\begin{lemma} \label{lem:signsgrad}
 Let $\gamma \in \cP_\lambda^{\le \alpha}$ and $\gamma_p$, $\gamma_q$ be 
 the orbits corresponding to the minimum $p$ and maximum $q$ of $f_\gamma$ 
 respectively. For $\delta >0$ small enough, the moduli space
 $\cM^A(\gamma_p,\gamma_q;H_\delta,J)$ is empty if $A\neq 0$, while for
 $A=0$ it consists of exactly two elements $u_1$, $u_2$ corresponding
 to the two 
 gradient trajectories of $f_\gamma$ running from $p$ to
 $q$. Moreover, they satisfy 
 $$
 \epsilon(u_1) + \epsilon(u_2) = 
   \left\{\begin{array}{rl} 0, &  \mbox{ if } \gamma \mbox{ is a good
orbit}, \\ 
                             \pm 2, &  \mbox{ if } \gamma \mbox{ is a
bad orbit}. 
   \end{array}\right.
 $$
\end{lemma}

\begin{proof} Let $c_1$, $c_2$ be the gradient trajectories of $f_\gamma$ 
 running from $p$ to $q$. By Theorem~\ref{thm:degen}, for
 $\delta>0$ small enough each element
 $[u_\delta]\in\cM^A(\gamma_p,\gamma_q;H_\delta,J)$ corresponds to a
unique Floer trajectory with gradient fragments $[\u]$ whose endpoints are
 $p$ and $q$. For energy reasons there can be no nonconstant Floer
 trajectory involved in $[\u]$ and therefore $[\u]$ is either $c_1$ or
 $c_2$. Since the cylinders $u_1$ and $u_2$ of the form
 $u_{\delta,\og,-\infty,+\infty}$ associated to $c_1$ and $c_2$ are
 already Floer trajectories for $H_\delta$, we infer that $[u_\delta]$
 equals either $[u_1]$ or $[u_2]$, and the homology class $A$ is
 necessarily zero. 
 Let us introduce the notation $\epsilon(\gamma):=1$ if $\gamma$ is a
 good orbit and $\epsilon(\gamma):=-1$ if $\gamma$ is a bad orbit. The
 conclusion of the Lemma is equivalent to the relation 
\begin{equation} \label{eq:signs}
\epsilon(u_1) = -\epsilon(\gamma)\epsilon(u_2).
\end{equation}

 Let us choose a symplectic trivialization $\Phi_\gamma:T\widehat
 W|_{S_\gamma}\to S_\gamma\times (\R\times \R^{2n-1})$ such that
 $\Phi_\gamma(X_H)=(1,0)$. We assume without loss of generality that $\dot
 c_1$ is a positive multiple of $X_H$, so that 
 $\Phi_\gamma(\p_s u_1)=(f_1,0)$ with $f_1>0$ and $\Phi_\gamma(\p_s
u_2)=(f_2,0)$ 
 with $f_2<0$. We denote by $\overline D_{u_1}$, $\overline D_{u_2}$
 the elements of $\cO(\R\times S^1,\theta_n; S_p,S_q)$ obtained by
 conjugation of $\widetilde D_{u_1}$, $\widetilde D_{u_2}$ with
 $\Phi_\gamma$. The main point is to 
 consider the operator $\psi(\overline D_{u_1})=\overline
 D_{u_1}\#_R T$, with $T$ as in 
 Lemma~\ref{lem:T}. A basis of $\Det(\overline D_{u_i})$
 corresponding to the coherent orientation is by definition  
 $\epsilon(u_i)(f_i,0)$, $i=1,2$. The image of this basis under the
 action of $\psi$ is given by $\epsilon(u_1)(f_1^\#,0)$, for some
 $f_1^\#\in W^{1,p}(\R\times S^1,\R)$ with $\|f_1^\#-f_1\|_{1,p}$
 arbitrarily small with $R\to\infty$, hence $f_1^\#>0$ for $R$ large
 enough. By Lemma~\ref{lem:T}, a basis of $\Det(\overline D_{u_1}\#_R
 T)$ corresponding to the coherent orientation is
 $\epsilon(\gamma)\epsilon(u_1)(f_1^\#,0)$. Finally, the operators
 $\overline D_{u_1}\#_R T$ and $\overline D_{u_2}$ can be connected
 by a continuous path of operators $D_t$, $t\in[0,1]$ satisfying
 properties (ii)-(iv) in the proof of
 Proposition~\ref{prop:Surjectivity_udelta}, as well as 
 the following weaker form of property (i) therein. 
\begin{itemize}
 \item[(i')] there exists a smooth path $c:\R\to S_\gamma$ with
 $c(\pm\infty)$ being fixed critical points of
 $f_\gamma$, such that $\|S(s,\theta)  - \oS(\theta+\vartheta\circ
 c(s)-\vartheta\circ c(-\infty))\|$ is bounded by a constant multiple
 of $\delta$.  
\end{itemize} 
The connected components of the set of operators satisfying (i') and
(ii)-(iv) are indexed by homotopy classes of paths $c$ as
above. Gluing $\overline D_{u_1}$ to $T$ has precisely the effect of
concatenating $c_1$ with $(\gamma_q|_{[0,1/m]})^{-1}$, which is homotopic
to $c_2$. The proof of Proposition~\ref{prop:Surjectivity_udelta}
works the same with the weaker assumption (i') and shows that the
operators $D_t$ are surjective and that $\ker\, D_t$ is generated by
an element of the form $(f_t,0)$, where $f_t\in W^{1,p}(\R\times S^1,\R)$
has constant sign for $t\in[0,1]$. We conclude that
$\epsilon(u_1)\epsilon(\gamma)f_1^\#$ and $\epsilon(u_2)f_2$ have the
same sign, hence~\eqref{eq:signs} is proved. 
\end{proof} 

\medskip 

We now generalize the construction of coherent orientations to the
moduli spaces of Morse-Bott trajectories with gradient
fragments. We define 
$\tcS$\index{$\tcS$, loops of symmetric matrices with one degenerate direction} 
to be the space of loops of
symmetric matrices $S:S^1\to M_{2n}(\R)$ such that the symplectic
matrix $\Psi_S(1)$ defined by~\eqref{eq:PsiS} has exactly one
eigenvalue equal to $1$, corresponding to the eigenspace $\R\oplus
0\subset \R\oplus \R^{2n-1}=\R^{2n}$. Let $\beta:\R\to[0,1]$ be a
smooth function equal to $0$ near $-\infty$ and equal to $1$ near
$+\infty$. We define $\oV$, $\uV$ to be the one-dimensional real
vector spaces generated by the vector-valued functions
$(1-\beta(s))(1,0)$ and $\beta(s)(1,0)$ respectively. 
In the following we denote by $W^{1,p,d}=W^{1,p}(e^{d|s|}ds\,
  d\theta)$, $L^{p,d}=L^p(e^{d|s|}ds\, d\theta)$. Given
$\oS,\uS\in\tcS$ we denote by
\index{$\tcO$, space of CR operators|(}   
$$
\tcO(\R\times S^1,E;\oS,\uS)
$$
the space of linear operators $D:W^{1,p,d}(\R\times
S^1,E)\oplus\oV\oplus\uV\to L^{p,d}(\R\times S^1,\Lambda^{0,1}E)$ of
the form $(\p_s+J_0\p_\theta+S(s,\theta))(ds-id\theta)$ in a local
trivialization of $E$, for which there exist $\otheta,\utheta\in\R/\Z$
such that $\lim_{s\to-\infty} S(s,\theta)=\oS(\theta+\otheta)$ and 
$\lim_{s\to\infty} S(s,\theta)=\uS(\theta+\utheta)$ in the given
trivializations at infinity of $E$. Given
$\oS_0\in\cS$, $\uS\in\tcS$ we denote by   
$$
\tcO^u(\R\times S^1,E;\oS_0,\uS)
$$
the space of linear operators 
$$
D:W^{1,p}(\R\times
S^1,E;g_+(s)ds \, d\theta)\oplus\uV\to L^p(\R\times
S^1,\Lambda^{0,1}E; g_+(s)ds \, d\theta)
$$ 
with $g_+(s):=\max(1,e^{ds})$, which are of the form
$(\p_s+J_0\p_\theta+S(s,\theta))(ds-id\theta)$ in a local 
trivialization of $E$, and for which there exists $\utheta\in\R/\Z$
such that $\lim_{s\to-\infty} S(s,\theta)=\oS_0(\theta)$ and 
$\lim_{s\to\infty} S(s,\theta)=\uS(\theta+\utheta)$ in the given
trivializations at infinity of $E$. Given
$\oS\in\tcS$, $\uS_0\in\cS$ we denote by   
$$
\tcO^s(\R\times S^1,E;\oS,\uS_0)
$$
the space of linear operators 
$$
D:W^{1,p}(\R\times
S^1,E; g_-(s)ds \, d\theta)\oplus\oV\to L^p(\R\times
S^1,\Lambda^{0,1}E; g_-(s) ds \, d\theta)
$$ 
with $g_-(s):=\max(1,e^{-ds})$, which are of the form
$(\p_s+J_0\p_\theta+S(s,\theta))(ds-id\theta)$ in a local 
trivialization of $E$, and for which there exists $\otheta\in\R/\Z$
such that $\lim_{s\to-\infty} S(s,\theta)=\oS(\theta+\otheta)$ and 
$\lim_{s\to\infty} S(s,\theta)=\uS_0(\theta)$ in the given
trivializations at infinity of $E$. Given $\widetilde S\in\tcS$
we denote by 
$$
\tcO_\pm(\C,E;\widetilde S)
$$
the space of linear operators $D:W^{1,p,d}(\C,E)\oplus V_\pm\to
L^{p,d}(\C,\Lambda^{0,1}E)$ of the 
form $(\p_x + J_0 \p_y + S(z))d\bar z$ in a local trivialization of
$E$ and such 
that, when expressed in holomorphic cylindrical coordinates
$(s,\theta)$ with $e^{\pm 2\pi(s+i\theta)}=z$ as 
$(\p_s+J_0\p_\theta+S(s,\theta))(ds-id\theta)$, there exists
$\theta_\pm\in\R/\Z$ so that $\lim_{s\to \pm\infty} S(s,\theta) =
\widetilde S(\theta+\theta_\pm)$   
in the given trivialization of $E$ near infinity. Here we use the
notation $V_+:=\oV$ and $V_-:=\uV$. 
\index{$\tcO$, space of CR operators|)}   

\medskip 

Due to the exponential weights, each of the above spaces $\tcO$
consists of Fredholm operators and comes equipped with a canonical
real line bundle $\Det(\tcO)$ whose fiber at $D$ is $\Det(D)$. 
Unlike in the nondegenerate case, the spaces $\tcO$ are not generally
contractible, hence we have to investigate the orientability of
$\Det(\tcO)$. 

Given $S\in \tcS$ we define $m=m(S)$ to be the maximal positive
integer such that $S(\theta+1/m)=S(\theta)$, $\theta\in\R/\Z$. The
number $m$ is infinite if and only if the loop $S$ is constant, in
which case the spaces $\tcO_\pm(\C,E;S)$, $\tcO^u(\R\times
S^1,E;\oS_0,S)$, $\tcO^s(\R\times S^1,E;S,\uS_0)$ are contractible. In the
following we shall restrict ourselves to nonconstant loops $S\in\tcS$,
in which case the above spaces have the homotopy type of $S^1$, while
$\tcO(\R\times S^1,E;\oS,\uS)$ has the homotopy type of $S^1\times
S^1$ (this is because they fiber over $S^1$, respectively
$S^1\times S^1$ with contractible fibers). We denote by $S_1\in\tcS$
the unique loop such that $S(\theta)=S_1(m\theta)$. 

\begin{lemma} \label{lem:tcap}
 Let $S\in \tcS$ be nonconstant. Then $\Det(\tcO_\pm(\C,E;S))$ is
 nonorientable if and only if $m$ is even and
 $\mu_{RS}(S)-\mu_{RS}(S_1)$ is odd. 
\end{lemma} 

\begin{proof}
 We prove the statement only for $\tcO_+:=\tcO_\pm(\C,E;S)$, the proof
 of the other case being similar. The following two remarks will allow
 us to apply Lemma~\ref{lem:cover}. First, $\Det(D)$ is
 naturally isomorphic to $\Det(D|_{W^{1,p,d}})\otimes V_+$ and $V_+$
 is a trivial bundle over $\tcO_+$. Second, the operator
 $D|_{W^{1,p,d}}$ is conjugated to an operator $\widetilde D\in
 \cO_+(\C,E;S-\frac d p \one)$. Hence it is enough to study the
 orientability of the bundle $\widetilde \Det(\tcO_+)$ over $\tcO_+$
 with fiber $\Det(\widetilde  D)$. 

The bundle $\widetilde \Det(\tcO_+)$ is
 orientable if and only if its restriction to a loop generating
 $\pi_1(\tcO)$ is orientable. After choosing $D\in\tcO_+$ which is
 invariant under conjugation with $z\mapsto e^{-2i\pi/m}$, the
 conjugation of $D$ by $r_t:\C\to\C$, $z\mapsto e^{-2i\pi t/m}z$
 provides such a loop $D_t$, $t\in[0,1]$ with $D_0=D_1=D$. The
 orientation on $\Det(\widetilde D_1)$ obtained by continuation along
 the path $D_t$ from an orientation on $\Det(\widetilde D_0)$ is the
 same as the one induced by the action of $\phi_m^{-1}$ (or $\phi_m$)
 in Lemma~\ref{lem:cover}. Since $\mu_{CZ}(S-\frac d p
 \one)=\mu_{RS}(S)-1/2$ and $\mu_{CZ}(S_1-\frac d p
 \one)=\mu_{RS}(S_1)-1/2$, the statement follows from
 Lemma~\ref{lem:cover}. 
\end{proof} 

The same kind of argument gives the following result.

\begin{lemma} \label{cor:tcyl}
 Let $S,\oS,\uS\in \tcS$ be nonconstant and $\oS_0,\uS_0\in\cS$. The
 line bundles $\Det(\tcO(\R\times S^1,E;\oS,\uS))$,
 $\Det(\tcO^u(\R\times S^1,E;\oS_0,S))$, $\Det(\tcO^s(\R\times
 S^1,E;S,\uS_0))$ are nonorientable if and only if the condition in
 Lemma~\ref{lem:tcap} holds for $S$ and for one of $\oS,\uS$. \hfill{$\square$}
\end{lemma}

The previous results motivate the following definition. We denote 
$$
\tcSg:=\{S\in\tcS\, : \, S \mbox{ constant or }
\mu_{RS}(S)-\mu_{RS}(S_1) \mbox{ is even}\},
$$
and 
$$
\tcSb:= \{S\in\tcS\, : \, S \mbox{ nonconstant and }
\mu_{RS}(S)-\mu_{RS}(S_1) \mbox{ is odd}\},
$$
so that $\tcSb=\tcS \setminus \tcSg$. Although the determinant lines over the 
various spaces $\tcO$ are nonorientable if one of the asymptotes is in $\tcSb$, 
we can construct {\bf covers $\ttcO$ of $\tcO$} over which
\index{$\ttcO$, cover of space of CR operators}
the determinant 
lines become orientable. Let $\uS_0\in\cS$ and $\oS\in\tcSb$ with 
$\oS(\theta)=\oS_1(m\theta)$, $\theta\in\R/\Z$. We define 
$\ttcOs(\R\times S^1,E;\oS,\uS_0)$ to consist of pairs $(D,\otheta)$ such 
that $\otheta\in\R/\frac 2 m \Z$, 
$D=(\p_s + J_0\p _\theta + S(s,\theta))(ds-id\theta)\in\tcO^s(\R\times S^1,E;\oS,\uS_0)$ with
$\lim_{s\to-\infty} S(s,\theta)=\oS(\theta+\otheta)$. 
The obvious projection is $\ttcOs\to \tcO^s$ is a double cover and the lift 
of the determinant bundle to $\ttcOs$ is orientable. We define in a completely 
analogous manner double covers 
$\ttcO_\pm(S)\to\tcO_\pm(S)$,  
$\ttcOu(\oS_0,S)\to \tcO^u(\oS_0,S)$, $S\in\tcSb$ and a cover 
$\ttcO(\oS,\uS)\to \tcO(\oS,\uS)$ which is double if exactly one of $\oS,\uS$ is in $\tcSb$, 
and quadruple if both $\oS,\uS$ are in $\tcSb$. 

\medskip

We define now gluing operations between elements of the various spaces
$\tcO$. Let $K$ in $\tcO_+(\C,E;S)$, $\tcO(\R\times S^1,E;\oS,S)$,
or $\tcO^u(\R\times S^1,E;\oS_0,S)$, and $L$ in $\tcO_-(\C,F;S)$,
$\tcO(\R\times S^1,F;S,\uS)$, or $\tcO^s(\R\times S^1,F;S,\uS_0)$. 
We denote by $S_K$, respectively $S_L$ the matrix valued functions
involved in $K$ and $L$ near infinity. We assume that
$$
\lim_{s\to+\infty}S_K(s,\cdot)=\lim_{s\to-\infty}
S_L(s,\cdot)=S(\cdot+\theta_0)=:S_{\theta_0} 
$$
for some $\theta_0\in\R/\Z$. We choose a cutoff function $\beta:\R\to
[0,1]$ such that $\beta(s)=0$ if $s\le 0$ and $\beta(s)=1$ if $s\ge 1$. 
Given $R>0$ large we define operators $K_R$ and $L_R$ by replacing
$S_K$ and $S_L$ by $S_{\theta_0}+\beta(R-s)(S_K-S_{\theta_0})$ and
$S_{\theta_0}+\beta(R+s)(S_L-S_{\theta_0})$ respectively.  
We cut out semi-infinite cylinders $\{s>R\}$ from the base of
$E$, $\{s<-R\}$ from the base of $F$, then identify their boundaries
using the coordinate $\theta$. We glue the vector bundles $E$ and $F$
using their given trivializations near infinity and denote the
resulting vector bundle by $E\# F$. We define $K\#_R L$ by
concatenating $K_R$ and $L_R$, so that $K\#_R L$ belongs to one of the
spaces $\cO(\C P^1)$, $\tcO_+(\C; \uS)$, $\cO_+(\C;\uS_0)$, or
$\tcO_-(\R\times S^1; \oS)$, $\tcO(\R\times 
S^1; \oS,\uS)$, $\tcO^s(\R\times S^1;\oS,\uS_0)$, or 
$\cO_-(\C;\oS_0)$, $\tcO^u(\R\times S^1;\oS_0,\uS)$,
$\cO(\R\times S^1;\oS_0,\uS_0)$, where we have omitted the symbol $E\#
F$ from the notation. 

The above gluing operations admit a straightforward extension to the spaces 
$\ttcO$. For example, two elements $(K,\utheta)\in\ttcOu(\oS_0,S)$, $(L,\otheta)\in
\ttcOs(S,\uS_0)$ can be glued if $\utheta=\otheta$, in which case they 
give rise to an element $K\#_R L\in\cO(\oS_0,\uS_0)$. 

Recall that the domain of an operator $D$ in some $\tcO$ contains a
canonically oriented $1$-dimensional summand for each asymptote in
$\tcS$, together with a canonical isomorphism with $\R$. We denote by
$V_K$, $V_L$ the summands corresponding to the asymptote $S$ of $K$
and $L$ respectively, and we let $V:=V_K\oplus_\R V_L$ be their
(canonically oriented) fibered sum. By~\cite[Corollary~6]{BM}, for
$R>0$ large enough there is a natural isomorphism $\Det(K\oplus_\R
L)\simeq \Det(K\#_R L)$ defined up to homotopy, where $K\oplus_\R L$
is the restriction of $K\oplus L$ to the fibered sum of their
domains. Since $V$ is canonically oriented, it follows that
$\Det(K\oplus_\R L)$ is canonically isomorphic to $\Det(K\oplus
L)\simeq \Det(K)\otimes \Det(L)$. Hence we obtain a canonical
isomorphism $\Det(K)\otimes \Det(L)\stackrel \sim \to \Det(K\#_R L)$
defined up to homotopy, and inducing an associative gluing operation
for orientations. Similar considerations apply to the elements of the spaces $\ttcO$. 

\medskip 

\begin{remark} 
 We can construct a system of coherent orientations on the
determinant line bundles $\Det(\tcO_\pm(\C,E;S))$ and
$\Det(\tcO(\R\times S^1,E;\oS,\uS))$ with $S,\oS,\uS\in \tcSg$ by the
same procedure as for the spaces $\cO$. We can moreover extend this to
a system of coherent orientations involving all spaces $\cO$, $\tcO$ and $\ttcO$. 
Nevertheless, if we want that certain orientations have a
geometric meaning, we have to impose compatibility conditions which
seem ad-hoc in such a general setup. This is why we restrict ourselves
in the sequel to the spaces $\cO$, $\tcO$ and $\ttcO$ which are relevant for
our geometric situation. 
\end{remark} 

We use now the notations of Section~\ref{sec:MBcomplex}. Given
$\gamma\in\cP(H)$ we denote by $\Psi_\gamma$ the linearization of the
Hamiltonian flow along $\gamma$ given by~\eqref{eq:CRtriv} and let
$S_\gamma:\R/\Z\to M_{2n}(\R)$ be the corresponding loop of symmetric
matrices defined by $\dot \Psi_\gamma = J_0 S_\gamma
\Psi_\gamma$. Then $S_\gamma\in \tcSg$ if and only if $\gamma$ is a
good orbit. We similarly define $S_{\gamma_q}$ for each
$\gamma_q\in\cP(H_\delta)$, with $q\in\textrm{Crit}(f_\gamma)$. 
For $\og\in \cP(H)$, $\ug_q\in\cP(H_\delta)$ we denote
$\tcO^s(\R\times
S^1,E;\og,\ug_q):= \tcO^s(\R\times S^1,E;S_\og,S_{\ug_q})$ etc. 

\medskip 

\noindent {\bf Convention.} In what follows the spaces $\tcO$ will be 
understood to be indexed only by good orbits, whereas if one of the 
asymptotic orbits is bad we use the corresponding double or quadruple 
cover $\ttcO$. 

\medskip 

We construct orientations on the determinant bundles over all spaces 
$\cO$, $\tcO$, $\ttcO$ indexed by the elements of $\cP(H)$ and $\cP(H_\delta)$ as follows. 
We start by choosing arbitrary orientations of $\Det(\tcO_+(\C,E;\gamma))$, 
respectively $\Det(\ttcO_+(\C,E;\gamma))$, 
$\gamma\in\cP(H)$ such that the trivialization of $E$ at infinity
extends to $\C$. We then choose orientations
of $\Det(\tcO^s(\R\times S^1,E;\gamma,\gamma_q))$, respectively 
$\Det(\ttcOs(\R\times S^1,E;\gamma,\gamma_q))$,
$\gamma\in\cP(H)$, $q\in\textrm{Crit}(f_\gamma)$ such that the trivializations 
of $E$ at infinity extend to $\R\times S^1$, as follows. If $\gamma$ is good, the space
$\tcO^s(\R\times S^1,\theta_n;\gamma,\gamma_q)$ contains a distinguished
family of operators of the form $\Phi_\gamma\circ D_u\circ
\Phi_\gamma^{-1}$, where $u=u_{\delta,\gamma,-1,\infty}$ is the cylinder
corresponding to a semi-infinite gradient trajectory ending at
$q$ and $\Phi_\gamma:T\widehat W|_{S_\gamma} \to S_\gamma\times
\R^{2n}$ is a fixed trivialization satisfying
$\Phi_\gamma(X_H)=(1,0)\in \R\oplus \R^{2n-1}$. 
This family is naturally parametrized by $W^s(q)$, hence it is
connected. As seen in Proposition~\ref{prop:Surjectivity_udelta} the
above Fredholm operators are surjective and have index $1-\ind(q)$.
If the index is zero we choose the orientation sign to be $+1$. If the
index is one the kernel is generated by a nonvanishing section of the
form $(f,0)$, hence is canonically isomorphic to $\R\oplus 0$ and
therefore admits a canonical orientation. 
If $\gamma$ is bad, we choose in an arbitrary way a lift of the operator
$\Phi_\gamma\circ D_u\circ
\Phi_\gamma^{-1}$, where $u=u_{\delta,\gamma,-1,\infty}$ is the cylinder
corresponding to a constant semi-infinite gradient trajectory at
$q$. This determines a lift of the whole path of operators described above, 
and hence an orientation of $\Det(\ttcOs(\R\times S^1, E;\gamma,\gamma_q))$ 
by the previous rule. 

We induce orientations on $\Det(\cO_+(\theta_n))$ by
gluing orientations on the line bundles $\Det(\tcO_+(\theta_n))$ and 
$\Det(\tcO^s(\theta_n))$. The orientations on $\Det(\tcO_+(\theta_n))$
and $\Det(\cO_+(\theta_n))$ determine orientations on
$\Det(\tcO_\pm(E))$ and $\Det(\cO_\pm(E))$ by requiring that the
glued orientation on $\Det(\cO(\C P^1,E))$ is the canonical one.      
We get orientations of $\Det(\tcO(\R\times S^1,E))$ by requiring
that the orientation induced on $\Det(\cO(\C P^1,\theta_n
\# E \# \theta_n))$ by the gluing operation 
$$
\tcO_+(\C,\theta_n;\og)\ _{\uev}\times_{\oev} \tcO(\R\times
S^1\!,E;\og,\ug) _{\uev}\times_{\oev} 
\tcO_-(\C,\theta_n;\ug) \!\to \cO(\C P^1\!,\theta_n \# E \# \theta_n)
$$
is the canonical one. Here we have denoted by $\oev$, $\uev$ the
evaluation maps to $S^1$ at $-\infty$ and $+\infty$ respectively.
Similarly, we get orientations on $\Det(\cO(\R\times S^1,E))$,
$\Det(\tcO^u(\R\times S^1,E))$ and $\Det(\tcO^s(\R\times S^1,E))$, as well as 
orientations on $\Det(\ttcO)$ for the various spaces $\ttcO$.

\begin{lemma} \label{lem:coh}
 The above recipe defines a system of coherent orientations.
\end{lemma} 

\begin{proof} 
We have to prove that, given operators $K$, $L$ that can be glued
lying in one of the 
spaces $\cO$, $\tcO$ or $\ttcO$, the coherent orientations 
$o_K$, $o_L$ of $\Det(K)$ and
$\Det(L)$ induce an orientation $o_K\# o_L$ that coincides with the
coherent orientation of $\Det(K\# L)$. In the case when $K$, $L$
belong to some $\cO(\R\times S^1)$, $\tcO(\R\times S^1)$ or $\ttcO(\R\times S^1)$ 
this means that, for a suitable 
choice of operators $A\in\ttcO_+(\C,\theta_n)$, $A\in\tcO_+(\C,\theta_n)$ or
$A\in\cO_+(\C,\theta_n)$, and $B\in\ttcO_-(\C,\theta_n)$, $B\in\tcO_-(\C,\theta_n)$ or $B\in
\cO_-(\C,\theta_n)$, with $o_A,o_B$ the coherent
orientations on the respective determinant line bundles, 
$o_A\#(o_K\# o_L)\#o_B$ is the canonical orientation on $\Det(\cO(\C
P^1))$. 

Let $E$ and $F$ be the symplectic vector bundles corresponding to $K$
and $L$ respectively. If $E=\theta_n$ or $F=\theta_n$ the conclusion
is a direct consequence of the definitions and of the associativity of
gluing. In the general case $E\neq \theta_n$ and $F\neq \theta_n$ we
give the proof when $K\in \tcO^u(\R\times S^1,E;\og_p,\gamma)$ and $L\in
\tcO^s(\R\times S^1,F;\gamma,\ug_q)$, the other cases being similar. 
Let us introduce an auxiliary loop of symmetric matrices $S_0\in\cS$
such that $[S_0,J_0]=0$, and we define the orientations on
$\Det(\cO_\pm(\C,E';S_0))$ to be the canonical ones (see
Remark~\ref{rmk:Clin}). This determines in turn orientations on 
$\Det(\tcO^u(\R\times S^1,E';S_0,S_\gamma))$, $\gamma\in \cP(H)$ by
requiring that gluing induces the coherent orientation on
$\Det(\tcO_+(\C,\theta_n\#E';\gamma))$.    

Let $A_1\in\cO_+(\C,E_1;S_0)$, $K_1\in\tcO^u(\R\times
S^1,\theta_n;S_0,S_\gamma)$ with $E_1\#\theta_n=\theta_n\#E$. 
By the above definition, we have $o_{A_1}\# o_{K_1}=o_A\#o_K$. 
We obtain 
\begin{eqnarray*}
  o_A\#(o_K\# o_L)\#o_B & = & (o_A\#o_K)\# o_L\#o_B =
 (o_{A_1}\#o_{K_1})\# o_L\#o_B \\ 
 & = & o_{A_1}\#o_{K_1\# L}\#o_B = o_{A_1}\#o_{K_1\# L\#B}.
\end{eqnarray*} 
The operators $A_1$ and $K_1\#L\#B$ are homotopic to $\C$-linear
operators with asymptotic condition $S_0$. The main observation now is
that the gluing of two $\C$-linear operators is again $\C$-linear,
hence the gluing of the above orientations is the canonical one on
$\Det(\cO(\C P^1))$. 
\end{proof} 

Let $\og,\ug$ be good orbits. In this case the procedure for 
orienting the Morse-Bott spaces of
Floer trajectories $\widehat \cM^A(S_\og,S_\ug;H,J)$ 
is entirely similar to the corresponding procedure in the nondegenerate case (it
is actually simpler since we do not need the intermediate transition
from $L^{p,d}$ to $L^p$ spaces). Namely, we pull back the orientation
on $\Det(\tcO)$ using the natural map $\widehat \cM \to \tcO$. This in
turn induces orientations on the quotient spaces
$\cM^A(S_\og,S_\ug;H,J)$. Recall that, given oriented vector spaces
$V\subset W$, we define an orientation on $W/V$ by requiring that the
isomorphism $V\oplus (W/V) \simeq W$ is orientation preserving.   

Since the stable and unstable manifolds of the functions $f_\gamma$
are canonically oriented, one gets orientations (i.e. signs) on all
zero-dimensional moduli spaces of Floer trajectories with gradient
fragments $\cM^A(p,q;H,\{f_\gamma\},J)$ which involve only good
orbits. This is done by the following {\bf fibered sum rule}. Let 
$f_i:W_i\to W$, $i=1,2$ be linear maps 
of oriented vector spaces such that $f:W_1\oplus W_2\to W$,
$(w_1,w_2)\mapsto f_1(w_1) - f_2(w_2)$ is surjective. The orientation
on the fibered sum $W_1 {_{f_1}\oplus_{f_2}} W_2 :=\ker f$ is
defined such that the isomorphism of vector spaces $(W_1\oplus
W_2)/\ker f \stackrel \sim \to W$ induced by $f$ changes orientations
by the sign $(-1)^{\dim\,W_2 \cdot \dim\, W}$. Note that this rule is
such that the fibered sum operation is associative for oriented vector
spaces, and moreover, if $f_2$ is an orientation preserving
isomorphism, the natural isomorphism $W_1 {_{f_1}\oplus_{f_2}} W_2
\simeq W_1$ is orientation preserving. Similarly, if $f_1$ is an
orientation preserving isomorphism, the natural isomorphism $W_1
{_{f_1}\oplus_{f_2}} W_2 \simeq W_2$ is orientation preserving. 

The important remark now is that, although the spaces 
$\widehat \cM^A(S_\og,S_\ug;H,J)$ with $\og$ or $\ug$ being a bad orbit 
may not be orientable, we can nevertheless define orientations (i.e. signs)
on \emph{all} zero-dimensional moduli spaces 
of Floer trajectories with gradient fragments
$\cM^A(p,q;H,\{f_\gamma\},J)$. The sign 
of an (isolated) point $[\u]=(c_m,[u_m],\ldots,c_1,[u_1],c_0)$ in this
moduli space is determined as follows. For each operator $D_{u_i}$,
$i=1,\ldots,m$ with at least one bad asymptote we choose a lift in the
corresponding space $\ttcO(\R\times S^1)$. For each $c_i$,
$i=0,\ldots,m$ lying on a bad orbit $\gamma_i$ the corresponding
operator $D_{u_{\delta,\gamma_i,-T_i/2,T_i/2}}$ admits a unique lift
to the space $\ttcO(S_{\gamma_i},S_{\gamma_i})$ such that it can be
glued with both $D_{u_{i+1}}$ and $D_{u_i}$. Since all these operators
are surjective, the orientations of the determinant line bundles over
the spaces $\tcO$ and $\ttcO$ induce orientations on $T_{u_i}\widehat
\cM^{A_i}(S_{\gamma_i},S_{\gamma_{i-1}})$, respectively 
$TW^u(p)$, $TW^s(q)$ and $T_{(c_i(-T_i/2),T_i)} (S_{\gamma_i}\times
\R^+)$, $i=1,\ldots,m-1$. By the fibered sum rule we get an
orientation on $T_\u \widehat \cM^A(p,q;H,\{f_\gamma\},J)$ which we
call ``the coherent orientation''. On the other hand this vector space
carries the ``geometric orientation'' of the basis $(\p_s
u_m,\ldots,\p_s u_1)$. We define the sign 
\begin{equation} 
\label{eq:signepsilon} \index{$\epsilon(\u)$} 
\epsilon(\u)=\epsilon([\u])
\end{equation}
to be $+1$ if these two orientations coincide, and $-1$ if they are
different.

We now want to compare the signs $\epsilon(\u)$ with the
signs $\epsilon(u_\delta)$ of the glued trajectories $u_\delta$
corresponding to $\u$. The situation is expressed by the following
diagram, in which we dropped the decorations $A$, $(H,\{f_\gamma\},J)$
and $(H_\delta,J)$ and in which we have
indicated on the morphism arrows the way in which the corresponding
isomorphisms of vector spaces act on orientations. 

$$
\xymatrix
@C=25pt
@R=30pt
{
\genfrac{}{}{0pt}{3}{\mbox{\scriptsize{Coherent}}}
{\mbox{\scriptsize{orientation}}} \ar@{.>}[r]
&
T\widehat \cM(p,q) \ar[r]_-{\phi}^-1
\ar[d]_{\textrm{Id}}^{\epsilon(\u)}  & 
T\widehat \cM(\og_p,\ug_q)\oplus \R^{m-1} 
\ar[d]_{\textrm{Id}}^{\epsilon(u_\delta)} 
&
\genfrac{}{}{0pt}{3}{\mbox{\scriptsize{Coherent}}}
{\mbox{\scriptsize{orientation}}} \ar@{.>}[l]
\\
\genfrac{}{}{0pt}{3}
{\genfrac{}{}{0pt}{3}{\mbox{\scriptsize{Geometric}}}
 {\mbox{\scriptsize{orientation}}}}
{\stackrel {\, }
{\langle\p_s u_m,\ldots,\p_s u_1\rangle}}
\ar@{.>}[r]
&
T\widehat \cM(p,q) \ar[r]_-{\phi}^-? & T\widehat
\cM(\og_p,\ug_q) \oplus \R^{m-1} 
&
\genfrac{}{}{0pt}{3}
{\genfrac{}{}{0pt}{3}{\mbox{\scriptsize{Geometric}}}
 {\mbox{\scriptsize{orientation}}}}
{\stackrel {\, }
{\langle\p_s u_\delta\rangle \, \oplus \, \R^{m-1}}}
\ar@{.>}[l] 
}
$$

The map $\phi$ is defined from gluing as follows. The tangent space
$T\widehat \cM(p,q)$ is the kernel of the operator $D_{\widetilde w}$,
$\widetilde w = (v_m,u_m,\ldots,v_1,u_1,v_0)$ considered in 
Lemma~\ref{lem:Dwtilde}. Moreover, since the cokernel of
$D_{\widetilde w}$ is naturally oriented, the coherent orientation of
$\Det(D_{\widetilde w})$ induces a ``coherent'' orientation on
$\ker\,D_{\widetilde w}$. Recall that
the analytical expression of $D_\tw$ is $D_{v_m}\oplus D_{u_m} \oplus
D'_{v_{m-1}} \oplus \ldots \oplus D'_{v_1} \oplus D_{u_1} \oplus
D_{v_0}$, and note that $D_\tw$ admits 
a natural stabilization $D_\tw^{\R^{m-1}}$ obtained by replacing
$D'_{v_i}$, $i=1,\ldots,m-1$ 
with $D\{\dbar_T\}(v_i,T_{v_i})$ (see Remark~\ref{rmk:dbarT} for the 
definitions). By~\cite[Corollary~6]{BM} there is a natural isomorphism
$\widetilde \phi:\ker\,D_\tw^{\R^{m-1}} \stackrel \sim \to \ker\,
D_{G_\delta(\tw)}^{\R^{m-1}}$ which preserves the coherent orientations. 
We denote by $\phi:\ker\,D_\tw^{\R^{m-1}} \stackrel \sim \to
\ker\,D_{u_\delta}^{\R^{m-1}}$ the composition of $\widetilde \phi$ with
the projection $\Pi$ on $\ker\,D_{u_\delta}^{\R^{m-1}}$ along the
image of the right inverse $Q_\delta$ of $D_{G_\delta(\tw)}$ given by
Proposition~\ref{prop:gluing}. Since $D_{G_\delta(\tw)}$ and
$D_{u_\delta}$ are close in the relevant $\delta$-norm, we get that
$\phi$ is an isomorphism preserving coherent orientations.  

The vertical maps change orientations by $\epsilon(\u)$, respectively
$\epsilon(u_\delta)$ by definition, and the whole work now goes into
determining the action of $\phi$ on the geometric
orientations. 

\begin{remark} {\rm 
If $\gamma$ is a good orbit and $p\in\textrm{Crit}(f_\gamma)$,  
the geometric orientations on $W^u(p)$ and $W^s(p)$ coincide with the 
coherent ones. Indeed, the unstable manifold 
$W^u(p)$ is naturally identified with the
zero set of the section $\dbar_{-\infty,1}$ defined on
$\cB_\delta^{1,p,d}(p,S_\gamma; f_\gamma)$ by~\eqref{eq:Hprime},
whereas the stable manifold $W^s(p)$ is naturally identified with the
zero set of the section $\dbar_{-1,\infty}$ defined on
$\cB_\delta^{1,p,d}(S_\gamma,q; f_\gamma)$. The assertion for $W^s(p)$
is then a direct consequence of 
the definition of the orientation on $\Det(\tcO^s(\R\times
S^1,\theta_n;\gamma,\gamma_q))$. As for $W^u(p)$, let us consider
the gluing operation 
$$
\tcO^u(\R\times S^1,\theta_n;\gamma_p,\gamma) {_{\uev}}\times_{\oev}
\tcO^s(\R\times S^1,\theta_n;\gamma,\gamma_p) \to \cO(\R\times
S^1,\theta_n;\gamma_p,\gamma_p). 
$$
We choose the surjective operators $D_1:=D_{u_{\delta,\gamma,-\infty,1}}$,
$D_2:=D_{u_{\delta,\gamma, -1,\infty}}$ corresponding to the constant
gradient trajectory at $p$. With these choices $D_1\#
D_2=D_{u_{\delta,\gamma,-\infty,\infty}}=:D$ also corresponds to the
constant gradient trajectory at $p$. The operator $D$ is an
isomorphism and, by the coherent choice of the orientations, the
determinant line $\Det(D)\simeq \R$ is positively oriented. 
If $p$ is the maximum of $f_\gamma$ then $\ker\,D_2\simeq T_pS_\gamma$
as oriented vector spaces (by definition), the kernel of $D_1$ is trivial
and its determinant line must be positively oriented. If $p$ is the
minimum of $f_\gamma$ then $\ker\,D_2$ is trivial and its determinant
line is positively oriented by definition, therefore $\ker\, D_1\simeq
T_pS_\gamma$ must have the geometric orientation. 
}
\end{remark}

\begin{lemma} \label{lem:signs}
 Assume $\dim \, \cM^A(p,q;H,\{f_\gamma\},J) =0$ and fix an element 
 $[\u]\in\cM^A(p,q;H,\{f_\gamma\},J)$ with $m\le 2$ 
 sublevels. Then $\epsilon(\u)=\epsilon(u_\delta)$ if $m=0,1$ and
$\epsilon(\u)=-\epsilon(u_\delta)$ if $m=2$. 
\end{lemma} 

\begin{proof} If
$m=0$ the statement is obvious since $\u$ consists of a single
gradient trajectory and $\u=u_\delta$ (see
Lemma~\ref{lem:signsgrad}). We now have to show that the map
$\phi$ in the previous diagram preserves the geometric orientation if
$m=1$, respectively reverses it if $m=2$. Since a shift $\sigma$ on $u_1$
produces a glued trajectory $u_\delta$ shifted by the same amount
$\sigma$, we infer that $\phi(\p_s u_1)=\p_s u_\delta$ and, for
$m=1$, the statement follows from the commutativity of the
diagram.  

Let us now examine the case $m=2$. We recall that $\phi=\Pi \circ
\widetilde \phi$, where the isomorphism $\widetilde \phi$ is the
composition of the gluing map $G$ in the proof of
Proposition~\ref{prop:gluing} with  the projection to
$\ker\,D_{G_\delta(\tw)}^\R$ along the image of $Q_\delta$
(see~\cite{BM}). We first show that $\phi(\p_s u_1 + \p_s u_2)$ is
close in $\|\cdot\|_{1,\delta}$-norm to $\p_s u_\delta$. We denote
$\tw^{\sigma,\sigma}:=(v_2,u_2(\cdot+\sigma),v_1,u_1(\cdot+\sigma),v_0)$,
$\sigma\in\R$.  
Then $G(0\oplus\p_s u_2 \oplus 0 \oplus
\p_s u_1 \oplus 0)$ is $\|\cdot\|_{1,\delta}$-close to $\frac d
{d\sigma} \big|_{\sigma=0} G_\delta(\tw^{\sigma,\sigma})$, which is
$\|\cdot\|_{1,\delta}$-close to $\frac d
{d\sigma} \big|_{\sigma=0} G_\delta(\tw)(\cdot +\sigma)$, which is in
turn close to $\frac d {d\sigma} \big|_{\sigma=0}
v_\delta(\cdot+\sigma) = \p_s v_\delta$. Then $\phi(\p_s u_1 +\p_s
u_2) = \Pi(G(0\oplus\p_s u_2 \oplus 0 \oplus
\p_s u_1 \oplus 0))$ is $\|\cdot\|_{1,\delta}$-close to $\Pi(\p_s
u_\delta)=\p_s u_\delta$. 

We now show that $\phi(\p_s u_1 - \p_s u_2)\in \ker \, D_{u_\delta}
\oplus \R$ is a vector having a negative component in the $\R$
direction and whose component on $\ker\,D_{u_\delta}$ is
$\|\cdot\|_{1,\delta}$-close to $-\p_s u_1$. Then the conclusion
follows. 

Let  
$\tw^{-\sigma,\sigma}:=(v_2,u_2(\cdot-\sigma),v_1,u_1(\cdot+\sigma),v_0)$,
$\sigma\in\R$. Then $G(0\oplus (-\p_s u_2) \oplus 0 \oplus
\p_s u_1 \oplus 0)$ is $\|\cdot\|_{1,\delta}$-close to $\frac d
{d\sigma} \big|_{\sigma=0} G_\delta(\tw^{-\sigma,\sigma})$. We define
$\epsilon(\sigma):= 2\delta\sigma$ and the section
\begin{equation} \label{eq:minussigma}
\frac d {d\sigma} \big|_{\sigma=0} G_{\delta,\epsilon(\sigma)}
(\tw^{-\sigma,\sigma}) = \frac d {d\sigma} \big|_{\sigma=0}
G_\delta(\tw^{-\sigma,\sigma}) + 2\delta \frac d {d\epsilon}
\big|_{\epsilon=0} G_{\delta,\epsilon}(\tw)
\end{equation}
is by construction $\|\cdot\|_{1,\delta}$-close to $\frac d {d\sigma}
\big|_{\sigma=0} G_\delta(\tw^{-\sigma,-\sigma})$, hence
$\|\cdot\|_{1,\delta}$-close to $-\p_s u_\delta$. 
By adapting the arguments in the
proof of Proposition~\ref{prop:catenation} one sees that the section 
$$
\frac d {d\epsilon} \big|_{\epsilon=0} \dbar_{T_{v_1}+\epsilon}
G_{\delta,\epsilon}(\tw) = D\dbar_{T_{v_1}} \frac d {d\epsilon}
\big|_{\epsilon=0} G_{\delta,\epsilon}(\tw) + D\{\dbar_T\}
(G_{\delta}(\tw),T_{v_1})\cdot (0,1)
$$ 
is $\|\cdot\|_\delta$-small. Here the sections $\dbar_T$ are of the
form $\dbar_{H_T,J}$, where $H_T$ is the $s$-dependent Hamiltonian
given respectively by~\eqref{eq:Hprime} on the intervals of definition
of $v_2$, $v_1$, $v_0$, and equal to $H$ on the intervals of
definition of $u_1$, $u_2$. The previous equation shows that $(\frac d
{d\epsilon} \big|_{\epsilon=0} G_{\delta,\epsilon}(\tw),1)\in
\textrm{dom}(D_{G_{\delta}(\tw)}^\R)$ is $\|\cdot \|_{1,\delta}$-close
to $\ker\,D_{u_\delta}^\R$. On the other hand,
equation~\eqref{eq:minussigma} 
shows that $G(0\oplus (-\p_s u_2) \oplus 0 \oplus \p_s u_1 \oplus 0)$
is $\|\cdot\|_{1,\delta}$-close to $-\p_s u_\delta -2\delta \frac d
{d\epsilon} \big|_{\epsilon=0} G_{\delta,\epsilon}(\tw)$. Hence, after
projecting to $\ker\,D_{u_\delta}^\R=\ker\,D_{u_\delta}\oplus\R$, we
get a vector having a negative component in the $\R$ direction and
whose component on $\ker\,D_{u_\delta}$ is
$\|\cdot\|_{1,\delta}$-close to $-\p_s u_\delta$.  
\end{proof}

\begin{proof}[{\bf Proof of Proposition~\ref{prop:signs}}] The
special statement concerning the case $m=0$ was proved in
Lemmas~\ref{lem:signsgrad} and~\ref{lem:signs}, 
whereas the equality $\epsilon(\u)=(-1)^{m-1}\epsilon(u_\delta)$ in
case $m=1,2$ was the content of Lemma~\ref{lem:signs}. 
The proof in the case $m\ge 3$ is just a more elaborate version of the
proof of Lemma~\ref{lem:signs}. We consider the basis of $T_\u\widehat
\cM(p,q)$ given by 
$$
\begin{array}{rcl}
e_0 & := & \p_s u_m + \p_s u_{m-1} +\ldots + \p_s u_2 + \p_s u_1, \\
e_1 & := & -\p_s u_m + \p_s u_{m-1} +\ldots + \p_s u_2 + \p_s u_1, \\
 & & \vdots \\
e_{m-2} & := & -\p_s u_m - \p_s u_{m-1} -\ldots + \p_s u_2 + \p_s u_1,
\\
e_{m-1} & := & -\p_s u_m - \p_s u_{m-1} -\ldots - \p_s u_2 + \p_s u_1.
\end{array}
$$
It is easy to see that the orientation determined by
$(e_0,\ldots,e_{m-1})$ is the same as the geometric orientation determined
by $(\p_s u_m,\ldots,\p_s u_1)$. We have to show that the orientation 
of the basis $(\phi(e_0),\ldots,\phi(e_{m-1}))$ differs from the canonical
orientation of $\langle \p_s u_\delta \rangle \oplus \R^{m-1}$ by
$(-1)^{m-1}$.  

As in Lemma~\ref{lem:signs} we see that $\phi(e_0)$ is
$\|\cdot\|_{1,\delta}$-close to $\p_s u_\delta$. We now show that 
$\phi(e_k)\in\ker D_{u_\delta} \oplus \R^{m-1}$, $k=1,\ldots,m-1$ has
a negative component which is bounded away from zero along the
corresponding factor $\R\subset \R^{m-1}$, that the other components
in $\R^{m-1}$ are close to zero, whereas the component along $\ker
D_{u_\delta}$ is close to $-\p_s u_\delta$ in
$\|\cdot\|_{1,\delta}$-norm. Then the conclusion will follow since 
the orientation defined by $(\phi(e_0),\ldots,\phi(e_{m-1}))$ is the
same as the orientation defined by 
$$
(\p_s u_\delta,0,\ldots,0),(-\p_s u_\delta,-1,0,\ldots,0), 
\ldots,(-\p_s u_\delta,0,\ldots,0,-1)).
$$

Let us fix $k=1,\ldots,m-1$. We shall freely use the notation $e_k$
for the vector $0\oplus(-\p_s
u_m)\oplus 0 \oplus \ldots \oplus (-\p_s u_{m-k+1})\oplus 0 \oplus
\p_s u_{m-k} \oplus \ldots \p_s u_1\oplus 0$ in the domain of the
gluing map $G$ defined in the proof of
Proposition~\ref{prop:gluing}. For $\sigma>0$ we denote 
$$
\tw_k^{-\sigma,\sigma} \!:=\! (v_m,u_m(\cdot-\sigma),\ldots,
u_{m-k+1}(\cdot-\sigma),v_{m-k},u_{m-k}(\cdot+\sigma),
\ldots,u_1(\cdot+\sigma),v_0),
$$
$$
\tw^{-\sigma,-\sigma} \!:=\! (v_m,u_m(\cdot-\sigma),\ldots,
u_{m-k+1}(\cdot-\sigma),v_{m-k},u_{m-k}(\cdot-\sigma),\ldots,u_1(\cdot-\sigma),v_0).
$$
Then $G(e_k)$ is $\|\cdot\|_{1,\delta}$-close to $\frac d
{d\sigma}\big|_{\sigma=0}G_\delta(\tw_k^{-\sigma,\sigma})$. 
We denote 
$$
\oeps_k(\epsilon):=(0,\ldots,\epsilon,\ldots,0),
$$
where the parameter $\epsilon>0$ appears on position $m-k$.
The section
\begin{equation} \label{eq:minussigmak}
 \frac d {d\sigma} \big|_{\sigma=0} G_{\delta,\oeps_k(2\delta\sigma)}
(\tw_k^{-\sigma,\sigma}) = \frac d {d\sigma} \big|_{\sigma=0}
G_\delta(\tw_k^{-\sigma,\sigma}) + 2\delta \frac d {d\epsilon}
\big|_{\epsilon=0} G_{\delta,\oeps_k}(\tw)
\end{equation}
is by construction $\|\cdot\|_{1,\delta}$-close to $\frac d
{d\sigma}\big|_{\sigma=0} G_\delta(\tw^{-\sigma,-\sigma})$, hence 
$\|\cdot\|_{1,\delta}$-close to $-\p_s u_\delta$. 
As in Lemma~\ref{lem:signs}, by adapting the arguments in the proof
of Proposition~\ref{prop:catenation} one sees that the section 
\begin{eqnarray*}
\frac d {d\epsilon} \big|_{\epsilon=0} \dbar_{T_{v_{m-k}}+\epsilon}
G_{\delta,\oeps_k(\epsilon)}(\tw) & = & 
D\dbar_{T_{v_{m-k}}} \frac d {d\epsilon}
\big|_{\epsilon=0} G_{\delta,\oeps_k(\epsilon)}(\tw) \\ 
& & + \ D\{\dbar_T\}
(G_{\delta}(\tw),T_{v_{m-k}})\cdot (0,1)
\end{eqnarray*}
is $\|\cdot\|_\delta$-small. As before, the sections $\dbar_T$ are of the
form $\dbar_{H_T,J}$, where $H_T$ is the $s$-dependent Hamiltonian
given respectively by~\eqref{eq:Hprime} on the intervals of definition
of $v_m, v_{m-1},\ldots,v_0$, and equal to $H$ on the intervals of
definition of $u_m,\ldots,u_1$. The previous equation shows that
$$
(\frac d {d\epsilon} \big|_{\epsilon=0}
G_{\delta,\oeps_k(\epsilon)}(\tw),0,\ldots,1,\ldots,0)\in
\textrm{dom}(D_{G_{\delta}(\tw)}^{\R^{m-1}})
$$
is $\|\cdot
\|_{1,\delta}$-close to $\ker\,D_{u_\delta}^{\R^{m-1}}$. On the other
hand, equation~\eqref{eq:minussigmak} shows that 
$G(e_k)$ is $\|\cdot\|_{1,\delta}$-close to
$-\p_s u_\delta -2\delta \frac d {d\epsilon} \big|_{\epsilon=0}
G_{\delta,\oeps_k(\epsilon)}(\tw)$. After projecting to
$\ker\,D_{u_\delta}^{\R^{m-1}}$ we get a vector
whose $k$-th component in $\R^{m-1}$ is negative, whose other
components in $\R^{m-1}$ are small, and whose component 
on $\ker\,D_{u_\delta}$ is $\|\cdot\|_{1,\delta}$-close to $-\p_s u_\delta$. 
\end{proof}

\medskip 

\begin{remark}{\rm 
We chose to define the signs $\epsilon(\u)$ by comparing
the orientation induced on $T_\u \widehat \cM^A(p,q;H,\{f_\gamma\},J)$
by the fiber sum rule from the coherent orientations on $T\widehat
\cM^{A_i}(S_{\gamma_i},S_{\gamma_{i-1}};H,J)$, $i=1,\ldots,m$ with the
orientation of the basis $(\p_s u_m,\ldots,\p_s u_1)$. Another
possible recipe would have been the following: induce orientations on
$T\cM^{A_i}(S_{\gamma_i},S_{\gamma_{i-1}};H,J)$ out of the coherent
orientations by quotienting out $\langle \p_s \rangle$, then apply
the fiber sum rule in order to get a sign on the zero-dimensional
spaces $T_{[\u]} \cM^A(p,q;H,\{f_\gamma\},J)$. The sign obtained in
this way would have differed from the previously defined $\epsilon(\u)$
by a factor $\pm 1$ which can be explicitly computed and which depends on
the combinatorics of the levels of $\u$. The curious reader can
test this procedure in the case $m=1$: it gives a sign equal to
$\epsilon(\u)$ if $p,q$ are both minima, respectively equal to
$-\epsilon(\u)$ if $p,q$ are both maxima. The following two properties
of the fibered sum constitute a useful tool for making the
verification (here $W_1$ and $W_2$ are oriented vector spaces). 
\begin{itemize}
 \item the natural isomorphism $W_1 \,
{_{\genfrac{}{}{0pt}{3}{{\ }}{\mbox{\scriptsize{0}}}}}
 \!\!\bigoplus_{f_2}
 W_2 \stackrel  
 \sim \to W_1\oplus \ker f_2$ changes the orientation by $(-1)^{\dim\,
 W_1 \cdot (\dim \, W_2 +1)}$; 
 \item the natural isomorphism $W_1 \, {_{f_1}}\!\!\bigoplus_{\,0}
 W_2 \stackrel 
 \sim \to \ker f_1 \oplus W_2$ preserves the orientation.
\end{itemize} 
}
\end{remark}


\appendix
\section{Appendix: Asymptotic estimates} \label{sec:appendix}
  
For all $\gamma \in \cP(H)$, we choose coordinates $(\vartheta, z) \in
S^1 \times \R^{2n-1}$
parametrizing a tubular neighbourhood of $\gamma$, such that
$\vartheta \circ \gamma(\theta)
= \theta$ and $z \circ \gamma(\theta) = 0$. Given a smooth function
$f_\gamma : S_\gamma \to \R$, we denote by $\varphi^{f_\gamma}_s$ the
gradient flow
of $f_\gamma$ with respect to the natural metric on $S^1$.

In a neighbourhood of $\gamma\in \cP(H)$ 
the Floer equation $\p_s u+J\p_\theta u-JX_H=0$ becomes $\p_s
Z+J\p_\theta Z+J\frac \p {\p\vartheta}-JX_H=0$, where 
$Z(s,\theta):=(\vartheta\circ u(s,\theta)-\theta, z\circ 
u(s,\theta))$. Since 
$X_H=\frac \p {\p\vartheta}$ on $\{z=0\}$ this can be rewritten as   
$$
\p_s Z+ J\p_\theta Z + Sz=0
$$
for some matrix-valued function $S=S(\vartheta,z)$. The matrix 
$S_\infty(\theta):=S(\theta,0)$\index{$S_\infty$, asymptotic matrix} 
is symmetric.
Let $A_\infty:H^k(S^1,\R^{2n})\to H^{k-1}(S^1,\R^{2n})$ be the operator 
defined by 
$$
A_\infty Z:=J\frac d {d\theta} Z+ S_\infty(\theta)z.
\index{$A_\infty$, asymptotic operator}
$$
The kernel of $A_\infty$ has dimension one and is spanned  
by the constant vector $e_1:=(1,0,\ldots,0)$. We denote by $Q_\infty$
\index{$Q_\infty$, asymptotic operator} 
the orthogonal projection onto $(\ker \, A_\infty)^\perp$ and we
set
\index{$P_\infty$, asymptotic operator}
$P_\infty := \one -Q_\infty$. Then 
$A_\infty$ is invertible when restricted to $\im\,Q_\infty$ and 
$Q_\infty A_\infty=A_\infty$.

\begin{proposition}  \label{prop:asymptotic}
Let $H \in \cH'$ be fixed. There exists $r > 0$ such that for all $J 
\in \cJ^\ell$ and for all
$u \in \cM^A(S_{\og},S_{\ug};H,J)$, $\og, \ug \in \cP(H)$ we have
\begin{eqnarray*}
\vartheta \circ u(s,\theta) - \theta - \otheta_0 & \in & 
W^{1,p}(]-\infty, -s_0]
\times S^1, \R ; e^{r|s|} ds \, d\theta) , \\
z \circ u(s,\theta) &\in& W^{1,p}(]-\infty, -s_0] \times S^1,
\R^{2n-1} ; e^{r|s|} ds \, d\theta) , \\
\vartheta \circ u(s,\theta) - \theta - \utheta_0 & \in & W^{1,p}([s_0,\infty[
\times S^1, \R ; e^{r|s|} ds \, d\theta) , \\
z \circ u(s,\theta) &\in& W^{1,p}([s_0,\infty[ \times S^1, \R^{2n-1} ;
e^{r|s|} ds \, d\theta) ,
\end{eqnarray*}
for some $\otheta_0, \utheta_0 \in S^1$ and some $s_0 > 0$ sufficiently large.
\end{proposition}

\begin{proof} We make the proof only at $+\infty$ since the case of 
$-\infty$ is 
entirely similar. For $s$ large enough 
we set $S(s,\theta):=S(\vartheta \circ u(s,\theta), z\circ 
u(s,\theta))$, so that 
$S_\infty(\theta)=\lim _{s\to \infty} S(s,\theta)$  and 
$\lim_{s\to \infty} |\p_sS(s,\theta)| =0$.  

Let $A(s):H^k(S^1,\R^{2n})\to H^{k-1}(S^1,\R^{2n})$ be the operator 
defined by 
$$
A(s)Z:=J\frac d {d\theta} Z+ S(s,\theta)z,
$$
so that 
$A_\infty=\lim _{s\to \infty} A(s)$. 
We have $A(s)=A(s)Q_\infty$, $\p_sQ_\infty=Q_\infty\p_s$. Since
$A_\infty$ is invertible when restricted to $\im\,Q_\infty$ and 
$Q_\infty A_\infty=A_\infty$, the operators $A(s)$ and $Q_\infty A(s)$  
are also invertible when restricted to $\im\,Q_\infty$ for $s$ large 
enough and there exists $c>0$ such that 
$$
\|A(s)Q_\infty Z\|^2 \ge 
\|Q_\infty A(s)Q_\infty Z\|^2 \ge  c\|Q_\infty Z\|^2 
$$
for all $Z\in H^k(S^1,\R^{2n})$. For $s$ large enough we define 
$$
f(s):=\frac 1 2 \|Q_\infty Z(s)\|^2 .
$$
We have 
\begin{eqnarray*}
f''(s) 
&  = & \|\p_sQ_\infty Z\|^2 + \langle Q_\infty Z,\p_s^2Q_\infty Z 
\rangle \\
& = & \|\p_sQ_\infty Z\|^2 - \langle Q_\infty 
Z,\p_sQ_\infty A(s) Q_\infty Z \rangle \\
& = & \|Q_\infty A(s) 
Q_\infty Z\|^2 - \langle Q_\infty Z,Q_\infty (\p_sA(s)) Q_\infty Z 
-
Q_\infty A(s)^2 Q_\infty  Z \rangle \\
& \ge & \!(c- 
\eps)\|Q_\infty Z\|^2  \!+\! \langle (A(s)^*-A(s))Q_\infty Z, 
A(s)Q_\infty Z\rangle 
\!+\! 
\|A(s)Q_\infty Z\|^2 \\
& \ge & 
(2c-2\eps) \|Q_\infty Z\|^2 \ge 4\rho^2 f(s).
\end{eqnarray*}
Here $A(s)^*$ is the adjoint of $A(s)$ and we used the fact that 
$\|\p_sA(s)\|\to 0$, $A(s)^*-A(s)\to 0$ for $s\to \infty$ and 
$\|A(s)\|$ is uniformly bounded. 

Let now $s_0$ be large enough and 
define $g(s):=f(s_0)e^{-2\rho(s-s_0)}$. Then 
$g''=4\rho^2g$, $(f-g)''\ge 4\rho^2(f-g)$, $(f-g)(s_0)=0$ and
$\lim_{s\to \infty} f(s)-g(s) = 0$. Then $f-g\le 0$ on 
$[s_0,\infty[$ because it cannot have a strictly positive
maximum. Therefore  
$$
\|Q_\infty Z(s)\| \le 
\|Q_\infty Z(s_0)\| e^{-\rho(s-s_0)}.
$$

It is important to note that this estimate holds for any Sobolev norm
$H^k$. By the Sobolev embedding theorem this implies the following
pointwise estimate
$$
|Q_\infty Z(s,\theta)|\le C e^{-\rho s}, \quad 
|\p_\theta Q_\infty Z(s,\theta)|=|\p_\theta Z(s,\theta)| \le C 
e^{-\rho s}, \quad s\ge s_0.
$$
Because $\p_s Z + A(s) Z = \p_s Z + A(s) Q_\infty Z =0$ we obtain  
$$
|\p_s Z(s,\theta)| \le C e^{-\rho s}, \quad s\ge s_0
$$
and, by integration on $[s,\infty[$ and taking into account that 
$Z(s,\theta)$ converges to $(\utheta_0,0,\ldots,0)$ for $s\to \infty$,
we obtain the pointwise estimate
$$
|(\vartheta - \theta - \utheta_0,z)|\le C e^{-\rho s}. 
$$
This implies the conclusion for any $r<\rho$. 
\end{proof}

\begin{proposition}  \label{prop:asymptoticdelta}
Let $H \in \cH'$ and  $\{ f_\gamma : S_\gamma \to \R \}$ be a
collection of perfect Morse functions
indexed by $\gamma \in \cP_\lambda$. There exist $r > 0$ and $\delta_0
>  0$ such that for all
$J \in \cJ$, $\og, \ug \in \cP(H)$, $p \in \mathrm{Crit}(f_{\og}), q
\in \mathrm{Crit}(f_{\ug})$
and for all $(\delta, u) \in
\cM^A_{]0,\delta_0]}(\og_p,\ug_q;H,\{f_\gamma\},J)$, we have
\begin{eqnarray*}
\vartheta \circ u(s,\theta) - \theta - \varphi^{\delta
f_{\og}}_s(\otheta_0) &\in& W^{1,p}(]-\infty, -s_0]
\times S^1, \R ; e^{r|s|} ds \, d\theta) , \\
z \circ u(s,\theta) &\in& W^{1,p}(]-\infty, -s_0] \times S^1, 
\R^{2n-1} ; e^{r|s|} ds \, d\theta) , \\
\vartheta \circ u(s,\theta) - \theta - \varphi^{\delta 
f_{\ug}}_s(\utheta_0) &\in& W^{1,p}([s_0,\infty[
\times S^1, \R ; e^{r|s|} ds \, d\theta) , \\
z \circ u(s,\theta) &\in& W^{1,p}([s_0,\infty[ \times S^1, \R^{2n-1} 
; e^{r|s|} ds \, d\theta) ,
\end{eqnarray*}
for some $\otheta_0, \utheta_0 \in S^1$ and some $s_0 > 0$
sufficiently large. 
\end{proposition} 

\begin{proof}
  The proof is similar to the one of 
Proposition~\ref{prop:asymptotic}. 
With the same notations as before 
the Floer equation satisfied by $u$ 
can be written in local 
coordinates $Z=(\vartheta-\theta,z)$ as 
\begin{equation} 
\label{eq:CRdelta}
\p_s Z+ J\p_\theta Z +Sz -\delta \nabla 
f_\ug(Z_1)=0,
\end{equation}
where $Z_1:=\vartheta-\theta$. We again 
show that 
$f(s)=\frac 1 2 \|Q_\infty Z\|^2$ satisfies an inequality 
of the 
form $f''(s)\ge 4\rho^2f(s)$. There are
two additional terms 
to estimate in the expression of $f''(s)$, namely 
\begin{equation} 
\label{eq:app1}
\langle Q_\infty Z, \delta Q_\infty A(s) \nabla 
f_\ug(Z_1) \rangle
\end{equation}
and
\begin{equation} 
\label{eq:app2}
\langle Q_\infty Z, \delta Q_\infty \p_s(\nabla 
f_\ug(Z_1)) \rangle.
\end{equation}
Let $P_\infty:=\one-Q_\infty$ be 
the orthogonal projection on $\ker\,A_\infty$. 
The main observation 
is that 
$Q_\infty \nabla f_\ug (P_\infty Z_1)=0$. As a consequence 
there exists a 
matrix-valued function $L=L(s,\theta)$ such that 
$$
Q_\infty \nabla f_\ug (Z_1) = L Q_\infty(Z_1). 
$$
The 
term~\eqref{eq:app1} is then estimated by 
\begin{eqnarray*}
\langle 
Q_\infty Z, \delta Q_\infty A(s) \nabla f_\ug(Z_1) \rangle
& = & 
\langle Q_\infty Z, \delta Q_\infty A(s) Q_\infty \nabla f_\ug(Z_1) 
\rangle \\
& \le & C\delta \|Q_\infty Z\|^2
\end{eqnarray*}
for $s\ge s_0$,  where $s_0$ depends on $u$, but $C$ depends only  
on $\ug$ and $f_\ug$. Similarly, the term~\eqref{eq:app2} is estimated
by  
\begin{eqnarray*}
\langle Q_\infty Z, \delta Q_\infty \p_s(\nabla 
f_\ug(Z_1))\rangle
& = & \langle Q_\infty Z, \delta \p_s Q_\infty 
\nabla f_\ug(Z_1)\rangle \\
& \le & C\delta 
\|Q_\infty Z\| \|\p_s 
Q_\infty (Z_1)\| \\
& \le & \frac {C\delta} 2 \Big( \|Q_\infty Z\| ^2 
+ 
\|\p_s Q_\infty Z\|^2\Big). 
\end{eqnarray*}
The norm of $\p_s 
Q_\infty Z = Q_\infty \p_s Z$ satisfies
$$
\|\p_sQ_\infty Z\| = 
\|Q_\infty A(s) Z - \delta Q_\infty \nabla f_\ug (Z_1)\| \le 
C\|Q_\infty Z\|. 
$$
As a consequence, there exists $\delta_0>0$ and 
$\rho>0$ such that 
$f''(s)\ge 4\rho^2f(s)$ for $s\ge s_0$ and 
$0<\delta\le \delta_0$. As before, we infer 
the pointwise bounds 
\begin{equation} \label{eq:pointwiseQdelta} 
|Q_\infty 
Z(s,\theta)|\le C e^{-\rho s}, \ 
|\p_\theta Q_\infty 
Z(s,\theta)|=|\p_\theta Z(s,\theta)| \le C e^{-\rho s}, \ s\ge 
s_0.
\end{equation}

It remains to estimate $P_\infty Z$. For that we 
write $\nabla f_\ug (Z_1)= 
\nabla f_\ug(P_\infty (Z_1)) + K Q_\infty 
(Z_1)$ for some matrix-valued 
function $K=K(s,\theta)$. Again, for 
$s\ge s_0$, the norm $\|K\|$ is 
uniformly bounded by a constant 
depending only on $\ug$ and $f_\ug$. 
By applying $P_\infty$ to the 
equation~\eqref{eq:CRdelta} and using the fact that 
$P_\infty\nabla 
f_\ug(P_\infty (Z_1))=\nabla f_\ug(P_\infty (Z_1))$
and 
$P_\infty(Z_1)=P_\infty(Z)$ we obtain 
\begin{equation} 
\label{eq:pinfty}
|\p_s(P_\infty Z)-\delta\nabla f_\ug(P_\infty Z)| 
\le Ce^{-\rho s}.
\end{equation}

We claim that this implies 
\begin{equation} \label{eq:pointwisePdelta}
|P_\infty Z (s) - 
\varphi_s^{\delta f_\ug}(\utheta_0)|\le Ce^{-\rho s}, \ s\ge 
s_0
\end{equation}
for a suitable $\utheta_0$. We choose a Morse 
coordinate $x$ on $S_\ug$ 
around the critical point $q$ of $f_\ug$ 
in which the gradient 
$\nabla f_\ug(x)=\pm Mx$, $M>0$. Then 
equation~\eqref{eq:pinfty} becomes 
$$
\p_s(P_\infty Z)(s) \mp \delta 
M P_\infty Z (s) = G(s)
$$
with $|G(s)|\le Ce^{-\rho s}$. Then 
$P_\infty Z(s)=c(s)e^{\pm\delta Ms}$ with 
$e ^{\pm\delta Ms}\p_s 
c(s)=G(s)$. As a consequence, for $\delta < \rho/M$ 
the function $c$ 
admits a limit $c_\infty$ as $s\to\infty$ and 
$c(s)=c_\infty - 
\int_s^\infty G(\sigma) e^{\mp\delta M \sigma}\, d\sigma$. Let 
$\utheta_0$ be such that $\varphi_s^{\delta 
f_\ug}(\utheta_0)=c_\infty e^{\pm\delta M s}$ 
(note that $c_\infty=0$ if $q$ is a maximum). Then
\begin{eqnarray*}
|P_\infty Z 
(s) - \varphi_s^{\delta f_\ug}(\utheta_0)| & = & 
|e^{\pm\delta M 
s}\int_s^\infty G(\sigma) e^{\mp\delta M \sigma}\, d\sigma| \\
&  \le 
& Ce^{-\rho s}.
\end{eqnarray*}
The estimates~\eqref{eq:pointwiseQdelta}
and~\eqref{eq:pointwisePdelta} imply the conclusion.
\end{proof}

\begin{proposition} \label{prop:intervaldelta}
Let $\delta\in]0,\delta_0]$ and  
let $u_\delta\in \widehat \cM^A(\og_p,\ug_q;H_\delta,J)$. Let  
$I_\delta=[s_0(\delta),s_1(\delta)]\subset \R$ be an interval 
such that $u_\delta(I_\delta\times S^1)$ is contained in the domain  
of a coordinate chart $Z=(\vartheta,z)$ around $S_\gamma$ for some
$\gamma\in \cP(H)$.

There exist $\rho>0$, $\theta_0\in S^1$, 
$C>0$ and $M>0$ such that 
$z \circ u(s,\theta)$ and its (first
order) derivatives are bounded by 
\begin{equation} \label{eq:boundz}
C\max(\|Q_\infty Z(s_0)\|, \|Q_\infty Z(s_1)\|)
\frac {\cosh(\rho(s-\frac
{s_0+s_1}2))} {\cosh(\rho(s_1-s_0)/2)}
\end{equation}
for $s\in I_\delta$, $\theta\in S^1$. 
If $P_\infty Z(s)$, $s\in I_\delta$
stays away from all but one of the critical points of $f_\gamma$, then
$\vartheta \circ u(s,\theta) - \theta - \varphi^{\delta
f_{\gamma}}_s(\theta_0)$ and its (first
order) derivatives are bounded by 
$$
C\max(\|Q_\infty Z(s_0)\|, \|Q_\infty Z(s_1)\|)
e^{\delta M(s_1-s_0)}\frac {\cosh(\rho(s-\frac
{s_0+s_1}2))} {\cosh(\rho(s_1-s_0)/2)}.
$$
Moreover, if $P_\infty Z(s)$, $s\in I_\delta$
stays away from all critical points of $f_\gamma$, the above bound is
improved to~\eqref{eq:boundz}.
\end{proposition}

\begin{proof}
With the notations of Proposition~\ref{prop:asymptoticdelta}, 
the Floer equation satisfied by $u$ 
can be written in local 
coordinates $Z=(\vartheta-\theta,z)$ as 
\begin{equation} 
\label{eq:CRdelta_bis}
\p_s Z+ J\p_\theta Z +Sz -\delta \nabla 
f_\gamma(Z_1)=0,
\end{equation}
where $Z_1:=\vartheta-\theta$. Let $A_\infty=J\frac d {d\theta} +
S_\infty(\theta)$ the asymptotic operator at $\gamma$, let 
$Q_\infty$ be the orthogonal
projection onto $(\ker \, A_\infty)^\perp$ and $P_\infty:=\one-
Q_\infty$. Then, as in Proposition~\ref{prop:asymptoticdelta}, the 
quantity 
$f(s)=\frac 1 2 \|Q_\infty Z\|^2$ satisfies an inequality 
of the form $f''(s)\ge 4\rho^2f(s)$. Define
$$
g(s):=\max(f(s_0),f(s_1))\frac {\cosh (2\rho(s-\frac {s_0+s_1}2))} 
{\cosh(\rho(s_1-s_0))}.
$$ 
Then $(f-g)''\ge 4\rho^2(f-g)$ and $f-g$
cannot have a strictly positive maximum. Since $f-g\le 0$ at $s_0$ and
$s_1$, we infer that $f-g\le 0$ on
$I_\delta$. As in Proposition~\ref{prop:asymptotic}, we infer the
pointwise bounds for $s\ge s_0$
\begin{eqnarray} \label{eq:pointwiseintervalQdelta} 
|Q_\infty 
Z(s,\theta)| &\le& Cg_1(s), \\
|\p_\theta Q_\infty 
Z(s,\theta)|  =  |\p_\theta Z(s,\theta)| 
&\le& Cg_1(s),\nonumber \\
|\p_s(P_\infty Z)(s)-\delta\nabla f_\gamma(P_\infty Z)(s)|
&\le& C_1g_1(s), \nonumber 
\end{eqnarray}
where 
$$
g_1(s):=\max(\|Q_\infty Z(s_0)\|, \|Q_\infty Z(s_1)\|)
\sqrt{\frac {\cosh (2\rho(s-\frac {s_0+s_1}2))} 
{\cosh(\rho(s_1-s_0))}}. 
$$
If $P_\infty Z(s)$ stays away from $\textrm{Crit}(f_\gamma)$, 
we can assume that $\nabla f_\gamma(P_\infty Z(s))=M$ in some suitable 
coordinate on $S^1$. Then the last equation becomes 
$$
\p_s(P_\infty Z)(s)-\delta M = G(s),
$$
where $|G(s)|\le C_1g_1(s)$. By direct integration we obtain 
\begin{eqnarray*}
 |(P_\infty Z)(s)-\delta M s - c_0| & = & \Big|\int _{\frac {s_0+s_1} 2}^s
G(\sigma) \, d\sigma \Big| \\
& \le & C_2\Big| \int _{\frac {s_0+s_1} 2}^s
 \sqrt {\cosh(2\rho(s-\frac {s_0+s_1}2))}\, d\sigma \Big| \\
& \le & C_2\frac {\sqrt 2} \rho \big|\sinh(\rho(s-\frac
{s_0+s_1}2))\big|
\\
& \le & C_2\frac {\sqrt 2} \rho \cosh(\rho(s-\frac
{s_0+s_1}2)).
\end{eqnarray*}
Here $C_2=C_1\max(\|Q_\infty Z(s_0)\|, \|Q_\infty Z(s_1)\|)/
\sqrt{\cosh(\rho(s_1-s_0))}$ and we have used the inequality 
$\sqrt{\cosh x}\le \sqrt 2 \cosh (x/2)$. Therefore, there exists a
uniquely determined $\theta_0\in S^1$ such that 
\begin{eqnarray*}
\lefteqn{|(P_\infty Z)(s)-\varphi_s^{\delta f_\gamma}(\theta_0)|} \\ 
& \le&
\frac {C_1\sqrt 2} \rho 
\max(\|Q_\infty Z(s_0)\|, \|Q_\infty Z(s_1)\|)
\frac {\cosh(\rho(s-\frac
{s_0+s_1}2))} {\sqrt{\cosh(\rho(s_1-s_0))}} \\
&\le  &
\frac {C_1\sqrt 2} \rho 
\max(\|Q_\infty Z(s_0)\|, \|Q_\infty Z(s_1)\|)
\frac {\cosh(\rho(s-\frac
{s_0+s_1}2))} {\cosh(\rho (s_1-s_0)/2)}. 
\end{eqnarray*}
The last inequality follows from $\cosh(x/2)\le \sqrt {\cosh x}$. A
similar manipulation on~\eqref{eq:pointwiseintervalQdelta} 
gives 
$$
|Q_\infty Z(s,\theta)|\le C\sqrt 2 
\max(\|Q_\infty Z(s_0)\|, \|Q_\infty Z(s_1)\|)
\frac {\cosh(\rho(s-\frac
{s_0+s_1}2))} {\cosh(\rho(s_1-s_0)/2)}. 
$$ 
The last two inequalities imply the conclusion of the Proposition in
the case when $P_\infty Z(s)$, $s\in I_\delta$ stays away from
$\textrm{Crit}(f_\gamma)$. 

If $P_\infty Z(s)$ is allowed to approach one of the critical points
of $f_\gamma$, the estimate on
$|Q_\infty Z(s)|$ stays the same, but the estimate involving $P_\infty
Z (s)$ has to be modified as follows. In a suitable Morse coordinate
chart around the critical point we can assume that $\nabla
f_\gamma(x)=\pm Mx$, $M>0$ and we have to study the equation 
$$
\p_s(P_\infty Z)(s) \mp \delta M P_\infty Z(s) = G(s),
$$ 
with $|G(s)|\le C_1g_1(s)$. As in
Proposition~\ref{prop:asymptoticdelta} we have $P_\infty
Z(s)=c(s)e^{\pm\delta Ms}$ with $e^{\pm\delta Ms}\p_sc(s)=G(s)$. Then 
$c(s)=c_0+\int_{\frac {s_0+s_1} 2}^s e^{\mp\delta M
\sigma}G(\sigma)\,d\sigma$ and there exists a $\theta_0\in S^1$ such
that $\varphi_s^{\delta f_\gamma}(\theta_0)=c_0e^{\pm\delta Ms}$. We
obtain 
\begin{eqnarray*}
 |(P_\infty Z)(s)-\varphi_s^{\delta f_\gamma}(\theta_0)| & \le & 
\Big|\int _{\frac {s_0+s_1} 2}^s
e^{\pm\delta M(s-\sigma)}G(\sigma) \, d\sigma \Big| \\
&\le & e^{\delta M (s_1-s_0)} \Big|\int _{\frac {s_0+s_1} 2}^s
G(\sigma) \, d\sigma \Big|. 
\end{eqnarray*}
The last integral is bounded by 
$$
\frac {C_1\sqrt 2} \rho 
\max(\|Q_\infty Z(s_0)\|, \|Q_\infty Z(s_1)\|)
\frac {\cosh(\rho(s-\frac
{s_0+s_1}2))} {\cosh(\rho (s_1-s_0)/2)}
$$
as in the previous case and the conclusion follows.
\end{proof}


\nopagebreak[4] 

{\scriptsize{\printindex}}

\end{document}